\documentclass[10pt]{article}
\usepackage{graphicx}
\usepackage{amsfonts,amsmath,amssymb,amsthm,amscd}
\usepackage{mathtools}
\usepackage{indentfirst}
\usepackage{placeins}
\usepackage{listings}
\usepackage{caption}
\usepackage{subcaption}
\usepackage{xcolor}
\usepackage{comment}
\usepackage{url}
\usepackage{bm}
\usepackage[toc,page]{appendix} 
\usepackage{fancyvrb} 
\usepackage{multirow}

\allowdisplaybreaks

\usepackage{fullpage} 
\usepackage{enumitem} 
\usepackage{hyperref}
\hypersetup{
  colorlinks   = true, 
  urlcolor     = black, 
  linkcolor    = blue, 
  citecolor   = blue 
}
\usepackage{dsfont}

\newtheorem{theorem}{Theorem}[section]
\newtheorem{thm}[theorem]{Theorem}
\newtheorem{prop}[theorem]{Proposition}
\newtheorem{lemma}[theorem]{Lemma}
\newtheorem{corollary}[theorem]{Corollary}
\newtheorem{example}[theorem]{Example}
\newtheorem{assump}[theorem]{Assumption}
\newtheorem{remark}[theorem]{Remark}

\numberwithin{equation}{section} 

\def\R{\mathbb{R}}
\def\eps{\varepsilon}
\def\lb{{(\ell)}}
\def\l{\ell}
\def\cmu{\mu}
\def\E{\mathbf{E}}
\def\Tnum{\Theta^{A,\kappa}_{\Delta t, \eps_{\mathrm{tol}}}}
\DeclareMathOperator{\rank}{rank}

\begin{document}
\title{Non-reversible sampling schemes on submanifolds}

\author{
  Upanshu Sharma\thanks{Fachbereich Mathematik und Informatik, Freie Universit\"at Berlin, Arnimallee 9, 14195 Berlin, Germany.\ Email: 
  upanshu.sharma@fu-berlin.de}
  \and Wei Zhang\thanks{Zuse Institute Berlin, Takustrasse 7, 14195 Berlin, Germany.\  Email: wei.zhang@fu-berlin.de} 
}

\maketitle

\begin{abstract}
Calculating averages with respect to probability measures on submanifolds is
  often necessary in various application areas such as molecular dynamics,
  computational statistical mechanics and Bayesian statistics. In recent
  years, various numerical schemes have been proposed in the literature to
  study this problem based on appropriate reversible constrained stochastic
  dynamics. In this paper we present and analyse a non-reversible
  generalisation of the projection-based scheme developed by one of the
  authors [ESAIM: M2AN, {\bf 54} (2020), pp.\ 391--430]. This scheme consists
  of two steps -- starting from a state on the submanifold, we first update the state
  using a non-reversible stochastic differential equation which takes the
  state away from the submanifold, and in the second step we project the state
  back onto the manifold using the long-time limit of an ordinary differential
  equation. We prove the consistency of this numerical scheme and provide
  quantitative error estimates for estimators based on finite-time running
  averages. Furthermore, we present theoretical analysis which shows that this
  scheme outperforms its reversible counterpart in terms of asymptotic variance. 
  We demonstrate our findings on an illustrative test example.
\end{abstract}

\begin{keywords}
submanifold, constrained sampling, non-reversible process, reaction coordinate, conditional probability measure
\end{keywords}


\section{Introduction}\label{sec:Intro}
Sampling probability measures on submanifolds is a relevant numerical task in
various research fields such as molecular dynamics (MD), computational statistical
mechanics and Bayesian
statistics~\cite{LelievreRoussetStoltz10,Tony-constrained-langevin2012,goodman-submanifold,manifold-mcmc-Bayesian}.
In MD, usual systems under consideration are extremely high dimensional and
quantities of interest evolve at time scales which are order of magnitudes
larger than those achievable by typical numerical
methods~\cite{Leimkuhler-Matthews-MD-book}. Therefore in practice it is common
to project the system onto a lower-dimensional set of variables which capture
the relevant behaviour of the system by means of a so-called \emph{reaction
coordinate} (also called collective variables or coarse-graining map in the literature)
\begin{equation}
  \xi=(\xi_1,\xi_2, \dots,\xi_k)^T: \R^d\rightarrow \R^k, \quad \mbox{where}~ \xi_\alpha: \R^d\rightarrow\R, \quad 1 \le \alpha \le k< d\,.
  \label{def-xi}
\end{equation}
Working with these reaction coordinates often requires us to compute averages on its level-sets
\begin{equation}\label{levelset-sigma-intro}
  \Sigma:=\xi^{-1}(\bm{0})=\Bigl\{ x \in \mathbb{R}^d ~\Big|~ \xi(x) = \bm{0} \in \mathbb{R}^k\Bigr\},
\end{equation}
with respect to certain probability measures. One particularly important probability measure which is central to this paper is the so-called 
\emph{conditional probability measure} $\cmu$ on $\Sigma$~\cite{blue-moon,legoll2010effective,ZhangHartmannSchutte16}, 
\begin{align}\label{mu1-intro}
d\cmu := \frac{1}{Z} \mathrm{e}^{-\beta U} \big[\mbox{det}(\nabla\xi^T
\nabla\xi)\big]^{-\frac{1}{2}} d\nu_\Sigma.
\end{align}
In \eqref{mu1-intro}, $\beta > 0$ is a positive constant related to the inverse of system's temperature, $U: \mathbb{R}^d \rightarrow
\mathbb{R}$ is a $C^2$-smooth potential, $\nabla\xi\in\R^{d\times k}$ is
the Jacobian of $\xi$, $\nu_\Sigma$ is the (normalised) surface measure on $\Sigma$ induced from the Lebesgue measure on $\R^d$ and $Z$ is the normalisation constant
\begin{equation*}
Z:=\int_{\Sigma}\mathrm{e}^{-\beta U}\big[\mbox{det}(\nabla\xi^T
\nabla\xi)\big]^{-\frac{1}{2}} d\nu_\Sigma\, < \infty,
\end{equation*}
which ensures that $\cmu$ is a probability measure. Calculating averages with respect to $\cmu$ on $\Sigma$ is an ubiquitous challenge in computation of
thermodynamic quantities pertaining to free
energy~\cite{blue-moon,LelievreRoussetStoltz10,Tony-constrained-langevin2012,non-equilibrium-2018}.
A relatively new albeit important application, which in part motivates this
paper, is the so-called \emph{effective dynamics} which arises as the approximation of  $\xi$-projection of diffusion processes, and whose coefficients involve the aforementioned averages~\cite{legoll2010effective,ZhangHartmannSchutte16,legoll2017pathwise,DLPSS18,LelievreZhang18,LegollLelievreSharma18,HartmannNeureitherSharma20}. 

In recent years, several numerical algorithms have been developed to sample probability measures on
submanifolds~\cite{projection_diffusion,hmc-submanifold-tony,pmlr-v22-brubaker12,goodman-submanifold,Zhang20,multiple-projection-mcmc-submanifolds}.
A common feature in all these algorithms is that they are essentially
reversible, i.e.\ either based on reversible SDEs on submanifolds or using
reversible Markov chains. Meanwhile, it is well known that non-reversible dynamics on $\mathbb{R}^d$ offer considerable advantages over their reversible counterparts when sampling
probability measures, for instance improved convergence rates and reduced
asymptotic variance~\cite{HwangHwangSheu93,lelievre2013optimal,DuncanLelievrePavliotis16,improve-reversible-sampler,DuncanPavliotisZygalakis17,LuSpiliopoulos18}.
Inspired by these developments, the central aim of this paper is:
\begin{center}
\emph{Develop and analyse a non-reversible algorithm for sampling the
  conditional probability measure $\mu$ on the level-set $\Sigma$
  \eqref{levelset-sigma-intro} of possibly nonlinear reaction coordinate $\xi$.}
\end{center}
Specifically, in this paper we will focus on a non-reversible generalisation conjectured by one of the authors in~\cite[Remark 3.8]{Zhang20}. 
At a current state on $\Sigma$, the scheme that we propose in this paper consists of two steps:
\begin{enumerate}[label=(\arabic*)]
\item First, the state is updated using a discrete scheme that is linked to a non-reversible diffusion process, which pushes the state out but in close vicinity of $\Sigma$.
\item Second, the state is projected back to $\Sigma$ using the long-time limit of an appropriate ordinary differential equation (ODE). 
\end{enumerate}

We need two quantities to present the numerical scheme. Let $A\in \mathbb{R}^{d\times d}$ be a constant skew-symmetric matrix, i.e.\  $A^T=-A$. Furthermore, let $\sigma:\R^d\rightarrow\R^{d\times d_1}$  with integer $d_1\ge d$, for which $a:=\sigma\sigma^T:\R^d\rightarrow\R^{d\times d}$ is uniformly positive definite. 
The \emph{non-reversible numerical scheme}, which is the central focus of this paper, is 
\begin{align}\label{scheme-non-reversible}
  \begin{split}
    x^{(\ell + \frac{1}{2})}_i &= x^\lb_i + \sum_{j=1}^{d} \Big( (A_{ij}-a_{ij})\frac{\partial U}{\partial x_j}  
    + \frac{1}{\beta}\frac{\partial
    a_{ij}}{\partial x_j}\Big)(x^\lb)\, h
    + \sqrt{2 \beta^{-1}h}\, \sum_{j=1}^{d_1}\sigma_{ij}(x^\lb)\, \eta^\lb_j\,, \quad 1 \le i \le
  d \,,\\
  x^{(\ell+1)} &= \Theta^A\big(x^{(\ell+\frac{1}{2})}\big)\,,
\end{split}
\end{align}
for $\ell = 0, 1, \dots$,
where $x^{(0)}\in \Sigma$, $h$ is the step-size, 
$\bm{\eta}^{(\ell)}=(\eta_1^{(\ell)},\dots,\eta_{d_1}^{(\ell)})^{T}\in\mathbb{R}^{d_1}$ where $\eta_i^{(\ell)}$ are independent and identically distributed \emph{bounded} random variables, which for some constant $C_\eta>0$ satisfy
\begin{equation}
  \begin{aligned}
    & |\eta_i^{(\ell)}| \le C_\eta < +\infty \quad \text{almost surely}\,,\\
    & \E[\eta_i^{(\ell)}]= \E[(\eta_i^{(\ell)})^3] = 0\,,\quad 
  \E[(\eta_i^{(\ell)})^2]=1\,,\quad \forall\, 1 \le i \le d_1, \quad \forall\, \ell \ge 0\,.
  \end{aligned}
  \label{conditions-on-eta}
\end{equation}
See Remark~\ref{rmk-on-choice-of-noise}, Remark~\ref{rmk-well-posedness-of-limit-theta} and ~\cite[Remark 3.2]{rk-schemes-submanifolds} for the motivation behind this choice of bounded random variables. 

The map $\Theta^A:\R^d\rightarrow\R^d$ used in \eqref{scheme-non-reversible} is defined via the long-time limit
\begin{equation}
\Theta^{A}(x)=\lim\limits_{s\rightarrow +\infty} \varphi^{A}(x,s),
  \label{def-theta-map-intro}
\end{equation} 
  where, for any $x\in \mathbb{R}^d$, $\varphi^A:\R^d\times[0,+\infty)\rightarrow\R^d$ is the solution to the ODE 
\begin{align}  \label{phi-map-a-intro}
  \begin{split}
    \frac{d\varphi^{A}(x,s)}{ds} =& -((a-A)\nabla
    F\big)\big(\varphi^{A}(x,s)\big)\,,\quad s \ge 0 \,,\\
    \varphi^{A}(x,0) =\,& x, 
\end{split}
\end{align}
with the function $F:\R^d\rightarrow\R$ given by
\begin{equation}
  F(x):=\frac{1}{2}|\xi(x)|^2=\frac{1}{2}\sum_{\alpha=1}^{k}\xi_\alpha^2(x)\,.
\label{fun-cap-f-intro}
\end{equation}
While our analysis applies to the case of general matrices $a$, the choice
$a=\sigma=I_d$ (i.e.\ identity matrix of order $d$) is particularly interesting
due to its simplicity, in which case the scheme \eqref{scheme-non-reversible} becomes 
\begin{align}\label{scheme-non-reversible-id}
  \begin{split}
    x^{(\ell + \frac{1}{2})} &= x^\lb + (A-I_d) \nabla U (x^\lb)\, h
    + \sqrt{2 \beta^{-1}h}\, \bm{\eta}^\lb\,, \\
  x^{(\ell+1)} &= \Theta^A\big(x^{(\ell+\frac{1}{2})}\big)\,.
\end{split}
\end{align}

Note that the appearance of the skew-symmetric matrix $A$ in both steps of the
numerical scheme~\eqref{scheme-non-reversible} is not a coincidence.
To see this, let us motivate~\eqref{scheme-non-reversible} by considering the
so-called (non-reversible) \emph{soft-constrained dynamics}~\cite{projection_diffusion}
\begin{equation}\label{eq:soft-const}
  d X^{i,\eps}_s  = \sum_{j=1}^d \Bigl[ (A_{ij}-a_{ij}) \frac{\partial}{\partial
x_j}\Big(U+\frac{1}{\eps} F\Big)\,
  + \frac{1}{\beta} \frac{\partial a_{ij}}{\partial x_j}\Bigr](X^{\eps}_s) \,ds
 + \sqrt{2\beta^{-1}} \sum_{j=1}^{d_1}\sigma_{ij}(X^{\eps}_s)\, dW^j_s\,,
  \quad 1 \le i \le d\,,
\end{equation}
where $\eps>0$, $X^{\eps}_s=(X^{1,\eps}_s, \ldots, X^{d,\eps}_s)^T\in \mathbb{R}^d$, and $W_s=
(W^1_s, \ldots, W^{d_1}_s)^T\in \mathbb{R}^{d_1}$ is a $d_1$-dimensional Brownian motion.
Under fairly general conditions on the coefficients, \eqref{eq:soft-const} is ergodic with respect to the $\eps$-dependent probability measure
\begin{equation}\label{mu-eps}
\forall \,x \in\R^d: \   d\mu^\eps(x) = \frac{1}{Z^\eps} \exp\Big[-\beta\Big(U(x) + \frac{1}{\eps}F(x) \Big)\Big]\,dx\,,
\end{equation}
with the corresponding normalisation constant $Z^\eps$. A straightforward calculation shows that $\mu^\eps$ defined on $\R^d$ converges to $\cmu$ \eqref{mu1-intro}
defined on $\Sigma$ (see Lemma~\ref{lem:mueps-conv} and ensuing discussion).
While the dynamics~\eqref{eq:soft-const} has the nice property that for small $\eps$
it typically stays close to the manifold $\Sigma$, it has drawbacks when directly used in sampling tasks. Theoretically, analysing the discrete version of
\eqref{eq:soft-const} is difficult because it involves both the finite step-size $h$ and the small parameter $\eps$.
Numerically, difficulties arise when using \eqref{eq:soft-const} in practice,
since small $\eps$ is required to ensure reliable sampling, which in turn restricts the step-size of the numerical discretisation (see Figure~\ref{fig-soft-constraint}
for a concrete example where the estimation error using \eqref{eq:soft-const} depends rather sensitively on the choice of both $\eps$ and step-size).
The scheme~\eqref{scheme-non-reversible} can be viewed as a two-step numerical method which handles the non-stiff part, i.e.\ the $\eps$-independent terms of~\eqref{eq:soft-const}, via propagation without the constraint and the stiff part, i.e.\  the $\eps$-dependent terms of~\eqref{eq:soft-const}, via a projection onto $\Sigma$ under $\Theta^A$.  
\begin{figure}[h!]
\includegraphics[width=0.5\textwidth]{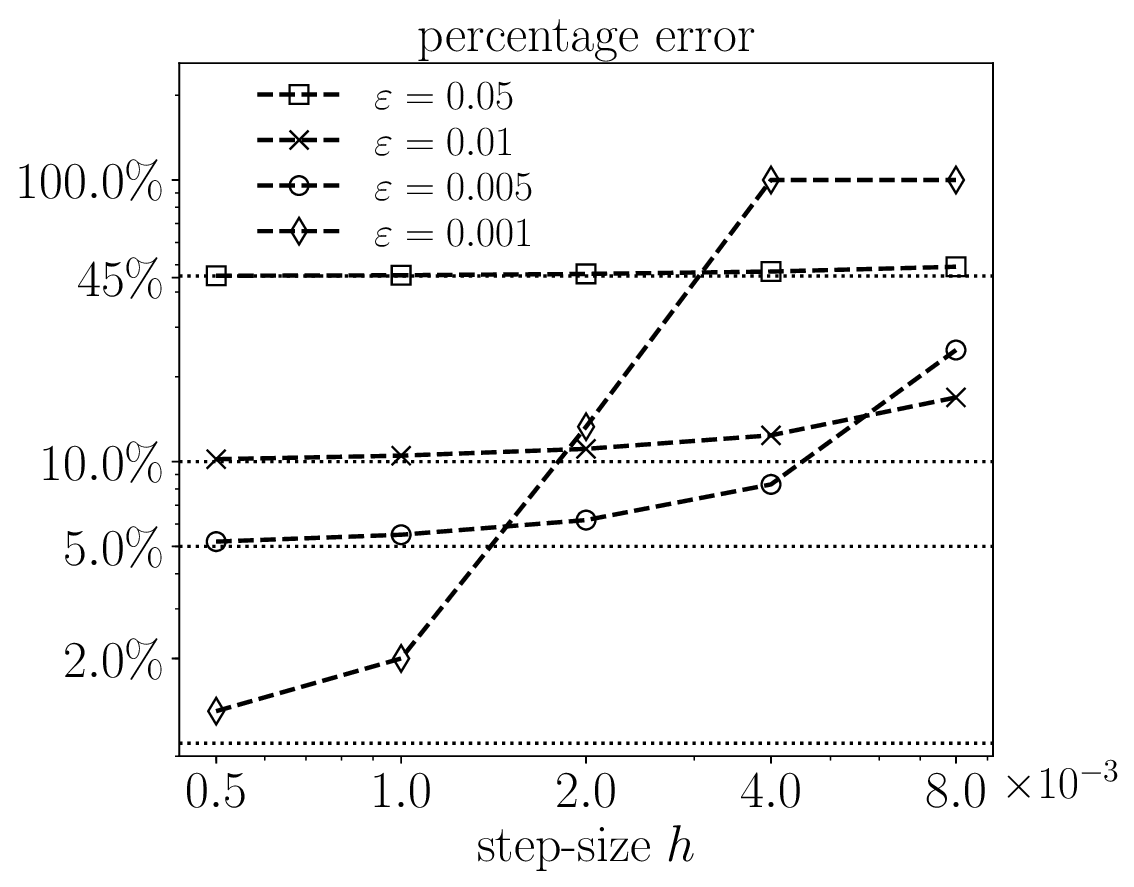}
\centering
  \caption{Percentage error for the estimation of the mean value $\mathbf{E}_\mu(f)$ with respect
  to the probability measure $\mu$ \eqref{mu1-intro} over the level-set
  $\Sigma$ \eqref{levelset-sigma-intro}, computed using numerical
  discretisation of the soft-constrained dynamics \eqref{eq:soft-const} with
  different $\eps>0$ and step-sizes $h$. Here $f(x)=3 + 2
  \cos(\pi \sqrt{x_1^2+x_2^2})$ and $\xi(x)=\frac{1}{2}(x_1^2+x_2^2-1)$, for $x=(x_1,x_2)\in \mathbb{R}^2$. 
  The horizontal dotted lines are the mean values of $f$ in $\mathbb{R}^2$ with respect to $ \mu^\eps$ \eqref{mu-eps} for different values of $\eps$. 
 For $\eps=0.001$, the percentage error $100 \%$ indicates that the
    numerical discretisation of \eqref{eq:soft-const} is unstable for the step-sizes $h=4.0 \times 10^{-3}$
    and $h=8.0 \times 10^{-3}$. To have percentage error less than $5\%$
    (respectively $2\%$), one has to choose $\eps$ to be $0.005$ (respectively $0.001$)
    in \eqref{eq:soft-const} and an even smaller step-size $h$ for numerical
    discretisation. In contrast, since $f \equiv 1$ on $\Sigma$, the numerical
    scheme proposed in this paper (as well as other projection-based schemes)
    will give the correct estimation $\mathbf{E}_\mu(f)=1$ up to a small
    numerical error due to the computation of $\Theta^A$.    
  \label{fig-soft-constraint}}
\end{figure}

We now outline the results of this work. The first main result is
Theorem~\ref{thm:Err}, which contains quantitative estimates comparing the
running average computed from the numerical
scheme~\eqref{scheme-non-reversible} and the average with respect to $\mu$,
i.e.\ for an observable $f:\Sigma\rightarrow\R$ we present estimates on the difference
\begin{equation*}
\frac1n \sum\limits_{\l=0}^{n-1}f(x^\lb) - \int_{\Sigma} f(x) d\mu(x),
\end{equation*}
in mean, $L^2$ and almost-sure sense. A key outcome of this result is that the
running average using the scheme \eqref{scheme-non-reversible} indeed converges to the average with respect to $\mu$ in the
limit of small step-size $h\rightarrow 0$ and long-times $T\rightarrow +\infty$. 
The second main result of this paper is Proposition~\ref{prop:impro-ass-var}, which states that 
in the long-time limit the non-reversible scheme~\eqref{scheme-non-reversible} has smaller
asymptotic variance (better sampling efficiency) as compared to its reversible
counterpart (i.e.\ with $A=0$) proposed in~\cite[Eq (1.11)]{Zhang20}.
The aforementioned results concern the scheme~\eqref{scheme-non-reversible} which uses the exact long-time limit
$\Theta^A$ \eqref{def-theta-map-intro}. The third main result in
Theorem~\ref{thm-numerical-theta} discusses the corresponding error estimates
when numerical approximations of $\Theta^A$ are used instead (see the
numerical scheme \eqref{scheme-non-reversible-numerical}).

We now discuss relevant literature. Sampling schemes on
submanifolds have been studied using constrained (i.e.\ by
introducing Lagrange multipliers) overdamped SDEs \cite{projection_diffusion} and constrained Langevin dynamics~\cite{Tony-constrained-langevin2012}.  Higher order schemes on submanifolds have been constructed in the recent work~\cite{rk-schemes-submanifolds}. 
Sampling schemes for the conditional probability measure $\mu$ \eqref{mu1-intro} where constraints are imposed via gradient ODE flows were analysed in~\cite{Zhang20}. 
In the recent work \cite{leimkuhler-constraint-regularization-nn}, the authors applied the framework of constrained Langevin dynamics to the training of deep
neural networks and demonstrated improved training results.
There are also various Markov chain Monte Carlo (MCMC) methods on
submanifolds in the literature~\cite{pmlr-v22-brubaker12,goodman-submanifold,hmc-submanifold-tony,multiple-projection-mcmc-submanifolds}.
In particular, the authors in \cite{goodman-submanifold} constructed a random
walk Monte Carlo method on submanifolds and pointed out the necessity of
``reversibility check'' for sampling on submanifolds. Following
\cite{goodman-submanifold}, MCMC methods on submanifolds with more general
proposals have been proposed in ~\cite{hmc-submanifold-tony,multiple-projection-mcmc-submanifolds}.
These methods have the advantage that they are unbiased and therefore allow the use of large sampling step-sizes.
Applications of MCMC methods on submanifolds in Bayesian inference can be found in \cite{manifold-mcmc-Bayesian}. 

In comparison to the literature, the current work is new in the following aspects.
First, existing algorithms in the aforementioned work except~\cite{Zhang20} aim at sampling
the Boltzmann distribution confined to submanifolds (i.e.\ without the
determinant factor in $\mu$ \eqref{mu1-intro}), and therefore require importance sampling or modification of potential when sampling $\mu$. In contrast, similar to~\cite{Zhang20}, the algorithm proposed in the current work directly samples
$\mu$ and is therefore expected to be more applicable in molecular dynamics
applications, e.g.\ free energy calculation and model reduction.
Second, the algorithm in the current work imposes constraints via ODE flows as
opposed to Lagrange multipliers. Note that even though the map $\Theta^A$ in
the scheme~\eqref{scheme-non-reversible} is defined as the long-time limit of \eqref{phi-map-a-intro}--\eqref{fun-cap-f-intro}, 
in practice one uses a modified ODE flow that converges to $\Theta^A$ in
finite time (see \eqref{phi-map-a-kappa} and Lemma~\ref{lemma-rescaled} for
details). Moreover, when $k$ is medium or large (e.g.\ $k\ge 10$), the total computational cost of the constraint step is comparable or even smaller
than the cost of performing the constraint step via Newton's method (see \cite[Remark 3.9 and Example 2]{Zhang20}). 
In contrast to the algorithms in \cite{multiple-projection-mcmc-submanifolds} which are built on numerical methods for computing multiple projections, the
numerical scheme \eqref{scheme-non-reversible} only uses the gradients of $\xi$ and is relatively simpler to implement. 
Consequently, we indeed expect that our algorithm scales well
for high-dimensional applications in molecular dynamics and machine
learning~\cite{leimkuhler-constraint-regularization-nn}, where $\xi$ is non-trivial and $k$ is usually large. 
Third, by applying analytical tools in~\cite{DuncanLelievrePavliotis16},
we extend previous results for non-reversible dynamics on $\mathbb{R}^d$ to numerical schemes on submanifolds.
To the best of our knowledge, sampling algorithms on submanifolds using non-reversible dynamics have not been considered before.

The remainder of this article is organised as follows. In
Section~\ref{sec-main-results} we introduce notations, assumptions, and state the main results of this paper. 
Section~\ref{sec-prepare} summarises preliminary properties of useful quantities.
Section~\ref{sec-ergodic-sde} studies the properties of a non-reversible diffusion process which plays a crucial role in the analysis of the numerical scheme. 
In Section~\ref{sec-map-theta} we analyse the ODE~\eqref{phi-map-a-intro} and the projection~\eqref{def-theta-map-intro}. 
Section~\ref{sec-compare} is devoted to the comparison between the non-reversible and reversible setting. 
In Section~\ref{sec-numerical} we study an illustrative example.
We conclude with discussions on various issues in Section~\ref{sec:discuss}.
Appendix~\ref{app:lemma-identity-proof} proves a
technical lemma which is used in the proofs of the main results. In
Appendix~\ref{app:thm-err} we present the proofs of Theorem~\ref{thm:Err}, Corollary~\ref{corollary-on-spectral-gap} and Theorem~\ref{thm-numerical-theta}.

\section{Notations, assumptions and results}
\label{sec-main-results}

In what follows we first present notations and the central assumptions
(Section~\ref{secsub-notations}). We then present crucial auxiliary results in
Section~\ref{subsec-summary-of-useful-results}, which are used to obtain the main results of this paper in Section~\ref{subsec-summary-of-main-results}.

\subsection{Notations and assumptions}
\label{secsub-notations}

We use $\mathbb N^+$ for positive integers. Given two
subsets $\Omega, \Omega'$ of Euclidean spaces and $r\in \mathbb N^+$, 
$C^r(\Omega, \Omega')$ is the space of all $C^r$-differentiable functions from
$\Omega$ to $\Omega'$. When $\Omega'=\mathbb R$, we use $C^r(\Omega):=C^r(\Omega,\mathbb R)$. 
The set $C_b(\Omega)$ is the space of bounded continuous functions from $\Omega$ to $\mathbb{R}$.
For all $C^1$-differentiable functions $f: \mathbb{R}^d\rightarrow \mathbb{R}^{d'}$ and $x \in \mathbb{R}^d$, 
$\nabla f(x)$ denotes the $d\times d'$ matrix whose entries are $(\nabla
f)_{ij}(x)=\frac{\partial f_j(x)}{\partial x_i}$, for $1 \le i \le d$ and $1 \le j \le d'$. 
The relation $A\succeq B$ is the Loewner ordering for square matrices $A,B$, i.e.\  $A-B$ is positive semi-definite. The same notation will be used for the Loewner ordering between operators.
The matrix $I_m\in \mathbb{R}^{m\times m}$ is the identity matrix of order
$m\in \mathbb N^+$. We make the following assumptions throughout this paper.

\begin{assump}  \label{assump-1}
The matrix-valued function $\sigma:\R^d\rightarrow\R^{d\times d_1}$ is $C^\infty$-smooth, where the integer $d_1\ge d$, such that
  $a:=\sigma\sigma^T:\R^d\rightarrow\R^{d\times d}$ is uniformly positive
  definite on $\mathbb{R}^d$, i.e.\   
  \begin{equation*}
    v^T a(x) v \ge c_0 |v|^2, \quad \forall~ x, v \in
    \mathbb{R}^d ~ \,,
  \end{equation*}
  for some constant $c_0>0$. The potential $U:\R^d\rightarrow\R$ is $C^2$-smooth. The constant matrix $A\in \mathbb{R}^{d\times d}$ is a constant skew-symmetric matrix, i.e.\  $A^T=-A$.
\end{assump}
\begin{assump}  \label{assump-2}
  The function $\xi: \mathbb{R}^d \rightarrow \mathbb{R}^k$ is $C^4$-smooth.
  The zero level set $\Sigma$ \eqref{levelset-sigma-intro} is both connected and compact. Furthermore, 
  $\rank(\nabla\xi) = k$ for all $x \in \Sigma$, where $\nabla\xi:\R^d\rightarrow\R^{d\times k}$.  
\end{assump}

\begin{remark} \label{rmk-on-non-constant-A}
Throughout this paper we assume that $A$ is a constant skew-symmetric matrix. While the numerical scheme and the corresponding results can be generalised to non-constant skew-symmetric matrices $A:\R^d\rightarrow\R^{d\times d}$ (see Remark~\ref{rem:genA} for details), we work with constant $A$ for the ease of presentation.    
\end{remark}

\begin{remark}
Due to Assumption~\ref{assump-2}, there exists $\delta > 0$ such that $\rank(\nabla\xi) = k$ 
  in a neighbhourhood $\Sigma^{(\delta)}$ of $\Sigma$, defined by
  \begin{equation}
    \Sigma^{(\delta)} := \bigcup_{|z| < \delta} \Sigma_z\,,\quad \mbox{where}
    ~ \Sigma_z := \{x\in \mathbb{R}^d\,|\,\xi(x) = z\}\,,\,  z\in\R^k\,. 
    \label{sigma-delta}
  \end{equation}
 This implies that $\nabla\xi^T\nabla\xi\in \mathbb{R}^{k\times k}$ is
  positive definite. Without any loss of generality, we assume that $\nabla\xi^T\nabla\xi \succeq c_1 I_k$ on $\Sigma^{(\delta)}$
  for some $c_1>0$. 
  \label{rmk-on-set-sigma-delta}
\end{remark}

In the following, we introduce several quantities which will be useful in later analysis.
For notational simplicity, we will often omit the argument (typically $x$) of
a function if it is clear from the context or if it does not add any ambiguity.

We denote by 
\begin{equation*}
  L^2(\Sigma, \cmu) = \Big\{g\,\Big|\,g: \Sigma\rightarrow \mathbb{R}, \int_{\Sigma}
  g^2\,d\cmu < +\infty\Big\}
\end{equation*}
the Hilbert space endowed with the $\cmu$-weighted inner product
\begin{equation}  \label{weighted-ip}
  ( g_1, g_2 )_{\cmu} := \int_{\Sigma} g_1\,g_2\,d\cmu\,, 
  \ \ \forall~ g_1, g_2\in L^2(\Sigma, \cmu)\,.
\end{equation}
For an operator $\mathcal{T}$ with domain 
$D(\mathcal{T}) \subseteq L^2(\Sigma, \cmu)$, we
denote by $\mathcal{T}^*$ its adjoint with respect to \eqref{weighted-ip}. 

We define the matrix-valued functions $\Phi: \mathbb{R}^d\rightarrow \R^{k\times k}$ and
$\Gamma: \mathbb{R}^d\rightarrow \R^{d\times d}$ as  
\begin{equation}
  \Phi := \nabla\xi^T(a-A)\nabla\xi\,, \quad   \Gamma := (a-A)\nabla\xi\nabla\xi^T\,.
  \label{phi-gamma}
\end{equation}
Assumption~\ref{assump-2} implies that $\Phi$ is invertible in the
neighbourhood $\Sigma^{(\delta)}$ (see Remark~\ref{rmk-on-set-sigma-delta} and
the first item of Lemma~\ref{lemma-on-matrix-pi-phi}).
Moreover, we define $P,B:\Sigma^{(\delta)} \rightarrow\R^{d\times d}$ as
\begin{equation}  \label{projection-p}
  P := I_d - (a-A)\nabla\xi  \Phi^{-1}\nabla\xi^T\,, \quad B:= P(a-A).
\end{equation}
Next we introduce the symmetric and antisymmetric parts of $B$
  \begin{equation}
    \label{symmetric-parts-of-b}
     B^{\mathrm{sym}} := \frac12(B+B^T)\,,  \quad B^{\mathrm{asym}} :=
    \frac12(B-B^T)\,, 
  \end{equation}
  and the vector $J=(J_1, J_2, \dots, J_d)^T: \Sigma^{(\delta)} \rightarrow\R^{d}$
  with  \begin{equation}
     J_i := \frac{\mathrm{e}^{\beta U}}{\beta} \sum_{j=1}^d \frac{\partial
      (B^{\mathrm{asym}}_{ij}\mathrm{e}^{-\beta U})}{\partial x_j} = 
      \frac{1}{\beta} \sum_{j=1}^d\frac{\partial B^{\mathrm{asym}}_{ij}}{\partial x_j} - 
    \big[B^{\mathrm{asym}} \nabla U\big]_{i}
\,, \quad 1 \le i \le d\,.
      \label{vector-j}
  \end{equation}

We will also compare the non-reversible numerical scheme
\eqref{scheme-non-reversible} with $A\neq 0$ to the corresponding reversible counterpart with $A=0$,  for which we introduce the following quantities from~\cite{Tony-constrained-langevin2012,Zhang20} 
\begin{equation}  \label{p-b-reversible}
  P_0 := I_d - a\nabla\xi (\nabla\xi^Ta \nabla\xi)^{-1}\nabla\xi^T\,, \quad
  B_0:= P_0a\,, \quad \forall\, x \in \Sigma^{(\delta)}.
\end{equation}
Note that \eqref{p-b-reversible} corresponds to the reversible case (i.e.\  \eqref{projection-p} with $A=0$), for which  $B_0^T=B_0$.

\begin{remark}[$P$ is a projection]
  As we will show in Lemma~\ref{lemma-p-b}, $P=P(x)$ defines a
  projection onto the tangent space $T_x\Sigma_z$ of the level set $\Sigma_z
  := \{x'\in \mathbb{R}^d\,|\,\xi(x') = z\}$, where 
  $x \in \Sigma^{(\delta)}$ and $z=\xi(x) \in \mathbb{R}^k$. Since $P\neq P^T$, in general $P$ is an oblique
  (non-orthogonal) projection. In the reversible setting, i.e.\ $A=0$,  $P_0$
  in \eqref{p-b-reversible} is the orthogonal projection with respect to the inner product weighted by the matrix $a$~\cite{Zhang20}. 
  \label{rmk-p-projection}
\end{remark}

\subsection{Auxiliary results}\label{subsec-summary-of-useful-results}

The main result of this paper (see Theorem~\ref{thm:Err} below) is concerned with the computation of averages with respect to $\mu$. 
As in~\cite[Theorem~3.5]{Zhang20}, the proof of Theorem~\ref{thm:Err}
is based on the Poisson-equation approach developed in~\cite{MattinglyStuartTretyakov10}. In
the following, we discuss two ingredients that are necessary to prove Theorem~\ref{thm:Err}.

The first ingredient is the following stochastic differential equation (SDE) on $\R^d$
    \begin{equation}   \label{sde-on-sigma}
      dX_{s}^i = -\sum_{j=1}^d B_{ij} \frac{\partial U}{\partial x_j}\,ds +
      \frac{1}{\beta} \sum_{j=1}^d \frac{\partial B_{ij}}{\partial x_j}\,ds +
      \sqrt{2\beta^{-1}} \sum_{j=1}^{d_1} (P\sigma)_{ij}\,dW_{s}^j\,, \quad 1\le i \le d\,,
  \end{equation}
  for $s\ge 0$ and $X_0\in \Sigma$. 
  Note that in \eqref{sde-on-sigma} we have omitted the dependence of the coefficients on the 
  state $X_s$ for notational convenience. 
 We have the following two results concerning \eqref{sde-on-sigma} as well as
 its infinitesimal generator $\mathcal{L}$. The proofs are presented in Section~\ref{sec-ergodic-sde}.
  \begin{prop}\label{prop-generator-l}
    The infinitesimal generator $\mathcal{L}$ of the SDE~\eqref{sde-on-sigma} satisfies:
    \begin{enumerate}[label=(\arabic*)]
      \item \label{it:prop-on-l-1}
For any $f\in C^2(\R^d)$,
\begin{equation}  \label{generator-l}
  \mathcal{L}f 
    = \frac{\mathrm{e}^{\beta U}}{\beta} \sum_{i,j=1}^d\frac{\partial}{\partial
    x_j}\left(B_{ij} \mathrm{e}^{-\beta U} \frac{\partial f}{\partial x_i} \right) =
    \frac{\mathrm{e}^{\beta U}}{\beta} \nabla\cdot \left( \mathrm{e}^{-\beta U} B^T \nabla f
    \right)\,.
\end{equation}
      \item \label{it:prop-on-l-2}
	It admits the decomposition
    \begin{equation}  \label{generator-l-decomp}
    \mathcal{L}f 
     =  (\mathcal{A} + \mathcal{S}) f 
     = J\cdot \nabla f + \frac{\mathrm{e}^{\beta
    U}}{\beta} \nabla\cdot \big(\mathrm{e}^{-\beta U}\,B^{\mathrm{sym}} \nabla f\big)\,,
\end{equation}
  where the vector $J=(J_1, J_2, \dots, J_d)^T$ (defined in \eqref{vector-j}) satisfies 
    \begin{equation}	 \label{conditions-met-by-j}
      PJ = J \, ~\mbox{and}\quad \nabla\cdot(J\,\mathrm{e}^{-\beta U}) =
	 0\,, \quad \mbox{on}~\,\Sigma^{(\delta)}\,,
    \end{equation}
and
    \begin{equation}      \label{sym-antisym-of-l}
      \mathcal{A}:=J\cdot \nabla\,, \quad
  \mathcal{S}:=\frac{\mathrm{e}^{\beta U}}{\beta} 
    \nabla\cdot \big(\mathrm{e}^{-\beta U}\,B^{\mathrm{sym}} \nabla \big),
    \end{equation}
    are the non-reversible and reversible parts of $\mathcal{L}$ respectively .
  \item\label{it:prop-on-l-3}
    We have the integration by parts formula 
    \begin{equation}
      \label{integration-by-parts}
      \forall\,f,g\in C^2(\R^d): \ \int_{\Sigma} (\mathcal{L}f)\,g\,d\cmu = -\frac{1}{\beta} \int_{\Sigma} B^T\nabla
      f\cdot\nabla g\,d\cmu.
    \end{equation}
    In particular, for any $f\in C^2(\R^d)$, 
    \begin{equation}
      \label{integration-by-parts-f}
      \int_{\Sigma} (\mathcal{L}f)f\,d\cmu = \int_{\Sigma} (\mathcal{S}f)f\,d\cmu 
      = -\frac{1}{\beta} \int_{\Sigma} B^{\mathrm{sym}}\nabla f\cdot\,\nabla f\,d\cmu  \,.
    \end{equation}
    \item\label{it:prop-on-l-4} For $g\in C^2(\Sigma)$, the adjoints of $\mathcal A$ and $\mathcal L$ in $L^2(\Sigma,\mu)$ satisfy
    \begin{equation*}
    \mathcal{A}^*=-\mathcal{A}, \ \  
      \mathcal{L}^*g=(-\mathcal{A} + \mathcal{S})g 
  =   \frac{\mathrm{e}^{\beta U}}{\beta} \nabla\cdot \left( \mathrm{e}^{-\beta
      U} B \nabla g \right).
    \end{equation*}
    \end{enumerate}
  \end{prop}
\begin{prop}\label{prop:inv-meas}
 The SDE \eqref{sde-on-sigma} can be written as 
    \begin{equation}
    dX_{s}^i =  \bigg(J_i- \sum_{j=1}^d B^{\mathrm{sym}}_{ij}\frac{\partial U}{\partial x_j}+
      \frac{1}{\beta} \sum_{j=1}^d\frac{\partial
      B^{\mathrm{sym}}_{ij}}{\partial x_j}\bigg)\,ds +
      \sqrt{2\beta^{-1}} \sum_{j=1}^{d_1} (P\sigma)_{ij}\,dW_{s}^j\,,\quad  1\le i \le d\,.
    \label{sde-on-sigma-decomp}
  \end{equation}
   Assume that $X_0\in \Sigma$. Then, we have $X_s\in \Sigma$ almost surely for $s \ge 0$. Moreover, $X_s$ is ergodic with respect to the unique invariant
  probability distribution $\cmu$ in \eqref{mu1-intro}. 
\end{prop}

\begin{remark}  \label{rmk-on-sde-and-l}
  We make two remarks regarding the results above.
  \begin{enumerate}[label=(\arabic*),itemsep=0pt]
    \item
Note that the matrix $B$ in \eqref{projection-p} satisfies $B\nabla\xi=B^T\nabla \xi = 0$ which implies $\nabla\xi^T (B^T\nabla f)=0$. Therefore $B^T\nabla f$ on $\Sigma$ can be intrinsically defined using values of $f$ on $\Sigma$. Consequently, when evaluated on $\Sigma$, the right hand sides of~\eqref{generator-l} and~\eqref{integration-by-parts} in fact only
depend on the values of $f,g$ on $\Sigma$. This implies that $\mathcal L$ defines an operator on $C^2(\Sigma)$, which is independent of the choice of extension to $C^2(\R^d)$. 
Similarly, while the coefficients
      of~\eqref{sde-on-sigma} are not well-defined on $\mathbb{R}^d$, but only
      $\Sigma^{(\delta)} \subseteq \mathbb{R}^d$, Proposition~\ref{prop:inv-meas} ensures that \eqref{sde-on-sigma} only evolves on the level set $\Sigma$.
\item
In general, the SDE~\eqref{sde-on-sigma} defines a non-reversible process on
      $\Sigma$. Equation~\eqref{sde-on-sigma-decomp} together with the decomposition~\eqref{generator-l-decomp} show that we can decompose SDE~\eqref{sde-on-sigma} as well as its generator $\mathcal{L}$ into reversible and non-reversible parts, which is similar to the case of related SDEs on $\mathbb{R}^d$~\cite{lelievre2013optimal,DuncanLelievrePavliotis16,ZhangHartmannSchutte16}.
  \end{enumerate}
\end{remark}

Given $f$, when applying the approach in~\cite{MattinglyStuartTretyakov10}, we will
consider the Poisson equation
\begin{equation}
  \mathcal L \psi = f-\bar f\,,\quad \mbox{on}~ \Sigma\quad \textrm{with}~ \E_{\cmu}[\psi] = 0\,,
  \label{eqn-poisson-equation}
\end{equation}
 where $\mathcal L$ is defined in \eqref{generator-l} and $\bar f = \E_{\cmu}[f]$.  
Since $\Sigma$ is a compact connected submanifold, one can easily verify that
the Foster-Lyapunov Criterion~\cite{glynn1996} holds for $\mathcal{L}$, 
which guarantees that the solution to \eqref{eqn-poisson-equation} is unique~\cite[Section 2]{DuncanLelievrePavliotis16}.
 We define the asymptotic variance (see \cite{DuncanLelievrePavliotis16} for the definition of asymptotic variance on $\mathbb{R}^d$)
\begin{equation}
  \chi_{f}^2 = \frac{2}{\beta}\int_{\Sigma} (B^{\textrm{sym}}\nabla\psi)\cdot \nabla\psi\,d\cmu\,,
  \label{eqn-chi}
\end{equation}
where $\psi$ is the solution to \eqref{eqn-poisson-equation}.

The second ingredient is the map $\Theta^A$ \eqref{def-theta-map-intro} and we
need the following properties on its first and second derivatives (for states on $\Sigma$). The proofs, based on analysing the ODE \eqref{phi-map-a-intro}--\eqref{fun-cap-f-intro},  are provided in Section~\ref{sec-map-theta}.

\begin{prop}\label{prop-1st-varphi}
  For $x \in \Sigma$ and $s\geq 0$, the gradient of the ODE flow \eqref{phi-map-a-intro}--\eqref{fun-cap-f-intro}  satisfies 
\begin{equation}
\begin{aligned}
  (\nabla\varphi^A(x,s))^T &= \mathrm{e}^{-s\Gamma} \\
  &= P+ (a-A)\nabla\xi \mathrm{e}^{-s\Phi}
  \Phi^{-1}\nabla\xi^T\,\\
  &=  I_d + (a-A)\nabla\xi (\mathrm{e}^{-s\Phi} - I_k)\Phi^{-1}\nabla\xi^T,
\end{aligned}
  \label{eqn-1st-varphi}
\end{equation}
where $\Phi, \Gamma$ are defined in \eqref{phi-gamma}. 
Furthermore the limiting flow $\Theta^A$ satisfies 
\begin{equation}
  (\nabla\Theta^A)^T(x)= P.
  \label{eqn-of-theta-map}
\end{equation}
\end{prop}
\begin{prop}\label{prop2}
  For $x\in \Sigma$ and $1\le i \le d$, the Hessian of $\Theta^A$
  satisfies
\begin{equation}   \label{eqn-2nd-derivative-theta}
  \sum_{j,r=1}^da_{jr} \frac{\partial^2 \Theta_i^A}{\partial x_j\partial x_r} 
    = \sum_{j=1}^d\frac{\partial B_{ij}}{\partial x_{j}} - \sum_{j,\l=1}^d P_{i\l} \frac{\partial a_{\l j}}{\partial x_{j}} \,.
\end{equation}
\end{prop}
In the following remark we compare Propositions \ref{prop-generator-l}--\ref{prop:inv-meas} 
and Propositions~\ref{prop-1st-varphi}--\ref{prop2} with their reversible counterparts (i.e.\ $A=0$) in~\cite{Zhang20}.

\begin{remark}[Reversible case]
  \begin{enumerate}[label=(\arabic*)]
    \item
    For $A=0$, Propositions \ref{prop-generator-l}--\ref{prop:inv-meas}
      recover the previous results~\cite[Theorem 2.3 and Remark 2.5]{Zhang20}, where the corresponding ergodic SDEs on $\Sigma$ that sample $\cmu$ were
      constructed. The proofs in~\cite{Zhang20} were involved due
      to lengthy calculations for the expression of the Laplacian operator on
      $\Sigma$ viewed as a Riemannian manifold (see~\cite[Appendix A]{Zhang20}). In
	contrast, in Propositions \ref{prop-generator-l}--\ref{prop:inv-meas}, we generalize the
	results in~\cite{Zhang20} using a much simpler argument, thanks to
	Lemma~\ref{lemma-convergence-to-mu1} which allows us to convert
	integrals on the manifold $\Sigma$ to integrals on $\mathbb{R}^d$ where integration by parts formula can be easily applied.
	Note that Propositions \ref{prop-generator-l}--\ref{prop:inv-meas}
	are accordant with \cite[Corollary $2.6$ and Remark $2.7$]{Zhang20}, where non-reversible
	ergodic SDEs on $\Sigma$ were discussed. 
    \item
      Similar to $\Theta^A$, we denote by $\Theta$ the long-time limit of the ODE flow \eqref{phi-map-a-intro}--\eqref{fun-cap-f-intro} with $A=0$. Note that in this case $B$ is replaced by $B_0=P_0a$ (in \eqref{p-b-reversible}).
Therefore~\eqref{eqn-2nd-derivative-theta} becomes
\begin{equation*}
  \sum_{j,r=1}^d a_{jr}\frac{\partial^2 \Theta_i}{\partial x_j\partial x_r}
  = \sum_{j=1}^d \frac{\partial (P_0a)_{ij}}{\partial x_j}- \sum_{j,\l=1}^d
  (P_0)_{i\l}\frac{\partial a_{\l j}}{\partial x_j}\,, \quad 1 \le i \le d\,,
\end{equation*}
  which recovers the result in~\cite[Proposition $3.4$]{Zhang20}. Essentially,  Propositions~\ref{prop-1st-varphi}--\ref{prop2} generalises the previous result~\cite[Proposition $3.4$]{Zhang20} to the non-reversible setting, by bypassing the calculation
  based on studying the eigenvalues of $\Gamma$~\cite{fatkullin2010,Zhang20},
  which will be complex-valued in the current case ($A\neq 0$). We refer to the proofs of Propositions~\ref{prop-1st-varphi}--\ref{prop2} in Section~\ref{sec-map-theta} for details.
  \end{enumerate}
\end{remark}

Finally we introduce a modified version of the ODE
\eqref{phi-map-a-intro}--\eqref{fun-cap-f-intro}, given by  
  \begin{align} \label{phi-map-a-kappa}
  \begin{split}
    \frac{d\varphi^{A,\kappa}(x,s)}{ds} =& -\frac{1}{2}\Big((a-A)\nabla
    (|\xi|^{2-\kappa}\big)\Big)\big(\varphi^{A,\kappa}(x,s)\big)\\
    =& -\frac{2-\kappa}{2} \Big(|\xi|^{1-\kappa} \sum_{\alpha=1}^k
    \frac{\xi_\alpha}{|\xi|}(a-A)\nabla \xi_\alpha\Big)\big(\varphi^{A,\kappa}(x,s)\big)\,,\quad s \ge 0 \,,\\
    \varphi^{A,\kappa}(x,0) =\,& x, \qquad  \forall~x\in \mathbb{R}^d\,, 
\end{split}
\end{align}
where $\kappa \in [0,1)$. Note that the vector field in
\eqref{phi-map-a-kappa} is differentiable in $\Sigma^{(\delta)}\setminus
\Sigma$ and converges to zero continuously as the states approach $\Sigma$.
In fact, it is the same vector field as in \eqref{phi-map-a-intro} up to a
state-dependent scalar factor, and therefore the solution to \eqref{phi-map-a-kappa} has the same limit $\Theta^A$. 
In particular \eqref{phi-map-a-kappa} reduces to the ODE~\eqref{phi-map-a-intro}--\eqref{fun-cap-f-intro} when $\kappa=0$. 

The ODE \eqref{phi-map-a-kappa} is useful in numerically approximating the
map $\Theta^A$ since it converges to $\Theta^A$ in finite time, as summarised in the following result.
  \begin{lemma}
    Let $\kappa\in (0,1)$ and $\Sigma^{(\delta)}$ be the neighbourhood defined in \eqref{sigma-delta}. For any
    $x \in \Sigma^{(\delta)}$, the solution $\varphi^{A,\kappa}(x,s)$ to
    \eqref{phi-map-a-kappa} converges to $\Theta^A(x)$ within finite time
    \begin{equation}      \label{tim-s-bar}
      \bar{s} := \frac{2^{1+\frac{\kappa}{2}}\delta^\kappa}{\kappa (2-\kappa) c_0c_1},
    \end{equation} 
where $c_0, c_1 > 0$ are the constants in Assumption~\ref{assump-1} and Remark~\ref{rmk-on-set-sigma-delta} respectively. 
    \label{lemma-rescaled}
  \end{lemma}

The proof of Lemma~\ref{lemma-rescaled} is given at the end of Section~\ref{sec-map-theta}.

\subsection{Main results}\label{subsec-summary-of-main-results}

We now state our first main result concerning the error estimates for the scheme \eqref{scheme-non-reversible} which uses the exact projection $\Theta^A$.

\begin{thm}\label{thm:Err}
For any $f\in C^2(\Sigma)$, define $\bar{f}=\E_{\cmu}[f]$ and the running average 
  \begin{equation*}
\hat f_n =\frac1n \sum\limits_{\l=0}^{n-1}f(x^\lb)\,,
\end{equation*}
  where $n\in\mathbb N^+$, $x^{(\l)}\in \Sigma$, $\l=0,1,\dots, n-1$, are computed using the
  numerical scheme~\eqref{scheme-non-reversible} with step-size $h>0$, and $T:=nh$. Then there exists $h_{0} > 0$, such that for any $h\in (0,h_0)$ we have the following estimates. 
\begin{enumerate}[label=(\arabic*)]
\item\label{thm:main-1} There exists a constant $C>0$, independent of both $n$ and $h$, such that 
\begin{equation}\label{eq:MainThm-1}
  \big| \E[\hat f_n] - \bar f\big| \leq C\Big(h+\dfrac1T\Big)\,.
\end{equation}
\item \label{thm:main-2} 
There exists constants $C_1, C_2>0$, independent of both $n$ and $h$, such that 
\begin{equation}
\E\big[ |\hat f_n - \bar f|^2\big] \leq \frac{C_1\chi_f^2}{T} + C_2\Big(h^2 + \frac{h}{T}+\frac{1}{T^2}\Big),
  \label{mean-square-error-thm2}
\end{equation}
  where $C_1$ is any constant larger than one, $C_2$ depends on the choice of $C_1$, and $\chi_f^2$ is the asymptotic variance \eqref{eqn-chi}.

\item \label{thm:main-3} For any $\eps\in(0,\frac12)$, there exists 
  a constant $C>0$, independent of both $n$ and $h$, and an almost surely bounded positive random variable $\zeta = \zeta(\omega)$, such that 
\begin{equation*}
\big|\hat f_n - \bar f\big| \leq Ch+\frac{\zeta}{T^{\frac12-\eps}}\,,  \ \text{almost surely}
\end{equation*}
for sufficiently large $n$.
\end{enumerate}
\end{thm}
The proof of Theorem~\ref{thm:Err} is given in Appendix~\ref{app:thm-err}.

\begin{remark}[Choice of random variables $\bm{\eta}^{(\ell)}$]  \label{rmk-on-choice-of-noise}
  Note that the map $\Theta^A$ \eqref{def-theta-map-intro} may not be
  well-defined outside $\Sigma^{(\delta)}$ (also see
  Remark~\ref{rmk-well-posedness-of-limit-theta}).  To avoid technical issues,
  we assume that the random variables used in the scheme \eqref{scheme-non-reversible}
  are almost surely bounded (see \eqref{conditions-on-eta}), since this implies that starting from $x^{(\ell)}\in \Sigma$ the intermediate states $x^{(\ell + \frac{1}{2})}$
  will remain close to $\Sigma$, i.e.\ $x^{(\ell + \frac{1}{2})}\in
  \Sigma^{(\delta)}$, whenever the step-size $h$ is small enough. 
  We note that an alternative way to avoid this issue is by modifying the definition of $\xi$ for states outside $\Sigma^{(\delta)}$, such that the long-time limit $\Theta^A(x)$ of ODE
  \eqref{phi-map-a-intro}--\eqref{fun-cap-f-intro} is well-defined with
  $\Theta^A(x)\in \Sigma$ for all $x\in \mathbb{R}^d$. This follows for
  instance if $\nabla\xi^T\nabla\xi \succeq c_1 I_k$ for some $c_1>0$ on entire
  $\mathbb{R}^d$ (see Remark~\ref{rmk-on-set-sigma-delta} and the proof of
  Proposition~\ref{prop:varphi}). Such an assumption is often adopted when dealing with reaction coordinates~\cite[Section 1.1]{legoll2010effective} (see discussion in Section~\ref{sec:discuss} for more details). 

  In any case, in spite of this technical issue, we indeed expect that
  Theorem~\ref{thm:Err} remains true when 
  $\bm{\eta}^{(\ell)}$ are standard Gaussian random variables, since
  on the one hand $\Theta^A$ is well-defined with value on $\Sigma$
  for quite general starting points and on the other hand 
  it becomes rarer for the intermediate states $x^{(\ell + \frac{1}{2})}$
  to escape from $\Sigma^{(\delta)}$ when $h$ is small. 
\end{remark}

In applications, the conditional probability measure $\cmu$ and the
infinitesimal generator $\mathcal{L}$ are often assumed to satisfy the
Poincar{\'e} inequality~\cite{LelievreZhang18} with some constant $K>0$, i.e.\
       \begin{equation}
	 \mbox{Var}_{\cmu}(g) := \int_{\Sigma} (g-\overline{g})^2\,d\cmu 
	 \le -\frac{1}{K} \int_{\Sigma} (\mathcal{L}g) g\,d\cmu =
	 \frac{1}{K\beta} \int_{\Sigma} (B^{\mathrm{sym}} \nabla g)\cdot \nabla g\,d\cmu \,,
	 \label{poincare-inequality-mu1}
       \end{equation}
       for any $\,g : \Sigma\rightarrow \mathbb{R}$ such that the right hand side above is finite and $\overline{g}=\E_{\mu}[g]$. Here we have
       used~\eqref{integration-by-parts-f} to arrive at the final equality.
       Under this additional assumption, we can further express the mean square error estimate in Theorem~\ref{thm:Err} as follows.
\begin{corollary}  \label{corollary-on-spectral-gap}
  Under the same assumptions as in Theorem~\ref{thm:Err} and further assuming that 
  the Poincar{\'e} inequality (\ref{poincare-inequality-mu1}) is satisfied, we have 
  \begin{align}
    \E\left[\big|\widehat{f}_n - \overline{f}\big|^2\right] \le  
    \frac{2C_1 \mbox{\textnormal{Var}}_{\cmu}(f)}{KT} + C_2\Big(h^2 +
    \frac{h}{T}+\frac{1}{T^2}\Big)\,,
    \label{mean-square-estimate-poincare}
  \end{align}
  where $C_1$ is any constant larger than one, $C_2$ 
  depends on $C_1$ but is independent of both $h$ and $n$, and
  $\mathrm{Var}_{\cmu}(f):= \E_{\mu}[|f-\bar f|^2]$ is the variance of $f$. 
\end{corollary}
The proof of Corollary~\ref{corollary-on-spectral-gap} is given in Appendix~\ref{app:thm-err}.

The following result compares the numerical scheme \eqref{scheme-non-reversible} (and the SDE \eqref{sde-on-sigma}) with $A\neq 0$ (non-reversible) to the case when $A=0$ (reversible). 
The asymptotic variance when $A=0$ is 
\begin{equation}
  \chi_{f,0}^2 = \frac{2}{\beta}\int_{\Sigma} (B_0\nabla\psi_0)\cdot \nabla\psi_0\,d\cmu\,,
  \label{eqn-chi-reversible}
\end{equation}
where $\psi_0$ is the (unique) solution to the Poisson equation 
\begin{equation}
  \mathcal L_0 \psi_0 = f-\bar f\,,   \text{ on } \Sigma, \    \textrm{ such that } \E_{\cmu}[\psi_0] = 0\,,
  \label{eqn-poisson-equation-reversible}
\end{equation}
with 
\begin{equation}
\mathcal L_0=\frac{\mathrm{e}^{\beta U}}{\beta} \nabla\cdot \left(
\mathrm{e}^{-\beta U} B_0 \nabla \right)
  \label{generator-l-reversible}
\end{equation}
and $\bar f = \E_{\cmu}[f]$. 
Consider the following Poincar{\'e} inequality~\cite{LelievreZhang18,Zhang20} with constant $K_0>0$
       \begin{equation}	 \label{poincare-inequality-mu1-reversible}
	 \int_{\Sigma} (g-\overline{g})^2\,d\cmu 
	 \le -\frac{1}{K_0} \int_{\Sigma} (\mathcal{L}_0\,g) g\,d\cmu =
	 \frac{1}{K_0\beta} \int_{\Sigma} (B_0 \nabla g)\cdot \nabla g\,d\cmu 
       \end{equation}
       for any $\,g : \Sigma\rightarrow \mathbb{R}$ such that the right hand
       side is finite. Then we arrive at the following estimates, whose proof is given in Section~\ref{sec-compare}. 
\begin{prop}\label{prop:impro-ass-var}
  Let $K$, $K_0>0$ be the largest constants for which~\eqref{poincare-inequality-mu1}
  and~\eqref{poincare-inequality-mu1-reversible} are satisfied respectively (called the spectral gap), and  
  $\chi^2_{f}$, $\chi^2_{f,0}$ are the asymptotic variances
  in~\eqref{eqn-chi} and~\eqref{eqn-chi-reversible} respectively, where $f\in L^2(\Sigma, \cmu)$.
  We have 
  \begin{enumerate}[label=(\arabic*)]
    \item\label{it:cmp-poincare-constants}
  $K \ge K_0$\,.
  \item \label{it:cmp-asymptotic-variance}
  $\chi^2_{f} \le \chi^2_{f,0}$\,.
  \end{enumerate}
  \label{prop-cmp-poincare-and-asym-variance}
\end{prop}
The following remark summarises the importance of this result when comparing the reversible and non-reversible setting.
\begin{remark}  \label{rmk-compare}
In~\cite{DuncanLelievrePavliotis16} it is shown that linearly adding a non-reversible force to a reversible dynamics leaves the spectral gap unchanged but reduces the asymptotic variance (also see~\cite{HwangHwangSheu93,lelievre2013optimal,improve-reversible-sampler}). 

In our setting with $A\neq 0$,
  Proposition~\ref{prop-cmp-poincare-and-asym-variance} implies that the
  spectral gap of SDE~\eqref{sde-on-sigma} is always larger or equal to, and
  the asymptotic variance is always smaller than, the reversible case $A=0$. In particular, the second item in Proposition~\ref{prop-cmp-poincare-and-asym-variance} 
(together with the second item of Theorem~\ref{thm:Err} and
  Corollary~\ref{corollary-on-spectral-gap}) implies that the mean square
  error of the numerical scheme \eqref{scheme-non-reversible} is smaller than
  the reversible case (compare to~\cite[Corollary 3.7]{Zhang20}) and
  consequently the non-reversible scheme outperforms the reversible scheme in
  the long-time limit $T\rightarrow\infty$. Note that this only reflects the
  analytical improvement and we refer to Section~\ref{sec-numerical} for further discussions about their performance on a concrete example.
\end{remark}

So far we have presented results for the numerical
scheme~\eqref{scheme-non-reversible} which uses the exact long-time limit
$\Theta^A$ \eqref{def-theta-map-intro}. However in practice, numerical approximations of the map $\Theta^A$ are often used to project the states back to
$\Sigma$. Recall the ODE~\eqref{phi-map-a-kappa} with parameter $\kappa\in
[0,1)$, whose solution converges to the same limit $\Theta^A$
in finite time when $\kappa >0 $ (see Lemma~\ref{lemma-rescaled}). Consider the numerical scheme 
\begin{align}
\label{scheme-non-reversible-numerical}
\begin{split} 
  \widetilde{x}^{(\ell + \frac{1}{2})}_i &= \widetilde{x}^\lb_i +
  \sum_{j=1}^{d} \Big( (A_{ij}-a_{ij})\frac{\partial U}{\partial x_j} +
  \frac{1}{\beta}\frac{\partial a_{ij}}{\partial x_j}\Big)(\widetilde{x}^\lb)\, h
    + \sqrt{2 \beta^{-1}h}\, \sum_{j=1}^{d_1}\sigma_{ij}(\widetilde{x}^\lb)\, \eta^\lb_j\,, \quad 1 \le i \le
  d \,, 
  \\
  \widetilde{x}^{(\ell+1)} &= \Tnum\big(\widetilde{x}^{(\ell+\frac{1}{2})}\big)\,, 
\end{split}
\end{align}
for $\ell = 0, 1, \dots$, where $\widetilde{x}^{(0)}\in \Sigma$, and $\Tnum$ denotes a numerical approximation of $\Theta^A$, obtained by
integrating the ODE~\eqref{phi-map-a-kappa} with a rescaling parameter
$\kappa\in (0, 1)$ using the (initial) time step-size $\Delta t>0$, until the
convergence criterion $|\xi|\le \eps_{\mathrm{tol}}$ is met for a given
$\eps_{\mathrm{tol}}>0$. 
Note that the states generated by the scheme
\eqref{scheme-non-reversible-numerical} belong to
$\Sigma^{(\eps_{\mathrm{tol}})}$, i.e.\ $\widetilde x^{(\l)}\in
\Sigma^{(\eps_{\mathrm{tol}})}$ for $\l \ge 0$ (the neighbourhood $\Sigma^{(\eps_{\mathrm{tol}})}$ is defined similarly as $\Sigma^{(\delta)}$ in \eqref{sigma-delta}).
To analyse the scheme~\eqref{scheme-non-reversible-numerical} we make the following assumption on $\Tnum$. 
\begin{assump}
  For $\eps_{\mathrm{tol}} \in (0, \delta)$, there exists $\Delta t_{\mathrm{max}}>0$
and $p\in\mathbb N^+$ such that, for any $\Delta t \in (0, \Delta
  t_{\mathrm{max}})$, the map $\Tnum: \Sigma^{(\delta)}\rightarrow \Sigma^{(\eps_{\mathrm{tol}})}$
  is well-defined and satisfies
  \begin{equation}    \label{assump-on-tnum}
    |\Tnum(x) - \Theta^A(x)| \le C (\Delta t)^p,  \quad \forall\, x \in \Sigma^{(\delta)}\,,
  \end{equation}
  where $C>0$ is a constant independent of both $\Delta t$ and $x\in \Sigma^{(\delta)}$.
  \label{assump-3}
\end{assump}
This assumption is indeed satisfied in practice (see Remark~\ref{rmk-numerical-computation-theta} below).
The following result summarises the error estimates for the scheme~\eqref{scheme-non-reversible-numerical}.

\begin{thm}  \label{thm-numerical-theta}
  Let $n\in\mathbb N^+$, $\kappa \in (0,1)$, step-sizes $h, \Delta t >0$, $\eps_{\mathrm{tol}}>0$ and set $T:=nh$. 
Furthermore let the numerical approximation $\Tnum$ satisfy Assumption~\ref{assump-3}.
  For any $f\in C^2(\Sigma)$, define $\bar{f}=\E_{\cmu}[f]$ and the running average 
  \begin{equation}    \label{running-average-approx}
    \widetilde{f}_n =\frac1n \sum\limits_{\l=0}^{n-1}f(\widetilde{x}^\lb)\,,
\end{equation}
  where $\widetilde{x}^{(\l)}\in \Sigma^{(\eps_{\mathrm{tol}})}$, $\l=0,1,\dots, n-1$, are computed using the
  numerical scheme~\eqref{scheme-non-reversible-numerical},
  and the same notation $f$ is used for some extension of $f$ to $C^2(\Sigma^{(\eps_{\mathrm{tol}})})$.
  Then there exists $h_{0} > 0$, such that for any $h\in (0,h_0)$ and $\Delta
  t\in (0, \Delta t_{\mathrm{max}})$ the following estimates hold. 
\begin{enumerate}[label=(\arabic*)]
\item There exists a constant $C>0$, independent of $n$, $h$ and $\Delta t$, such that 
\begin{equation*}
  \big| \E[\widetilde{f}_n] - \bar f\big| \leq C\Big((\Delta t)^p + h+\dfrac1T \Big)\,.
\end{equation*}
\item  
There exists constants $C_1, C_2>0$, independent of $n$, $h$ and $\Delta t$, such that 
\begin{equation}  \label{mean-square-error-numerical-version}
  \E\big[ |\widetilde f_n - \bar f|^2\big] \leq \frac{C_1\chi_f^2}{T} + C_2\Big((\Delta t)^p + h^2 + \frac{h}{T}+\frac{1}{T^2}\Big),
\end{equation}
  where $C_1$ is any constant larger than one, $C_2$ depends on the choice of $C_1$, and $\chi_f^2$ is the asymptotic variance \eqref{eqn-chi}.

\item  For any $\eps\in(0,\frac12)$, there exists 
  a constant $C>0$, independent of $n$, $h$ and $\Delta t$, and an almost surely bounded positive random variable $\zeta = \zeta(\omega)$, such that 
\begin{equation*}
  \big|\widetilde f_n - \bar f\big| \leq C\Big((\Delta t)^p + h \Big)+\frac{\zeta}{T^{\frac12-\eps}}\,,  \ \text{almost surely}
\end{equation*}
for sufficiently large $n$.
\end{enumerate}
\end{thm}
The proof of Theorem~\ref{thm-numerical-theta} is adapted from the proof of Theorem~\ref{thm:Err}, and we outline it in Appendix~\ref{app:thm-err}.

We note that the result in Corollary~\ref{corollary-on-spectral-gap} can also be extended to the case of numerical scheme~\eqref{scheme-non-reversible-numerical}. 
Let us conclude with the following remark on Assumption~\ref{assump-3}.
\begin{remark}    \label{rmk-numerical-computation-theta}
  In practice, $\Tnum(x)$ corresponds to the numerical solution of $\Theta^{A}(x)$, obtained by
integrating the ODE~\eqref{phi-map-a-kappa} with a rescaling parameter
$\kappa\in (0, 1)$, using numerical ODE methods until the convergence
  criterion $|\xi| < \eps_{\mathrm{tol}}$ is met. The integer $p \in \mathbb{N}^+$ in
  Assumption~\ref{assump-3} corresponds to the order of
  the numerical ODE method used (e.g.\ $p=4$ for the classical fourth-order Runge-Kutta method).
   The theory of numerical ODE methods guarantees that 
  there exists a constant $C>0$, such that the error of the numerical
  solution to \eqref{phi-map-a-kappa} up to the finite time 
$\bar{s}$ \eqref{tim-s-bar} is bounded by $C(\Delta t)^p$ for all initial states in
  $\Sigma^{(\delta)}$. Since the true solution of \eqref{phi-map-a-kappa} converges to $\Theta^A$ within the finite time $\bar{s}$ (see Lemma~\ref{lemma-rescaled})
  and $\xi(\Theta^A(x)) \equiv 0$ for all $x\in \Sigma^{(\delta)}$, given $\eps_{\mathrm{tol}}>0$, there exists $\Delta t_{\mathrm{max}}>0$ such that, for any $\Delta t \in (0,
  \Delta t_{\mathrm{max}})$, the convergence criterion $|\xi| \le \eps_{\mathrm{tol}} $ is met before time $\bar{s}$ and the bound \eqref{assump-on-tnum} is satisfied. 

Numerical methods with adaptive step-sizes can also be used, in which case $\Delta t$ is the initial step-size.
\end{remark}

As stated in Remark~\ref{rmk-on-non-constant-A}, the results in this paper can be generalised to non-constant $A$ and we briefly outline the differences in the following remark. 

\begin{remark}\label{rem:genA}
When $A$ is state-dependent and skew-symmetric, i.e.\ $A: \mathbb{R}^d\rightarrow \mathbb{R}^{d\times d}$ such that $A^T(x)=-A(x)$ for any $x \in \mathbb{R}^d$, the following soft-constrained dynamics 
\begin{equation}
  d X^{i,\eps}_s  = \sum_{j=1}^d \Bigl[ (A_{ij}-a_{ij}) \frac{\partial}{\partial x_j}\Big(U+\frac{1}{\eps} F\Big)\,
  + \frac{1}{\beta} \frac{\partial (a_{ij}-A_{ij})}{\partial x_j}\Bigr](X^{\eps}_s) \,ds
 + \sqrt{2\beta^{-1}} \sum_{j=1}^{d_1}\sigma_{ij}(X^{\eps}_s)\, dW^j_s\,,
  \quad 1 \le i \le d\,,
\end{equation}
admits $\mu^\eps$ \eqref{mu-eps} as an invariant measure under fairly general conditions, with the only change being the $-\partial A_{ij}/\partial x_j$ term as compared to~\eqref{eq:soft-const}. Following the discussion in Section~\ref{sec:Intro}, this motivates the corresponding numerical scheme 
  \begin{align*} 
    x^{(\ell + \frac{1}{2})}_i &= x^\lb_i + \sum_{j=1}^{d} \Big( (A_{ij}-a_{ij})\frac{\partial U}{\partial x_j}  
    + \frac{1}{\beta}\frac{\partial (a_{ij}-A_{ij})}{\partial x_j}\Big)(x^\lb)\, h
    + \sqrt{2 \beta^{-1}h}\, \sum_{j=1}^{d_1}\sigma_{ij}(x^\lb)\, \eta^\lb_j\,, \quad 1 \le i \le
  d \,,\\
  x^{(\ell+1)} &= \Theta^A\big(x^{(\ell+\frac{1}{2})}\big)\,.
\end{align*}
following the same analysis presented in this paper.  The same auxiliary results as in Propositions~\ref{prop-generator-l},~\ref{prop:inv-meas},~\ref{prop-1st-varphi} hold here, with the only difference arising in Proposition~\ref{prop2} where~\eqref{eqn-2nd-derivative-theta} reads
\begin{equation*} 
  \sum_{j,r=1}^da_{jr} \frac{\partial^2 \Theta_i^A}{\partial x_j\partial x_r} 
    = \sum_{j=1}^d\frac{\partial B_{ij}}{\partial x_{j}} - \sum_{j,\l=1}^d P_{i\l} \frac{\partial ( a_{\l j} - A_{\l j})}{\partial x_{j}} \,.
\end{equation*}
The proofs of these results are simple modifications of the constant $A$ case. Finally, all the results stated above in this section carry over to this setting as well.  
\end{remark}

\section{Preliminaries}
\label{sec-prepare}
In this section we state some preliminary results on the quantities introduced in Section~\ref{sec-main-results}. 

Recall that Assumption~\ref{assump-2} implies that $\rank(\nabla\xi) = k$
holds in the neighbourhood $\Sigma^{(\delta)}$ \eqref{sigma-delta} (see Remark~\ref{rmk-on-set-sigma-delta}). 
Let $V:\Sigma^{(\delta)}\rightarrow\R^{d\times(d-k)}$ be a matrix-valued function whose
column vectors are linearly independent and are orthogonal to the column vectors of $\nabla\xi$, i.e.\ 
\begin{equation}
\forall x\in\Sigma^{(\delta)} : \  V^T(x)\nabla\xi(x) = 0\in\R^{(d-k)\times k}.
  \label{condition-for-v}
\end{equation}
Such a function $V$, while not unique and not assumed to be smooth on $\Sigma^{(\delta)}$, always exists since $\nabla\xi$ has full rank in $\Sigma^{(\delta)}$ (see Remark~\ref{rmk-on-v-p-b} below). 

Furthermore, we introduce the matrix-valued function $\Pi:\Sigma^{(\delta)} \rightarrow \R^{(d-k)\times (d-k)}$
defined as  
\begin{equation}
  \Pi := V^T(a-A)^{-1}V\,. 
  \label{pi}
\end{equation}
Note that $a-A$ is invertible on $\mathbb{R}^d$, since $v^T(a-A)v
=v^Tav > 0$, for all $v \in \mathbb{R}^d$ with $v\neq 0$ (Assumption~\ref{assump-1}). 
 The following lemma collects some basic properties of $\Pi$ and $\Phi$.
    \begin{lemma} Assume $x \in \Sigma^{(\delta)}$. The matrices $\Phi,\Pi$ defined in \eqref{phi-gamma} and \eqref{pi} satisfy
      \begin{enumerate}[label=(\arabic*)]
	\item \label{it:Pi-Psi-Inv} 
     $\Pi$ and $\Phi$ are invertible.
	\item\label{it:Phi-Spec}
	  All $k$ eigenvalues of the matrix $\Phi$ have positive real parts.
	\item\label{it:V-Iden}
We have the identity
\begin{equation}  \label{projection-sum-identity}
  V\Pi^{-1}V^T(a-A)^{-1}+ (a-A)\nabla\xi  \Phi^{-1}\nabla\xi^T =I_d\,.
\end{equation}
\end{enumerate}
      \label{lemma-on-matrix-pi-phi}
    \end{lemma}
    \begin{proof}
\noindent\ref{it:Pi-Psi-Inv}   We will show that $\Pi\zeta= V^T(a-A)^{-1}V\zeta = 0$, with $\zeta \in
      \mathbb{R}^{d-k}$, implies $\zeta = 0$. 
      Since the columns of $V$ and
      $\nabla \xi$ span the entire $\mathbb{R}^d$ and the orthogonality \eqref{condition-for-v} holds, $\Pi\zeta=0$ implies that there exists a $v \in \mathbb{R}^k$ such that $(a-A)^{-1}V\zeta = \nabla\xi v$.
      Therefore, 
      \begin{equation*}
	(\nabla\xi v)^T a\nabla\xi v
	= v^T \nabla\xi^T (a-A)\nabla\xi v  = v^T \nabla\xi^T V \zeta = 0\,,
      \end{equation*}
      which implies 
$(a-A)^{-1}V\zeta = \nabla\xi v = 0$. Since the column vectors of $V$ are
      linearly independent, we conclude that $\zeta =0$ and therefore $\Pi$ is
      invertible.

      The invertibility of $\Phi$ follows from the fact that
      $v^T\Phi v = (\nabla \xi v)^T(a-A) \nabla \xi v = (\nabla \xi v)^Ta
      \nabla \xi v > 0$ for all $v \in \mathbb{R}^k$ with $v\neq 0$ (Assumption~\ref{assump-1}).\vspace{0.05in}

\noindent\ref{it:Phi-Spec} 
  	  Let $\lambda\in \mathbb{C}$ be a (complex) eigenvalue of $\Phi$ and
	  assume the corresponding eigenvector is $v\in \mathbb{C}^k$, where $v \neq 0$. 
	    Multiplying both sides of $\nabla\xi^T(a-A)\nabla\xi v = \lambda v\,$
	  by $\bar{v}^T$ (the conjugate transpose of $v$), gives $\bar{v}^T \nabla\xi^T(a-A)\nabla\xi v =
	  \lambda |v|^2$. Taking conjugate transpose and using $\overline{a-A}^T=a+A$, we obtain
	  $\bar{v}^T \nabla\xi^T(a+A)\nabla\xi v = \bar\lambda |v|^2$. 
	  Summing up these two identities, we deduce $\mbox{Re}(\lambda) |v|^2
	  = (\nabla\xi \bar{v})^T a \nabla\xi v$, where $\mbox{Re}(\lambda)$ denotes the real part of $\lambda$.
	  Note that we also have $\nabla\xi v\neq 0$, since $v\neq 0$ and the columns of $\nabla\xi$ are linearly
	  independent. Therefore, using Assumption~\ref{assump-1}, 
	  we conclude $\mbox{Re}(\lambda)>0$. \vspace{0.05in}

\noindent \ref{it:V-Iden}  First, let us show that the column vectors of $V$ and
	  $(a-A)\nabla\xi$ form a linearly independent basis of $\mathbb{R}^d$. Assume 
	  that $V\zeta_1 + (a-A)\nabla\xi\zeta_2 = 0$, where $\zeta_1 \in
	  \mathbb{R}^{d-k}$ and $\zeta_2\in \mathbb{R}^{k}$. 
	  It suffices to show that both $\zeta_1$ and $\zeta_2$ are zero vectors.
	  Multiplying both sides by $\nabla\xi^T$, we get
	  $\nabla\xi^T(a-A)\nabla\xi\zeta_2 = \Phi\zeta_2=0$. The
	  invertibility of $\Phi$ gives $\zeta_2 = 0$, which in turn implies
	  that $\zeta_1 = 0$, since the column vectors of $V$ are linearly independent. 

Next, define $Q:=V\Pi^{-1} V^T(a-A)^{-1} + (a-A)\nabla\xi  \Phi^{-1}\nabla\xi^T$.
	  Since the column vectors of $V$ and $(a-A)\nabla\xi$ form a basis of
	  $\mathbb{R}^d$, any vector $\zeta\in \mathbb{R}^d$ can be written as $\zeta =
	  V\zeta_1 + (a-A)\nabla\xi\zeta_2$, for some $\zeta_1\in
	  \mathbb{R}^{d-k}$ and $\zeta_2\in \mathbb{R}^{k}$.
	  Using $V^T\nabla\xi = 0$, the definitions of $\Phi$ in
	  \eqref{phi-gamma} and $\Pi$ in \eqref{pi}, one can verify that $QV = V$ and $Q(a-A)\nabla\xi = (a-A)\nabla\xi$.
	  Therefore, 
	  \begin{equation*} 
	    Q\zeta = QV\zeta_1 + Q(a-A)\nabla\xi\zeta_2 = V\zeta_1 + (a-A)\nabla\xi\zeta_2 = \zeta\,,
	  \end{equation*}
	  which shows that $Q=I_d$.  
    \end{proof}

The following result summarises some crucial properties of $P$ and $B$ defined in \eqref{projection-p}.    
In particular this result states that $P$ is a projection.
\begin{lemma} \label{lemma-p-b}
For any $x\in \Sigma^{(\delta)}$ we can write 
\begin{equation}  \label{matrix-p-b}
  P= V\Pi^{-1}V^T(a-A)^{-1}, \,\quad B = V\Pi^{-1}V^T\,.
\end{equation}
The matrix $P$ satisfies $P^2=P$, $PV=V$ and $\nabla\xi^TP=0$. Furthermore, we have the relation 
\begin{equation} \label{relation-p-b}
PaP^T =B^\mathrm{sym}= \frac12 (B + B^T)\,.
\end{equation}
\end{lemma}
\begin{proof}
The relation~\eqref{matrix-p-b} follows readily from the definition of $P$ in~\eqref{projection-p} and the identity~\eqref{projection-sum-identity}. It is straightforward to verify that $P^2=P$ and $PV=V$.  Using~\eqref{matrix-p-b} we compute
\begin{align*}
  PaP^T &= V\Pi^{-1}V^T(a-A)^{-1}a (a+A)^{-1}V\Pi^{-T}V^T\\
&= \frac12 V\Pi^{-1}V^T(a-A)^{-1}\big[(a+A) + (a-A)\big] (a+A)^{-1}V\Pi^{-T}V^T\\
&=   \frac12 \Big[V\Pi^{-1}V^T(a-A)^{-1} V\Pi^{-T}V^T +
  V\Pi^{-1}V^T(a+A)^{-1} V\Pi^{-T}V^T\Big]\\
&= \frac12 \big(V\Pi^{-T}V^T + V\Pi^{-1}V^T\big)\\
  &= \frac12 (B + B^T) 
   =B^\mathrm{sym}\,,
\end{align*}
where the first equality follows since $(a-A)^{-T} = (a+A)^{-1}$, the second
  equality follows since $2a=(a+A)+(a-A)$ and the fourth equality follows from
  the definition of $\Pi$ in \eqref{pi}.      
\end{proof}

\begin{remark}  \label{rmk-on-v-p-b}
  Note that the function $V$ satisfying \eqref{condition-for-v} is not unique and we do
  not assume the smoothness of $V$ on $\Sigma^{(\delta)}$. In spite of the relationship~\eqref{matrix-p-b}, the matrices $P$ and $B$ are in fact independent of the choice of
  $V$ by definition~\eqref{projection-p} and are $C^3$-smooth on
  $\Sigma^{(\delta)}$ under Assumption~\ref{assump-2}. 
\end{remark}
\section{Ergodic SDEs on the submanifold}
\label{sec-ergodic-sde}
The goal of this section is to study the diffusion process \eqref{sde-on-sigma} with $X_0\in \Sigma$. In particular, we prove 
Proposition~\ref{prop-generator-l} which identifies its infinitesimal
generator $\mathcal{L}$ and decomposes it into symmetric and anti-symmetric
parts, and Proposition~\ref{prop:inv-meas} on its ergodicity. 

First, let us recall the convergence of the probability measure
$\mu^\eps$~\eqref{mu-eps} on $\mathbb{R}^d$ to $\cmu$ \eqref{mu1-intro} on $\Sigma$.
\begin{lemma}\label{lem:mueps-conv}
The probability measure $\mu^\eps$ converges to $\cmu$ in the sense that for any $f\in C_b(\R^d)$ we have
\begin{equation*}  \label{lemma-convergence-to-mu1}
\lim\limits_{\eps\rightarrow 0} \int_{\R^d} fd\mu^\eps = \int_{\Sigma} fd\cmu.
\end{equation*}
\end{lemma}
We omit the proof of Lemma~\ref{lemma-convergence-to-mu1}, since it is standard and follows by applying the co-area formula, the pointwise convergence
$\exp(-\frac{1}{2\eps}\sum\xi^2_\alpha(x)) \rightarrow \mathds{1}_{\Sigma}(x)$,
where $ \mathds{1}_{\Sigma}$ is the indicator function on the set $\Sigma$, and the
dominated convergence theorem. Lemma~\ref{lemma-convergence-to-mu1} is useful
below as it allows us to convert integrals on $\Sigma$ to integrals on $\mathbb{R}^d$, where it is easier to apply integration by parts formula.

Next, we give the proof of Proposition~\ref{prop-generator-l}.
\begin{proof}[Proof of Proposition~\ref{prop-generator-l}]
\noindent\ref{it:prop-on-l-1}
We prove \eqref{generator-l}.
Since  $(P\sigma)(P\sigma)^T=P\sigma\sigma^TP^T=PaP^T$, the generator of SDE \eqref{sde-on-sigma} for a test function $f$ is
  \begin{equation}
    \mathcal{L}f = - \sum_{i,j=1}^d B_{ij} \frac{\partial U}{\partial x_j} \frac{\partial
    f}{\partial x_i} + \frac{1}{\beta} 
    \sum_{i,j=1}^d \frac{\partial B_{ij}}{\partial x_j}
    \frac{\partial f}{\partial x_i} + \frac{1}{\beta} 
    \sum_{i,j=1}^d (PaP^T)_{ij} \frac{\partial^2 f}{\partial x_i\partial
    x_j}\,.
    \label{generator-l-in-the-proof}
  \end{equation}
  To prove \eqref{generator-l}, it is sufficient to note that the second-order derivative terms 
  in \eqref{generator-l-in-the-proof} satisfy that
\begin{equation*}
\sum\limits_{i,j=1}^d (PaP^T)_{ij} \frac{\partial^2 f}{\partial x_i\partial
  x_j}=\frac12\sum\limits_{i,j=1}^d B_{ij} \frac{\partial^2 f}{\partial
  x_i\partial x_j}+\frac12\sum\limits_{i,j=1}^d B_{ji} \frac{\partial^2
  f}{\partial x_i\partial x_j}=\sum\limits_{i,j=1}^d B_{ij} \frac{\partial^2 f}{\partial x_i\partial x_j},
\end{equation*}
  where we have used \eqref{relation-p-b}  to arrive at the first equality and the second equality follows by interchanging indices and using the symmetry of $\nabla^2 f$.  

\noindent\ref{it:prop-on-l-2}
We first prove~\eqref{conditions-met-by-j} for the vector $J$ and then \eqref{generator-l-decomp} for the generator $\mathcal{L}$.
  Since $B\nabla\xi =B^T\nabla\xi =0$ by~\eqref{matrix-p-b}, we have
  $(B^{\mathrm{asym}})^T\nabla\xi= B^{\mathrm{asym}}\nabla\xi =0$. Therefore using~\eqref{vector-j}, for all $1 \le \alpha \le k$,
  \begin{equation}\label{eq:J-orth-der-xi}
    \begin{aligned}
      J\cdot \nabla \xi_\alpha
      = \frac{\mathrm{e}^{\beta U}}{\beta} \sum_{i,j=1}^d \frac{\partial
      (B^{\mathrm{asym}}_{ij}\mathrm{e}^{-\beta U})}{\partial x_j}
      \frac{\partial \xi_\alpha}{\partial x_i}
      = \frac{\mathrm{e}^{\beta U}}{\beta}  
      \nabla\cdot\big[\mathrm{e}^{-\beta U}
      (B^{\mathrm{asym}})^T\,\nabla\xi_\alpha\big] - 
\frac{1}{\beta} \sum_{i,j=1}^d B^{\mathrm{asym}}_{ij}
      \frac{\partial^2 \xi_\alpha}{\partial x_i \partial x_j}
      = 0\,,
    \end{aligned}
  \end{equation}
  where the last equality follows since $B^{\mathrm{asym}}$ is anti-symmetric and therefore 
  $\sum_{i,j=1}^d B^{\mathrm{asym}}_{ij} \frac{\partial^2 \xi_\alpha}{\partial
  x_i\partial x_j} = 0$. Using~\eqref{projection-p} we find $PJ=J$. Concerning
  the second identity in \eqref{conditions-met-by-j}, using the anti-symmetry of $B^{\mathrm{asym}}$, we find
  \begin{equation*}
\nabla\cdot(J\,\mathrm{e}^{-\beta U}) = 
\frac{1}{\beta} \sum_{i,j=1}^d \frac{\partial^2
      (B^{\mathrm{asym}}_{ij}\mathrm{e}^{-\beta U})}{\partial x_i\partial x_j}
      = 0\,.
  \end{equation*}

  Concerning \eqref{generator-l-decomp}, from \eqref{generator-l} we can compute 
\begin{align}
  \mathcal{L}f   
    &= \frac{\mathrm{e}^{\beta U}}{\beta} \sum_{i,j=1}^d \frac{\partial}{\partial x_j}\left(B_{ij}
    \mathrm{e}^{-\beta U} \frac{\partial f}{\partial x_i} \right) \nonumber\\
    &= \frac{\mathrm{e}^{\beta U}}{\beta} \sum_{i,j=1}^d
    \frac{\partial}{\partial x_j}\left(B^{\mathrm{asym}}_{ij}
    \mathrm{e}^{-\beta U} \frac{\partial f}{\partial x_i} \right)  +
    \frac{\mathrm{e}^{\beta U}}{\beta} \sum_{i,j=1}^d \frac{\partial}{\partial x_j}\left(B^{\mathrm{sym}}_{ij}
    \mathrm{e}^{-\beta U} \frac{\partial f}{\partial x_i} \right) \nonumber\\
    &= \sum_{i=1}^d J_i\frac{\partial f}{\partial x_i} 
    + \frac{1}{\beta} \sum_{i,j=1}^d
    B^{\mathrm{asym}}_{ij} \frac{\partial^2 f}{\partial x_i\partial x_j}   +
    \frac{\mathrm{e}^{\beta
    U}}{\beta} \sum_{i,j=1}^d \frac{\partial}{\partial
  x_j}\left(B^{\mathrm{sym}}_{ij} \mathrm{e}^{-\beta U} \frac{\partial
    f}{\partial x_i} \right) \nonumber\\
    &= J\cdot \nabla f + \frac{\mathrm{e}^{\beta
    U}}{\beta} 
    \nabla\cdot \big(\mathrm{e}^{-\beta U}\,B^{\mathrm{sym}} \nabla f\big)\,,
  \label{generator-l-rmk}
\end{align}
  where the third equality follows from the definition~\eqref{vector-j} of $J$, and the final equality follows from the anti-symmetry of $B^{\mathrm{asym}}$. This proves~\eqref{generator-l-decomp}.

\noindent\ref{it:prop-on-l-3}
  We prove the integration by parts formula~\eqref{integration-by-parts}. 
  Since the integrals involved in~\eqref{integration-by-parts} are defined on $\Sigma$, without any loss of generality we can assume that $f=g=0$ on $\mathbb{R}^d\setminus \Sigma^{(\delta)}$. Under this assumption the integrations in the following calculations 
are well defined on $\mathbb{R}^d$, although $B$ is defined on $\Sigma^{(\delta)}$. Using the expression \eqref{generator-l} of $\mathcal{L}$, 
  the convergence of $\mu^\eps$ to $\cmu$ stated in Lemma~\ref{lemma-convergence-to-mu1}, 
  as well as integration by parts formula on $\mathbb{R}^d$, we find
  \begin{align}
    \int_{\Sigma} (\mathcal{L}f)\,g \,d\cmu 
      & =  \frac{1}{\beta}\int_{\Sigma}
      \mathrm{e}^{\beta U}\, \Big[\nabla\cdot\left( \mathrm{e}^{-\beta U} B^T
      \nabla f \right)\Big]\, g\, d\cmu \notag\\
      & = \frac{1}{\beta} \lim_{\eps \rightarrow 0} 
	     \int_{\mathbb{R}^d}
      \mathrm{e}^{\beta U} \Big[\nabla\cdot\left(\mathrm{e}^{-\beta U} B^T \nabla f
      \right)\Big]\,g\, d\mu_\eps \notag\\
    & = \frac{1}{\beta} \lim_{\eps \rightarrow 0} 
\frac{1}{Z^\eps} \int_{\mathbb{R}^d} 
\mathrm{e}^{\beta U} \Big[\nabla\cdot\left( \mathrm{e}^{-\beta U} B^T \nabla f
    \right)\Big]\,g\,
             \exp\Big[-\beta\Big(U + \frac{|\xi|^2}{2\eps} \Big)\Big]\,dx \notag
	     \\
    & = - \frac{1}{\beta} \lim_{\eps \rightarrow 0} \frac{1}{Z^\eps} \int_{\mathbb{R}^d} 
 \mathrm{e}^{-\beta U} \left(B^T \nabla f \right)\cdot
	     \nabla \left[g\,\exp\Big(-\frac{\beta|\xi|^2}{2\eps} \Big)\right] \,dx\notag \\
    & = - \frac{1}{\beta} \lim_{\eps \rightarrow 0} \frac{1}{Z^\eps} \int_{\mathbb{R}^d} 
    \mathrm{e}^{-\beta U} \Big[\left(B^T \nabla f \right)\cdot
	     \nabla g\Big]\,\exp\Big(-\frac{\beta|\xi|^2}{2\eps} \Big) \,dx\notag \\
   & \quad +    \sum_{\alpha=1}^k \lim_{\eps \rightarrow 0} 
\frac{1}{\eps Z^\eps} \int_{\mathbb{R}^d} 
  \mathrm{e}^{-\beta U} \Big[(\nabla f^T B \nabla \xi_\alpha)\,\xi_\alpha\Big]
  \exp\Big(-\frac{\beta|\xi|^2}{2\eps}\Big)\,g\,dx\notag \\
    & = - \frac{1}{\beta} \lim_{\eps \rightarrow 0} \frac{1}{Z^\eps} \int_{\mathbb{R}^d} 
    \mathrm{e}^{-\beta U} \Big[\left(B^T \nabla f \right)\cdot
	     \nabla g\Big]\,\exp\Big(-\frac{\beta|\xi|^2}{2\eps} \Big) \,dx\notag \,,
  \end{align}
    where the third equality follows from the definition~\eqref{mu-eps} of $\mu^\eps$, the fourth equality follows from integration by parts in $\R^d$ and the final equality follows since $B\nabla\xi=0$. Applying
    Lemma~\ref{lemma-convergence-to-mu1} once more we find
  \begin{align*}
    \int_{\Sigma} (\mathcal{L}f)\,g \,d\cmu 
    & = - \frac{1}{\beta} \lim_{\eps \rightarrow 0} \frac{1}{Z^\eps} \int_{\mathbb{R}^d} 
    \mathrm{e}^{-\beta U} \Big[\left(B^T \nabla f \right)\cdot
	     \nabla g\Big]\,\exp\Big(-\frac{\beta|\xi|^2}{2\eps} \Big) \,dx \\
    & = - \frac{1}{\beta} \lim_{\eps \rightarrow 0} \int_{\mathbb{R}^d} 
    \left(B^T \nabla f \right)\cdot
	     \nabla g \,d\mu_\eps\\
    & = - \frac{1}{\beta} \int_{\Sigma} \left(B^T \nabla f \right)\cdot \nabla g \,d\cmu\,,
    \end{align*}
    which proves~\eqref{integration-by-parts}. Note that we can use
    Lemma~\ref{lemma-convergence-to-mu1} since $B\in C^3(\Sigma^{(\delta)})$
    (see Remark~\ref{rmk-on-v-p-b}) and $B^T\nabla f=0$ in $\R^d\setminus \Sigma^{(\delta)}$ by the choice of $f$.
    Identity~\eqref{integration-by-parts-f} follows by using~\eqref{integration-by-parts} and noting that $B^{\mathrm{sym}}$ is the
    symmetric part of $B$. 
    
\noindent\ref{it:prop-on-l-4}   
Using Lemma~\ref{lem:mueps-conv} and integrating by parts in $\R^d$ we find
\begin{align*}
-\int_{\Sigma} (\mathcal A f) \,g\,d\mu &= \lim_{\eps\rightarrow 0}
  \frac{1}{Z^\eps}\int_{\R^d}f\,\nabla\cdot \Bigl\{J  \exp\Bigl[-\beta
  \Big(U+\frac{|\xi|^2}{2\eps} \Big) \Bigr]g\Bigr\}\,dx\\
&= \lim_{\eps\rightarrow 0} \int_{\R^d} f\, J\cdot\nabla g \,d\mu^\eps - \lim_{\eps\rightarrow 0} \frac{\beta}{\eps} \sum_{\alpha=1}^k\int_{\R^d}  \xi_\alpha  J\cdot \nabla\xi_\alpha  \,g\, d\mu^\eps \\
&= \lim_{\eps\rightarrow 0}\int_{\Sigma} f \, J\cdot\nabla g\, d\mu^\eps \\
  & =  \int_{\Sigma} f\, (\mathcal A g)\,d\mu,
\end{align*}
where the second equality follows from~\eqref{conditions-met-by-j} and the third equality follows from~\eqref{eq:J-orth-der-xi}.
Note that to perform this calculation we have continuously extended $f,g\in
C^2(\Sigma)$ to $C^2(\R^d)$ such that the extensions are supported within $\Sigma^{(\delta)}$.
This shows that $\mathcal A^* = - \mathcal A$. 

Define 
$\overline{\mathcal{L}}:= \frac{\mathrm{e}^{\beta U}}{\beta} \nabla\cdot
\left( \mathrm{e}^{-\beta U} B \nabla g \right)$. Since
$B^T\nabla\xi=B\nabla\xi=0$, we can repeat the calculations in the previous item above, which yields
\begin{equation*}
\int_{\Sigma} (\overline{\mathcal{L}}f)\,g \,d\cmu = - \frac{1}{\beta} \int_{\Sigma} \left(B \nabla f \right)\cdot \nabla g \,d\cmu =  - \frac{1}{\beta} \int_{\Sigma}  \nabla f\cdot \left(B^T\nabla g \right) \,d\cmu  =\int_{\Sigma} (\mathcal{L}g)\,f \,d\cmu.
\end{equation*}
and therefore $\overline{\mathcal L} = \mathcal L^*$.
\end{proof}
We are ready to prove Proposition~\ref{prop:inv-meas}, which shows that the SDE~\eqref{sde-on-sigma} evolves on $\Sigma$ and is ergodic with respect to the target measure $\cmu$~\eqref{mu1-intro}.

\begin{proof}[Proof of Proposition~\ref{prop:inv-meas}]
Identity~\eqref{sde-on-sigma-decomp} follows directly as a result of~\eqref{generator-l-decomp} and ~\eqref{relation-p-b}.
  Since $V^T\nabla\xi=0$, using \eqref{matrix-p-b} we have 
  $\nabla\xi^T P = \nabla\xi^T B=0$. Therefore~\eqref{generator-l} implies $\mathcal{L}\xi = 0$. Applying It\^o's lemma we obtain
  \begin{equation*}
    d\xi(X_s) = (\mathcal{L}\xi)(X_s)\,ds + \sqrt{2\beta^{-1}}(\nabla\xi^T P\sigma)(X_s)\,dW_s = 0\,,
  \end{equation*}
and, since $X_0\in\Sigma$, $\xi(X_s)=\xi(X_0)=0$ almost surely, i.e.\ $X_s\in\Sigma$ for any $s \ge 0$. 

  Applying the integration by parts formula~\eqref{integration-by-parts} (with $g\equiv 1$), we find 
  \begin{equation*}
    \int_{\Sigma} \mathcal{L}f\, d\cmu = 0\,,
  \end{equation*}
  for all test functions $f$, which implies that $\cmu$ is invariant under the evolution of the process $X_s$.
    The ergodicity follows from the fact that $\mathcal{L}$ is an elliptic operator on the compact connected submanifold $\Sigma$. We
    omit the details but refer to~\cite{projection_diffusion,FaouL09,Zhang20} for further discussions.
\end{proof}

\section{Study of the map $\Theta^A$}
\label{sec-map-theta}
In this section we collect crucial properties of the projection $\Theta^A$,
the flow $\varphi^A$ of the ODE~\eqref{phi-map-a-intro}--\eqref{fun-cap-f-intro}, and the flow
$\varphi^{A,\kappa}$ of the modified ODE~\eqref{phi-map-a-kappa}. In
particular, we prove Propositions~\ref{prop-1st-varphi}--\ref{prop2} which
concern the derivatives of $\Theta^A$ and Lemma~\ref{lemma-rescaled} which
concerns the convergence of the flow $\varphi^{A,\kappa}$ to $\Theta^A$ in finite time.

Recall the map $\Theta^A:\R^d\rightarrow\R^d$ is defined via the long-time limit
\begin{equation}
\Theta^{A}(x)=\lim\limits_{s\rightarrow +\infty} \varphi^{A}(x,s).
  \label{theta-a-def}
\end{equation}
Here $\varphi^A:\R^d\times[0,+\infty)\rightarrow\R^d$ is an ODE flow
\begin{align}
  \begin{split}
    \frac{d\varphi^{A}(x,s)}{ds} =& -((a-A)\nabla F\big)\big(\varphi^{A}(x,s)\big)\,,\quad
    \varphi^{A}(x,0) = x, \qquad  \forall~x\in \mathbb{R}^d\,, 
\end{split}
  \label{phi-map-a}
\end{align}
and the function $F:\R^d\rightarrow\R$ is 
\begin{equation}
  F(x):=\frac{1}{2}|\xi(x)|^2=\frac{1}{2}\sum_{\alpha=1}^{k}\xi_\alpha^2(x)\,.
\label{fun-cap-f}
\end{equation}

The following result summarises the well-posedness and regularity of the flow.
\begin{prop}\label{prop:varphi}
The system~\eqref{phi-map-a} admits a unique solution $\varphi^A\in C^4(\R^d\times[0,+\infty))$.
  For any $x\in\Sigma^{(\delta)}$, the limit in \eqref{theta-a-def} is
  well defined and $\Theta^A: \Sigma^{(\delta)}\rightarrow \Sigma$ is a
  $C^4$-differentiable map. 
Furthermore, $\Theta^A(x)=x$ for any $x \in \Sigma$.
\end{prop}
\begin{proof}
The well-posedness of solution to~\eqref{theta-a-def} is standard, since $F$ is a Lyapunov function. 
  In fact, note that Assumption~\ref{assump-2} implies $\nabla\xi^T\nabla\xi \succeq c_1 I_k$ on $\Sigma^{(\delta)}$ for some constant $c_1 > 0$ (see Remark~\ref{rmk-on-set-sigma-delta}). Therefore using \eqref{fun-cap-f} and Assumption~\ref{assump-1}, for any $x \in \Sigma^{(\delta)}$ we find
  \begin{equation}
  (\nabla F^T a \nabla F)(x)=\big(\xi^T(\nabla\xi^Ta\nabla\xi)\xi\big)(x) 
    \ge c_0 (\xi^T\nabla\xi^T\nabla\xi \xi) (x)\ge c_0c_1 |\xi(x)|^2 = 2c_0 c_1 F(x), 
    \label{estimate-gradf-by-f}
  \end{equation} 
     from which we can derive 
  \begin{equation}
    \frac{dF(\varphi^{A}(x,s))}{ds} = - \big(\nabla F^T (a-A)
    \nabla F\big)(\varphi^{A}(x,s)) 
    \le - 2 c_0 c_1 F(\varphi^{A}(x,s))\,,
    \label{estimate-of-Lyapunov-f-kappa0}
  \end{equation}
  where the first equality follows from \eqref{phi-map-a} and the inequality
  follows from the antisymmetry of $A$ and the estimate \eqref{estimate-gradf-by-f}. 
  Consequently, $F(\varphi^A(x,s))$ and thereby $|\xi(\varphi^A(x,s))|$ converge
  exponentially to zero for points on $\Sigma^{(\delta)}$. 
Integrating both sides of \eqref{phi-map-a} and using the exponential decay of
  $|\xi(\varphi^A(x,s))|$, one can obtain that the limit on the right hand side
  of \eqref{theta-a-def} exists and therefore the map $\Theta^A$ is
  well-defined. The differentiability of $\Theta^A$ can be verified similarly, by
  integrating the ODEs for the derivatives of $\varphi^A$ (see
  \eqref{flow-map-Grad-general}--\eqref{flow-map-Hess-general} below) and
  proving that the order of integrations and limits ($s \rightarrow +\infty$) can be switched.
  Furthermore, if the initial point of the flow~\eqref{phi-map-a} $x\in\Sigma$, 
  \eqref{eq:nablaF-Sigma} below implies the right hand side of \eqref{phi-map-a} is zero, and it follows that 
$\varphi^A(x,s)=x$ for any $s\geq 0$, and therefore $\Theta^A(x)=x$.
\end{proof}

\begin{remark}  \label{rmk-well-posedness-of-limit-theta}
Proposition~\ref{prop:varphi} states that the limit $\Theta^A(x)$ in \eqref{theta-a-def} is well-defined with value in $\Sigma$ for points $x\in \Sigma^{(\delta)}$. This is enough to analyse the
  scheme~\eqref{scheme-non-reversible}, since we
  assume the random variables \eqref{conditions-on-eta} used in the scheme~\eqref{scheme-non-reversible} are bounded (see Remark~\ref{rmk-on-choice-of-noise}). 
  At the same time,  thanks to the existence of the natural Lyapunov function \eqref{fun-cap-f}  of ODE \eqref{phi-map-a}, in concrete cases we actually can expect that the limit
  $\Theta^A(x)$ exists and $\Theta^A(x) \in \Sigma$, for quite general states
  $x$ belonging to a set (or even $\mathbb{R}^d$) that is larger than $\Sigma^{(\delta)}$ 
  (see Remark~\ref{rmk-on-set-sigma-delta} and related discussion in Section~\ref{sec:discuss}). 
\end{remark}

In what follows we will make use of the derivatives of $F$, for $1 \le i,j \le
d$ and $x \in \mathbb{R}^d$,
\begin{align*}
    (\nabla F(x))_{i} &= \frac{\partial F}{\partial x_i}(x) =
  \sum\limits_{\alpha=1}^k \xi_\alpha(x)\frac{\partial \xi_\alpha}{\partial
  x_i}(x), \\
    (\nabla^2 F(x))_{ij} &= \frac{\partial^2F}{\partial x_i\partial x_j}(x) =
    \sum\limits_{\alpha=1}^k\bigg[ \frac{\partial \xi_\alpha}{\partial
    x_i}(x)\frac{\partial \xi_\alpha}{\partial x_j}(x) +
    \xi_\alpha(x)\frac{\partial^2\xi_\alpha}{\partial x_i\partial x_j}(x)
    \bigg],
\end{align*}
and therefore, in particular,
\begin{equation}\label{eq:nablaF-Sigma}
\forall\, x\in\Sigma: \ \nabla F(x)=0, \quad \nabla^2 F(x) = \nabla\xi\nabla\xi^T.
\end{equation}

We will study the derivatives of $\Theta^A$ through the derivatives of the
flow $\varphi^A$. Taking derivatives in~\eqref{phi-map-a},  for $1 \le i, j \le d$ we obtain
\begin{equation} \label{flow-map-Grad-general}
  \begin{aligned}
    \frac{d}{ds} \frac{\partial \varphi^A_i}{\partial x_j}(x,s) &=
	-\sum_{r',i'=1}^d\Big((a-A)_{ir'}\frac{\partial^2
	F}{\partial x_{r'}\partial x_{i'}}+
	\frac{\partial a_{ir'}}{\partial x_{i'}}\frac{\partial F}{\partial x_{r'}} \Big)\big(\varphi^A(x,s)\big) \, \frac{\partial\, \varphi_{i'}^A}{\partial x_j} (x,s)\,, \\
    \frac{\partial \varphi^A_i}{\partial x_j}(x,0) &= \delta_{ij}\,, 
  \end{aligned}
\end{equation}
for any $(x,s) \in \mathbb{R}^{d}\times [0,+\infty)$.
Similarly, for $1 \le i, j,r \le d$, the second-order spatial derivatives
satisfy
\begin{align}
      \frac{d}{ds} \frac{\partial^2 \varphi_i^A}{\partial x_j\partial
      x_r}(x,s) &= -\sum_{i',r'=1}^d\left((a-A)_{ir'}\frac{\partial^2
	F}{\partial x_{r'}\partial x_{i'}}
	+ \frac{\partial a_{ir'}}{\partial x_{i'}}\frac{\partial F}{\partial x_{r'}} 
	\right)\big(\varphi^A(x,s)\big)
      \frac{\partial^2\, \varphi_{i'}^A}{\partial x_j \partial x_r}(x,s) \nonumber \\
  & \quad  	-\sum_{i',j',r'=1}^d \left(
	(a - A)_{ir'}\frac{\partial^3 F}{\partial x_{r'}\partial x_{i'}\partial
	    x_{j'}}
	    +2 \frac{\partial a_{ir'}}{\partial x_{i'}}\frac{\partial^2 F}{\partial x_{r'}\partial x_{j'}}
  + \frac{\partial^2 a_{ir'}}{\partial x_{i'}\partial x_{j'}}\frac{\partial F}{\partial x_{r'}}
	    \right)\big(\varphi^A(x,s)\big) \label{flow-map-Hess-general}\\
	    & \hspace{1cm} \times \frac{\partial\, \varphi_{i'}^A}{\partial x_j} (x,s)
      \frac{\partial\, \varphi_{j'}^A}{\partial x_r} (x,s)\,, \qquad \forall~\,(x,s) \in \mathbb{R}^{d}\times [0,+\infty)\,,\nonumber\\
  \frac{\partial^2 \varphi_i^A}{\partial x_j\partial x_r}(x,0) &= 0 \,, \quad
  \forall~ x \in \mathbb{R}^d\,.\nonumber
\end{align}

In particular, when $x\in\Sigma$, using the fact that $\varphi^A(x,s)\equiv x$ for
all $s \ge 0$ and $\nabla F(x)=0$ (see~\eqref{eq:nablaF-Sigma}), in a compact notation
the ODE~\eqref{flow-map-Grad-general} reads  
\begin{equation}\label{flow-map-Grad}
  \begin{aligned}
    \frac{d}{ds}(\nabla\varphi^A(x,s))^T &= -
    (a-A)\nabla^2F (\nabla\varphi^A(x,s))^T \,, \quad \forall~s \in [0,+\infty)\,,\\
    \nabla\varphi^A(x, 0) &= I_d\,,
  \end{aligned}
\end{equation}
while the ODE \eqref{flow-map-Hess-general} simplifies to
\begin{equation}\label{flow-map-Hess}
\begin{aligned}
      \frac{d}{ds} \frac{\partial^2 \varphi_i^A}{\partial x_j\partial
      x_r}(x,s) &=   	-\sum_{i',j',r'=1}^d \left(
	(a - A)_{ir'}\frac{\partial^3 F}{\partial x_{r'}\partial x_{i'}\partial
	    x_{j'}}
	    +2 \frac{\partial a_{ir'}}{\partial x_{i'}}\frac{\partial^2 F}{\partial x_{r'}\partial x_{j'}}
	    \right) \frac{\partial\, \varphi_{i'}^A}{\partial x_j} (x,s)
      \frac{\partial\, \varphi_{j'}^A}{\partial x_r} (x,s) \\
      &\quad -\sum_{i',r'=1}^d(a-A)_{ir'}\frac{\partial^2 F}{\partial x_{r'}\partial x_{i'}} 
      \frac{\partial^2\, \varphi_{i'}^A}{\partial x_j \partial x_r}(x,s) \,,
      \qquad \forall~s \in [0,+\infty)\,,\\
  \frac{\partial^2 \varphi_i^A}{\partial x_j\partial x_r}(x,0) &= 0 \,.
\end{aligned}
\end{equation}
In~\eqref{flow-map-Grad}, $\nabla \varphi^A(x,s)$ is the
$d\times d$ matrix whose entries are $(\nabla
\varphi^A(x,s))_{ij}=\frac{\partial \varphi^A_j}{\partial x_i}(x,s)$, where $1
\le i,j \le d$. Furthermore we have omitted the $x$-dependence of the coefficients in~\eqref{flow-map-Grad}--\eqref{flow-map-Hess} when they are time independent.

With these preliminaries in the following we prove
Propositions~\ref{prop-1st-varphi}--\ref{prop2}, regarding the first and second-order derivatives of $\Theta^A$ respectively.
\begin{proof}[Proof of Propostion~\ref{prop-1st-varphi}]
Using \eqref{eq:nablaF-Sigma} and the definition~\eqref{phi-gamma} of $\Gamma$, we find
\begin{equation}
  \left((a-A)\nabla^2F\right)(x) = \left((a-A)\nabla\xi\nabla\xi^T\right)(x) = \Gamma(x),   
\quad x \in \Sigma\,.
  \label{eqn-i-a-d2f-gamma}
\end{equation}
  The
  corresponding (matrix) ODE \eqref{flow-map-Grad} becomes 
\begin{equation*}
  \begin{aligned}
    \frac{d}{ds}(\nabla\varphi^A(x,s))^T &= -\Gamma(x)
    (\nabla\varphi^A(x,s))^T, \quad s \ge 0\,,\\
    \nabla\varphi^A(x,0) &= I_d\,,
  \end{aligned}
\end{equation*}
which admits the solution
  \begin{equation*}
  (\nabla\varphi^A(x,s))^T =
\mathrm{e}^{-s\Gamma}= \sum_{i=0}^{+\infty}\frac{(-s\Gamma)^i}{i!}, \quad x
    \in \Sigma\,,~ s \ge 0\,.
\end{equation*}
Using~\eqref{phi-gamma} and a straightforward induction argument it follows that 
\begin{equation}
  i \ge 1: \  \Gamma^i = (a-A)\nabla \xi \Phi^{i-1}\nabla\xi^T, \
  \mbox{and}\quad  i \ge 0: \   \Gamma^i (a-A)\nabla \xi = (a-A)\nabla \xi \Phi^i.
  \label{eqn-relation-power-of-gamma-and-phi}
 \end{equation} 
  Using \eqref{eqn-relation-power-of-gamma-and-phi} and $\nabla\xi^T V = 0$ (recall~\eqref{condition-for-v}) we find
\begin{equation*}
\mathrm{e}^{-s\Gamma} V = V, \qquad \mathrm{e}^{-s\Gamma}(a-A)\nabla \xi = (a-A)\nabla\xi\,\mathrm{e}^{-s\Phi},
\end{equation*}
and therefore we can write in matrix form 
\begin{equation}\label{eq-exp-sGamma}
  \mathrm{e}^{-s\Gamma} 
  = 
  \begin{pmatrix}
    V& (a-A)\nabla\xi \mathrm{e}^{-s\Phi} 
  \end{pmatrix}
  \begin{pmatrix}
    V& (a-A)\nabla \xi 
  \end{pmatrix}^{-1}, \quad \forall~s\ge 0\,.
\end{equation}
Using the definition of $\Phi, \Pi$ in
  \eqref{phi-gamma} and \eqref{pi}, along with  $\nabla\xi^T V = 0$, we can directly verify 
\begin{equation*}
  \begin{aligned}
  \begin{pmatrix}
    V & (a-A)\nabla \xi 
  \end{pmatrix}^{-1}
    = & 
  \begin{pmatrix}
    \Pi^{-1} V^T(a-A)^{-1}\\
    \Phi^{-1} \nabla \xi^T
  \end{pmatrix}.
  \end{aligned}
\end{equation*}
Substituting this expression into~\eqref{eq-exp-sGamma} we find
\begin{equation*}
  (\nabla\varphi^A(x,s))^T = \mathrm{e}^{-s\Gamma}  
  = V\Pi^{-1} V^T(a-A)^{-1} + (a-A)\nabla\xi \mathrm{e}^{-s\Phi}
  \Phi^{-1}\nabla\xi^T\,, \quad s \ge 0\,,
\end{equation*}
while the last expression in \eqref{eqn-1st-varphi} follows using the
definition of $P$ in \eqref{projection-p}.
Since all the eigenvalues of $\Phi$ have positive real parts (recall
Lemma~\ref{lemma-on-matrix-pi-phi}), we can pass the limit $s\rightarrow +\infty$ which gives \eqref{eqn-of-theta-map}.
\end{proof}

When $A=0$ (which corresponds to reversible case) $\nabla\Theta^A$ is symmetric, which is not true in general when $A\neq 0$, as illustrated by the following simple example. 
\begin{example}
    Consider $\xi(x_1,x_2):=\frac{1}{2}(x_1^2+x^2_2-1): \R^2\rightarrow \R$ at
    $x=(1,0)^T$. Choose $a=I_2$ and $A=\left(
    \begin{smallmatrix}
      0 & 1\\
      -1 & 0
    \end{smallmatrix}\right)$. We have 
    $V = (0,1)^T$, $\nabla\xi=(1,0)^T$ and 
    using~\eqref{eqn-of-theta-map} we find 
    \begin{equation*}
      \nabla\Theta^A = 
    \begin{pmatrix}
      0 & -1\\
      0 & 1
    \end{pmatrix}.
    \end{equation*}
\end{example}

Now, we prove Proposition~\ref{prop2} concerning second derivatives of $\Theta^A$.
\begin{proof}[Proof of Proposition~\ref{prop2}]
Using identity \eqref{projection-sum-identity} 
we write
\begin{equation}  \label{sum-of-i1-i2}
  \begin{aligned}
  a_{jr}\frac{\partial^2 \Theta_i^A}{\partial x_j\partial x_r} 
      &= \big[V\Pi^{-1} V^T(a-A)^{-1}\big]_{i\l} 
      a_{jr} \frac{\partial^2 \Theta_\l^A}{\partial x_j\partial x_r}
    +\big[(a-A)\nabla\xi \Phi^{-1}\nabla\xi^T\big]_{i\l} a_{jr}\frac{\partial^2
    \Theta_\l^A}{\partial x_j\partial x_r}\\
  &= P_{i\l} 
      a_{jr} \frac{\partial^2 \Theta_\l^A}{\partial x_j\partial x_r}
    +\big[(a-A)\nabla\xi \Phi^{-1}\nabla\xi^T\big]_{i\l} a_{jr}\frac{\partial^2
    \Theta_\l^A}{\partial x_j\partial x_r}\\
    &=: \mathcal{I}_1 + \mathcal{I}_2\,,
\end{aligned}
\end{equation}
where the repeated indices $j,r,\ell$ are summed over $1$ to $d$. This
  Einstein's summation notation will be used throughout this proof. 
See Remark~\ref{rem:proof-mot} for the motivation behind this particular
  splitting in~\eqref{sum-of-i1-i2}, which plays a crucial role in the forthcoming calculations. 

  First, we compute the term $\mathcal{I}_1$ in \eqref{sum-of-i1-i2}. Using \eqref{eqn-i-a-d2f-gamma}
  and applying the variation of constants formula to the
  ODE~\eqref{flow-map-Hess}, we find, for $1 \le \l, j, r \le d$,
\begin{equation*}
  \begin{aligned}
    \frac{\partial^2 \Theta_\l^A}{\partial x_j\partial x_r} 
      =&
 \lim_{t\rightarrow +\infty} \frac{\partial^2 \varphi_\l^A}{\partial x_j\partial
      x_r}(x,t) \\
      =& -\lim_{t\rightarrow +\infty}\int_0^t
      \big[\mathrm{e}^{-(t-s)\Gamma}\big]_{\l\l'}\, \left[(a- A)_{\l'r'}\frac{\partial^3
      F}{\partial x_{r'}\partial x_{i'}\partial x_{j'}} 
+2 \frac{\partial a_{\l'r'}}{\partial x_{i'}}\frac{\partial^2 F}{\partial
    x_{r'}\partial x_{j'}}\right]
      \frac{\partial\, \varphi_{i'}^A}{\partial x_j} (x,s)
      \frac{\partial\, \varphi_{j'}^A}{\partial x_r} (x,s)\,ds\\
      =& -
      \left[(a- A)_{\l'r'}\frac{\partial^3
      F}{\partial x_{r'}\partial x_{i'}\partial x_{j'}} 
+2 \frac{\partial a_{\l'r'}}{\partial x_{j'}}\frac{\partial^2 F}{\partial
    x_{r'}\partial x_{i'}}\right]
      \lim_{t\rightarrow +\infty}\int_0^t
      \big[\mathrm{e}^{-(t-s)\Gamma}\big]_{\l\l'}\,
      \big[\mathrm{e}^{-s\Gamma}\big]_{i'j} \big[\mathrm{e}^{-s\Gamma}\big]_{j'r}
      \,ds\,.
  \end{aligned}
\end{equation*}
Note that we have switched the indices $i',j'$ in the second term in the sum above to simplify the following calculations. This is allowed since the indices $j,r$ can be interchanged since $\frac{\partial^2 \Theta_\l^A}{\partial x_j\partial x_r} =   \frac{\partial^2 \Theta_\l^A}{\partial x_r\partial x_j}$.   From the definition of $F$ in \eqref{fun-cap-f}, using $\xi(x) = 0$ on $x \in
  \Sigma$, we can compute, for $1 \le i', j', r' \le d$,
    \begin{equation}\label{eq:F-3der}
\frac{\partial^3 F}{\partial x_{r'}\partial x_{i'}\partial x_{j'}} 
  = \frac{\partial\xi_\alpha}{\partial
      x_{r'}}\frac{\partial^2\xi_\alpha}{\partial
      x_{i'}\partial x_{j'}}+\frac{\partial\xi_\alpha}{\partial
      \partial x_{i'}}\frac{\partial^2\xi_\alpha}{\partial x_{j'}\partial x_{r'}}+ 
\frac{\partial\xi_\alpha}{\partial x_{j'}}\frac{\partial^2\xi_\alpha}{\partial
      x_{i'}\partial x_{r'}}\,, \quad x \in \Sigma\,.
\end{equation}
    Computing  $\mathrm{e}^{-(t-s)\Gamma}$ via~\eqref{eqn-1st-varphi} and
    using $P(a-A)\nabla\xi = B\nabla\xi=0$, we find 
    \begin{equation*}
	 P\,\mathrm{e}^{-(t-s)\Gamma}=P\,,
    \end{equation*}
  and therefore using \eqref{eq:F-3der} we find
\begin{align*}  
    \mathcal{I}_1  &= P_{i\l} 
    a_{jr}\frac{\partial^2 \Theta_\l^A}{\partial x_j\partial x_r} \\
      &= - P_{i\l}\, 
       \left[(a- A)_{\l r'}\frac{\partial^3
      F}{\partial x_{r'}\partial x_{i'}\partial x_{j'}} 
+2 \frac{\partial a_{\l r'}}{\partial x_{j'}}\frac{\partial^2 F}{\partial x_{r'}\partial x_{i'}}\right] 
a_{jr} \lim_{t\rightarrow +\infty}\int_0^t
  \big[\mathrm{e}^{-s\Gamma}\big]_{i'j} \big[\mathrm{e}^{-s\Gamma}\big]_{j'r} \,ds\\
      &= - P_{i\l} \bigg[
      (a-A)_{\l r'} \frac{\partial\xi_\alpha}{\partial x_{r'}}
    \frac{\partial^2\xi_\alpha}{\partial x_{i'}\partial x_{j'}}
    + (a-A)_{\l r'} \frac{\partial\xi_\alpha}{\partial x_{i'}}
    \frac{\partial^2\xi_\alpha}{\partial x_{j'}\partial x_{r'}}
     + (a-A)_{\l r'} \frac{\partial\xi_\alpha}{\partial x_{j'}}
    \frac{\partial^2\xi_\alpha}{\partial x_{i'}\partial x_{r'}}\\
    &\qquad\qquad + 2\frac{\partial a_{\l r'}}{\partial x_{j'}}
    \frac{\partial\xi_\alpha}{\partial x_{i'}}
    \frac{\partial\xi_\alpha}{\partial x_{r'}}\bigg]
    a_{jr}
\lim_{t\rightarrow +\infty} \int_0^t  \big[\mathrm{e}^{-s\Gamma}\big]_{i'j}
  \big[\mathrm{e}^{-s\Gamma}\big]_{j'r} \,ds\\
      &= -2 P_{i\l}
    \frac{\partial\xi_\alpha}{\partial x_{i'}}
    \left[(a-A)_{\l r'}\frac{\partial^2\xi_\alpha}{\partial x_{j'}\partial x_{r'}}
    + \frac{\partial a_{\l r'}}{\partial x_{j'}}
    \frac{\partial\xi_\alpha}{\partial x_{r'}}\right]
    a_{jr}
\lim_{t\rightarrow +\infty} \int_0^t  \big[\mathrm{e}^{-s\Gamma}\big]_{i'j}
  \big[\mathrm{e}^{-s\Gamma}\big]_{j'r} \,ds\\
      &= -2 P_{i\l}
    \left[(a-A)_{\l r'}\frac{\partial^2\xi_\alpha}{\partial x_{j'}\partial x_{r'}}
    + \frac{\partial a_{\l r'}}{\partial x_{j'}}
    \frac{\partial\xi_\alpha}{\partial x_{r'}}\right]
    \left[\nabla\xi^T \lim_{t\rightarrow +\infty}\int_0^t  \mathrm{e}^{-s\Gamma} a
    (\mathrm{e}^{-s\Gamma})^{T} \,ds\right]_{\alpha j'}\\
      &= -2 P_{i\l}
    \frac{\partial}{\partial
    x_{j'}}\left((a-A)_{\l r'}\frac{\partial\xi_\alpha}{\partial x_{r'}}
    \right)\,
    \left[\nabla\xi^T \lim_{t\rightarrow +\infty}\int_0^t  \mathrm{e}^{-s\Gamma} a
    (\mathrm{e}^{-s\Gamma})^{T} \,ds\right]_{\alpha j'}\,,
\end{align*}
where index $\alpha$ is summed over $1$ to $k$, and in the fourth equality
we have used $P_{i\l}(a-A)_{\l r'}\frac{\partial \xi_\alpha}{\partial
x_{r'}}=\big[P(a-A)\nabla\xi\big]_{i\alpha}=0$.
Using Lemma~\ref{lemma-identity} in Appendix~\ref{app:lemma-identity-proof} we find
\begin{align}  
    \mathcal{I}_1 &= -P_{i\l}
    \frac{\partial}{\partial x_{j'}}\left((a-A)_{\l r'}\frac{\partial\xi_\alpha}{\partial x_{r'}}
    \right)
    \left[\Phi^{-1}\nabla\xi^T(a-A)\right]_{\alpha j'}\notag\\
&= -P_{i\l} \frac{\partial}{\partial x_{j'}}\left[(a-A)_{\l r'}\frac{\partial\xi_\alpha}{\partial x_{r'}}
    \left(\Phi^{-1}\nabla\xi^T(a-A)\right)_{\alpha j'}\right] \notag\\
    &\quad + \left[P_{i\l}
    (a-A)_{\l r'}\frac{\partial\xi_\alpha}{\partial x_{r'}}\right]
    \frac{\partial}{\partial x_{j'}} \big[\Phi^{-1}\nabla\xi^T(a-A)\big]_{\alpha j'} \label{eqn-of-i1}\\
&= -P_{i\l}
    \frac{\partial}{\partial x_{j'}}\Big[(a-A)\nabla\xi
    \Phi^{-1}\nabla\xi^T(a-A)\Big]_{\l j'} \notag\\
&= -P_{i\l}
    \frac{\partial}{\partial x_{j'}}\left[(I_d-P)(a-A) \right]_{\l j'}\notag \\
&= P_{i\l} \frac{\partial B_{\l j'}}{\partial x_{j'}}
    -P_{i\l}
    \frac{\partial a_{\l j'}}{\partial x_{j'}}  \,, \notag
\end{align}
where we have used $P(a-A) \nabla \xi =0$ and~\eqref{projection-p} to arrive at the third and the fourth equality respectively.

Next, we compute $\mathcal{I}_2$ in \eqref{sum-of-i1-i2}.
Differentiating the identity $\xi(\Theta^A(\cdot)) \equiv 0$ on
$\Sigma^{(\delta)}$ twice and using $\Theta^A(x) = x$ for $x \in \Sigma$, we obtain
\begin{equation}  \label{2nd-of-theta-by-1st}
  \frac{\partial\xi_\alpha}{\partial x_\l} \frac{\partial^2\Theta^A_\l}{\partial
  x_j\partial x_{j'}} = - \frac{\partial^2\xi_\alpha}{\partial x_\l\partial x_{\l'}} 
\frac{\partial\Theta^A_\l}{\partial x_j} 
  \frac{\partial\Theta^A_{\l'}}{\partial x_{j'}} \,, \quad \mbox{on}~\Sigma\,,
\end{equation}
  for $1 \le \alpha \le k$ and $1 \le j, j' \le d$.
Using \eqref{2nd-of-theta-by-1st} and the explicit expression of
$\nabla\Theta^A$~\eqref{eqn-of-theta-map} for $x \in \Sigma$, we find
\begin{align} 
    \mathcal{I}_2 &=  \big[(a-A)\nabla\xi\Phi^{-1} \nabla\xi^T\big]_{i\l} 
a_{jj'} \frac{\partial^2\Theta^A_\l}{\partial x_j\partial x_{j'}}\notag \\
    &=     \big[(a-A)\nabla\xi\Phi^{-1}\big]_{i\alpha} 
a_{jj'}\frac{\partial\xi_\alpha}{\partial x_\l} \frac{\partial^2\Theta^A_\l}{\partial
    x_j\partial x_{j'}} \notag\\
    &= - \big[(a-A)\nabla\xi\Phi^{-1}\big]_{i\alpha} 
\frac{\partial^2\xi_\alpha}{\partial x_\l\partial x_{\l'}} 
\frac{\partial\Theta^A_\l}{\partial x_j} 
  \frac{\partial\Theta^A_{\l'}}{\partial x_{j'}} a_{jj'} \notag\\
&= -\big[(a-A)\nabla\xi\Phi^{-1}\big]_{i\alpha} 
 \frac{\partial^2\xi_\alpha}{\partial x_\l\partial x_{\l'}} 
    (P a P^T)_{\l\l'} \label{eqn-of-i2} \\
    &= -\frac{1}{2}\big[(a-A)\nabla\xi\Phi^{-1}\big]_{i\alpha} 
 \frac{\partial^2\xi_\alpha}{\partial x_\l\partial x_{\l'}} 
    (B + B^T)_{\l\l'} \notag\\
    &= -\big[(a-A)\nabla\xi\Phi^{-1}\big]_{i\alpha} 
 \frac{\partial^2\xi_\alpha}{\partial x_\l\partial x_{\l'}} 
    B_{\l\l'}\notag \\
&= - \big[(a-A)\nabla\xi\Phi^{-1}\big]_{i\alpha} 
    \frac{\partial \big(\nabla\xi^TB\big)_{\alpha \l'}}{\partial x_{\l'}} 
    + \big[(a-A)\nabla\xi\Phi^{-1}\big]_{i\alpha} 
    \frac{\partial\xi_\alpha}{\partial x_\l} 
    \frac{\partial B_{\l\l'}}{\partial x_{\l'}} \notag\\
&=  \big[(a-A)\nabla\xi\Phi^{-1}\nabla\xi^T\big]_{i\l} 
    \frac{\partial B_{\l\l'}}{\partial x_{\l'}} \,,  \notag
\end{align}
where we have used~\eqref{eqn-of-theta-map} to arrive at the fourth equality,
the relation~\eqref{relation-p-b} to arrive at the fifth equality, and $\nabla\xi^T B=0$ to arrive at the final equality. 

Finally, summing up~\eqref{eqn-of-i1},~\eqref{eqn-of-i2} and using 
the definition of $P$ in \eqref{projection-p}, we find
\begin{align*}
  a_{jr}\frac{\partial^2 \Theta_i^A}{\partial x_j\partial x_r} 
  &= \mathcal{I}_1+\mathcal{I}_2 \\
  &= P_{i\l} \frac{\partial B_{\l j'}}{\partial x_{j'}}
    -P_{i\l} \frac{\partial a_{\l j'}}{\partial x_{j'}} 
    + \big[(a-A)\nabla\xi\Phi^{-1}\nabla\xi^T\big]_{i\l} 
    \frac{\partial B_{\l\l'}}{\partial x_{\l'}} \\
    &= \frac{\partial B_{ij}}{\partial x_{j}} - 
    P_{i\l} \frac{\partial a_{\l j'}}{\partial x_{j'}} \,.
\end{align*}
\end{proof}

In the following remark we discuss the proof techniques used to prove Proposition~\ref{prop2}. 
\begin{remark}\label{rem:proof-mot} 
  The starting point of the proof above is the splitting~\eqref{sum-of-i1-i2}, which we recall
\begin{equation}\label{eq:split-eq}  
  \sum_{j,r=1}^d a_{jr}\frac{\partial^2 \Theta_i^A}{\partial x_j\partial x_r} 
  = \sum_{\l, j, r=1}^d P_{i\l} 
      a_{jr} \frac{\partial^2 \Theta_{\l}^A}{\partial x_j\partial x_r}
    +\sum_{\l,j,r=1}^d \big[(a-A)\nabla\xi \Phi^{-1}\nabla\xi^T\big]_{i\l} a_{j r}\frac{\partial^2
    \Theta_{\l}^A}{\partial x_j\partial x_r}
    =: \mathcal{I}_1 + \mathcal{I}_2\,.
\end{equation}
This splitting is motivated by the identity, for $1 \le \alpha\le k$ and $1 \le j, j'\le d$,
\begin{equation*}
  \sum_{\l=1}^d 
  \frac{\partial\xi_\alpha}{\partial x_\l} \frac{\partial^2\Theta^A_\l}{\partial
  x_j\partial x_{j'}} = - 
  \sum_{\l, \l'=1}^d 
  \frac{\partial^2\xi_\alpha}{\partial x_\l\partial x_{\l'}} 
\frac{\partial\Theta^A_\l}{\partial x_j} 
  \frac{\partial\Theta^A_{\l'}}{\partial x_{j'}} \,, \quad \mbox{on}~\Sigma\,,
\end{equation*}
  which follows by differentiating $\xi(\Theta^A(\cdot)) \equiv 0$ on
  $\Sigma^{(\delta)}$
twice and using $\Theta^A(x) = x$ for $x \in \Sigma$. Using this identity we can rewrite the second term $\mathcal I_2$ (with second-order
  derivatives of $\Theta^A$) in~\eqref{eq:split-eq} as a product of first-order derivatives of $\Theta^A$
  for which we have derived explicit expressions in
  Proposition~\ref{prop-1st-varphi}. This considerably simplifies the analysis, 
  since we need to study the ODE~\eqref{flow-map-Hess} only for
  the first term $\mathcal I_1$. 
\end{remark}

We conclude this section with the proof of Lemma~\ref{lemma-rescaled}.
\begin{proof}[Proof of Lemma~\ref{lemma-rescaled}]
  Similar to \eqref{estimate-of-Lyapunov-f-kappa0}, using~\eqref{estimate-gradf-by-f} 
  and the ODE~\eqref{phi-map-a-kappa}, we find 
  \begin{equation}
    \begin{aligned}
    \frac{dF(\varphi^{A,\kappa}(x,s))}{ds} &= - \frac{1}{2} \Big(\nabla F^T (a-A)
    \nabla |\xi|^{2-\kappa}\Big)(\varphi^{A,\kappa}(x,s)) \\
    & = - \frac{1}{2} \Big(\nabla F^T (a-A)
    \nabla \big[(2F)^{\frac{2-\kappa}{2}}\big]\Big)(\varphi^{A,\kappa}(x,s)) \\
    &=- (2-\kappa) 2^{-1-\frac{\kappa}{2}} 
    \big(F^{-\frac{\kappa}{2}} \nabla F^T a \nabla F\big) (\varphi^{A,\kappa}(x,s))\\
    &\le - (2-\kappa)2^{-\frac{\kappa}{2}} 
    c_0 c_1 F^{1-\frac{\kappa}{2}}(\varphi^{A,\kappa}(x,s))\,,
    \end{aligned}
    \label{estimate-of-Lyapunov-f}
  \end{equation}
  where the first equality follows from \eqref{phi-map-a-kappa}, the second
  equality follows from the definition \eqref{fun-cap-f-intro} of $F$ and  
  the third equality follows from the chain rule and the antisymmetry of $A$. 
  For $\kappa>0$ and $x \in \Sigma^{(\delta)}$, after integrating the inequality above, we arrive at 
  $|\xi(\varphi^{A,\kappa}(x,s))|^\kappa \le
  |\xi(x)|^\kappa - 2^{-(1+\frac{\kappa}{2})}\kappa (2-\kappa) c_0c_1 s$, 
     for any $s \in [0, s_f]$, where $s_f =
    \frac{2^{1+\frac{\kappa}{2}}|\xi(x)|^\kappa}{\kappa (2-\kappa) c_0c_1}$.
   This implies that $\varphi^{A,\kappa}(x,\cdot)$ reaches the limiting state $\Theta^A(x)\in \Sigma$ within finite time $s_f$. To conclude, it is sufficient to observe that, by the definition of
  the neighbhourhood $\Sigma^{(\delta)}$ \eqref{sigma-delta}, we have $|\xi(x)|\le
  \delta$ for any starting state $x \in \Sigma^{(\delta)}$, and therefore $s_f \le \bar{s}:=\frac{2^{1+\frac{\kappa}{2}}\delta^\kappa}{\kappa (2-\kappa) c_0c_1}$. 

\end{proof}
\section{Comparison between non-reversible and reversible schemes}
\label{sec-compare}
In this section, we prove Proposition~\ref{prop-cmp-poincare-and-asym-variance}, which compares the non-reversible case ($A\neq 0$) to the reversible case ($A=0$). We make use of the following subspace of $L^2(\Sigma,\mu)$
\begin{equation}\label{def-L2-0mean}
  \mathcal H:=\Big\{g\in L^2(\Sigma,\mu):  \|g\|_{L^2(\Sigma,\mu)}=1,\,
  \E_{\cmu}[g] = 0\Big\}.
\end{equation}
First we prove the following useful result on the Dirichlet forms associated to the generators $\mathcal L$~\eqref{generator-l}, $\mathcal S$~\eqref{sym-antisym-of-l} and $\mathcal L_0$~\eqref{generator-l-reversible}.

\begin{lemma}  \label{lemma-compare-of-dirichlet-energy}
  For any $g\in\mathcal H\cap C^2(\Sigma)$ we have
  \begin{equation*}
  \int_{\Sigma} g(-\mathcal{L})g\,d\cmu
    = 
  \int_{\Sigma} g(-\mathcal{S})g\,d\cmu
    \ge 
  \int_{\Sigma} g(-\mathcal{L}_0)g\,d\cmu\,.
  \end{equation*}
\end{lemma}
\begin{proof}
Using~\eqref{integration-by-parts-f}, for any $g\in\mathcal H\cap C^2(\Sigma)$ we find
\begin{equation}
  \begin{aligned}
     \int_{\Sigma} g(-\mathcal{L})g\,d\cmu
    = \int_{\Sigma} g (-\mathcal{S})g\,d\cmu 
    = \frac{1}{\beta} \int_{\Sigma} (B^{sym} \nabla g)\cdot \nabla g\,d\cmu
    = \frac{1}{2\beta} \int_{\Sigma} (V^T \nabla g)^T
     (\Pi^{-1} + \Pi^{-T})V^T\nabla g\,d\cmu\,,
  \end{aligned}
  \label{eqn-g-l-g-energy}
\end{equation}
  where the final equality follows from~\eqref{matrix-p-b} and the definition of $B^{\mathrm{sym}}$~\eqref{symmetric-parts-of-b}.
  In particular, for $A=0$, \eqref{eqn-g-l-g-energy} implies
  \begin{equation}
\int_{\Sigma} g(-\mathcal{L}_0)g\,d\cmu = \frac{1}{\beta} \int_{\Sigma}
    (V^T \nabla g)^T \Pi_0^{-1} V^T\nabla g\,d\cmu\, ,
    \label{eqn-g-l-g-energy-0}
  \end{equation}
  where  $\Pi_0 =V^Ta^{-1}V=\Pi_0^T$ (recall~\eqref{pi}). 

  Comparing~\eqref{eqn-g-l-g-energy} and~\eqref{eqn-g-l-g-energy-0}, it
  suffices to prove that
  \begin{equation}    \label{inequality-to-prove-poincare}
    \frac{1}{2}(\Pi^{-1} + \Pi^{-T}) \succeq \Pi_0^{-1}\,,
\end{equation}
i.e.\ $A-B$ is positive semi-definite.

  Let us define $Q_1 := \Pi^{sym} = \frac{1}{2} (\Pi + \Pi^T)$ 
  and $Q_2 := \Pi^{asym} = \frac{1}{2} (\Pi - \Pi^T)$. Since $\Pi =
  V^T(a-A)^{-1}V$ (see \eqref{pi}), we find
\begin{equation}  \label{q1-q2}
Q_1 = V^T(a-A)^{-1}a (a+A)^{-1} V\,,\quad Q_2 = V^T(a-A)^{-1} A (a+A)^{-1} V \,.
\end{equation}
It is easy to see that $Q_1\in \mathbb{R}^{(d-k)\times (d-k)}$ is positive
  definite (invertible) since $(a-A)$ is invertible and $V$ has linearly independent columns. For the term on the left hand side
  of \eqref{inequality-to-prove-poincare}, using $Q_1^T=Q_1$ and $Q_2^T=-Q_2$, we compute
\begin{align}
     \frac{1}{2} (\Pi^{-1} + \Pi^{-T}) 
    &=  \frac{1}{2} \big[(Q_1 + Q_2)^{-1} + (Q_1 - Q_2)^{-1}\big]\notag\\
    &=  \frac{1}{2} Q_1^{-\frac{1}{2}}\Big[\Big(I_{d-k} +Q_1^{-\frac{1}{2}}
    Q_2Q_1^{-\frac{1}{2}}\Big)^{-1} + \Big(I_{d-k} -Q_1^{-\frac{1}{2}}
    Q_2Q_1^{-\frac{1}{2}}\Big)^{-1}\Big]Q_1^{-\frac{1}{2}}\notag\\
    &=  \frac{1}{2} Q_1^{-\frac{1}{2}}\Big[\Big(I_{d-k} +Q_1^{-\frac{1}{2}}
    Q_2Q_1^{-\frac{1}{2}}\Big)^{-1}\Big(I_{d-k} -Q_1^{-\frac{1}{2}}
    Q_2Q_1^{-\frac{1}{2}}\Big)^{-1} \Big(I_{d-k} -Q_1^{-\frac{1}{2}}
    Q_2Q_1^{-\frac{1}{2}}\Big)\notag\\
    & \ \ \    +\Big(I_{d-k}+Q_1^{-\frac{1}{2}} Q_2Q_1^{-\frac{1}{2}}\Big) \Big(I_{d-k}+Q_1^{-\frac{1}{2}}
    Q_2Q_1^{-\frac{1}{2}}\Big)^{-1}\Big(I_{d-k}-Q_1^{-\frac{1}{2}} Q_2Q_1^{-\frac{1}{2}}\Big)^{-1} 
     \Big]Q_1^{-\frac{1}{2}} \notag \\
    &=  Q_1^{-\frac{1}{2}}\Big(I_{d-k} +Q_1^{-\frac{1}{2}}
    Q_2Q_1^{-\frac{1}{2}}\Big)^{-1}\Big(I_{d-k} -Q_1^{-\frac{1}{2}}
    Q_2Q_1^{-\frac{1}{2}}\Big)^{-1}Q_1^{-\frac{1}{2}}\notag\\
    &=  \Big[Q_1^{\frac{1}{2}} \Big(I_{d-k} -Q_1^{-\frac{1}{2}}
    Q_2Q_1^{-\frac{1}{2}}\Big)\Big(I_{d-k} +Q_1^{-\frac{1}{2}}
    Q_2Q_1^{-\frac{1}{2}}\Big)Q_1^{\frac{1}{2}}\Big]^{-1}\notag\\
    &=  (Q_1 - Q_2 Q_1^{-1}Q_2)^{-1}\notag\\
    &=  (Q_1 + Q_2^T Q_1^{-1}Q_2)^{-1}\,.\label{pi-sym-by-q1-q2}
\end{align}
  The above calculations are fairly standard, see for
  instance the proof of~\cite[Lemma $3$]{DuncanLelievrePavliotis16}.
Therefore to arrive at~\eqref{inequality-to-prove-poincare}, it is sufficient  to compare $Q_1 + Q_2^T Q_1^{-1}Q_2$ with $\Pi_0$.

  To continue, we define $R_1 = a^{\frac{1}{2}}(a+A)^{-1}V\in \mathbb{R}^{d\times (d-k)}$ and
  $R_2 = a^{-\frac{1}{2}}A (a+A)^{-1}V\in \mathbb{R}^{d\times (d-k)}$. Using
  \eqref{q1-q2} and the identity 
\begin{equation*}
   a = (a-A) a^{-1} (a+A) + Aa^{-1}A
\end{equation*}
along with $A=-A^T$ we find
\begin{equation*}
  \begin{aligned}
  Q_1 & = R_1^TR_1\\
&= V^T(a-A)^{-1} a (a+A)^{-1} V\\
    &= V^T(a-A)^{-1} \big[(a-A) a^{-1} (a+A) + Aa^{-1}A\big] (a+A)^{-1} V\\
    &= V^T a^{-1} V + V^T(a-A)^{-1} Aa^{-1}A (a+A)^{-1} V\\
    &= \Pi_0 - V^T(a-A)^{-1} A^Ta^{-1}A (a+A)^{-1} V\\
    &= \Pi_0 - R_2^TR_2\,,
  \end{aligned}
\end{equation*}
and $Q_2 = R_1^TR_2$. For the right hand side of~\eqref{pi-sym-by-q1-q2} we find
\begin{equation}
     Q_1 + Q_2^TQ_1^{-1}Q_2 
     =  \Pi_0 - R_2^TR_2 + R_2^TR_1(R_1^TR_1)^{-1} R_1^TR_2
     =  \Pi_0 + R_2^T\big[R_1(R_1^TR_1)^{-1} R_1^T - I_d\big] R_2\,.
  \label{q1-q2-inv-q2-sum}
\end{equation}

In what follows we will show that 
\begin{equation}  \label{cmp-r1-inv-r1-le-i}
R_1(R_1^TR_1)^{-1} R_1^T \preceq I_d\,,
\end{equation}
Using~\eqref{pi-sym-by-q1-q2}-\eqref{cmp-r1-inv-r1-le-i}, we find 
$\left(\frac{1}{2} (\Pi^{-1} + \Pi^{-T})\right)^{-1} \preceq \Pi_0$. 
Since $Q_1$ is positive definite,~\eqref{pi-sym-by-q1-q2} implies that $\frac{1}{2} (\Pi^{-1} + \Pi^{-T})$ is positive definite as well. Therefore $\left(\frac{1}{2} (\Pi^{-1} + \Pi^{-T})\right)^{-1} \preceq \Pi_0$ implies that $\frac{1}{2} (\Pi^{-1} + \Pi^{-T}) \succeq \Pi_0^{-1}$, which is the required result (see~\eqref{inequality-to-prove-poincare}).

Now we prove~\eqref{cmp-r1-inv-r1-le-i}. Define $R^{\perp}_1 :=
a^{-\frac{1}{2}} (a-A)\nabla \xi\in \mathbb{R}^{d\times k}$. Since
$R_1^TR_1^{\perp} = V^T \nabla\xi =0$, the columns of $R_1,R^{\perp}_1$ are
linearly independent vectors that span $\R^d$ (recall the definition of $V$).
Therefore, any $u \in \mathbb{R}^d$ can be written as $u = R_1 v_1+ R^{\perp}_1 v_2$, for some
$v_1 \in \mathbb{R}^{d-k}$ and $v_2 \in \mathbb{R}^{k}$.
We have $R_1^T u = (R_1^T R_1)v_1$ and $|u|^2 = |R_1v_1|^2 + |R_1^{\perp}v_2|^2$. Using these facts,  for any $u\in \mathbb{R}^{d}$ we compute
\begin{equation*}
  \begin{aligned}
   u^T \big[R_1(R_1^TR_1)^{-1}R_1^T - I_d\big] u
  &=  (R_1^Tu)^T (R_1^TR_1)^{-1}R_1^Tu - |u|^2\\
  &=  v_1^T R_1^TR_1 (R_1^TR_1)^{-1}R_1^TR_1 v_1 - |u|^2\\
    &= |R_1v_1|^2 - \big(|R_1v_1|^2 + |R^{\perp}_1v_2|^2\big) \\
    &= - |R^{\perp}_1v_2|^2 \\
    &\le 0\,,  
  \end{aligned}
\end{equation*}
which implies \eqref{cmp-r1-inv-r1-le-i}. 
\end{proof}

We are ready to prove Proposition~\ref{prop-cmp-poincare-and-asym-variance}.
\begin{proof}[Proof of Proposition~\ref{prop-cmp-poincare-and-asym-variance}]
\noindent\ref{it:cmp-poincare-constants}
  Note that the Poincar{\'e} constants in~\eqref{poincare-inequality-mu1} and~\eqref{poincare-inequality-mu1-reversible} can be characterised as 
\begin{equation}  \label{def-spectral-gap}
  K = \inf_{g\in\mathcal H\cap C^2(\Sigma)} \int_{\Sigma} g(-\mathcal{L})g\, d\cmu\,,\qquad K_0 = \inf_{g\in\mathcal H\cap C^2(\Sigma)} \int_{\Sigma} g(-\mathcal{L}_0)g\, d\cmu\,, 
\end{equation}
respectively, where $\mathcal H$ is defined in~\eqref{def-L2-0mean}. 
  Lemma~\ref{lemma-compare-of-dirichlet-energy} and~\eqref{def-spectral-gap}
  immediately imply that $K \ge K_0$.

\noindent\ref{it:cmp-asymptotic-variance}
Assume without loss of generality that $f\in\mathcal H$. Using the definition~\eqref{eqn-chi} and  applying the integration by parts formula~\eqref{integration-by-parts-f} we have 
\begin{equation}
  \chi_{f}^2 =
  2\int_{\Sigma} (-\mathcal{L}\psi)\psi\,d\cmu
  = 2\int_{\Sigma} f\, (-\mathcal{L})^{-1}f\,d\cmu
  = 2\int_{\Sigma} f\, [(-\mathcal{L})^{-1}]^{\textrm{sym}} f\,d\cmu\,,
  \label{eqn-chi-by-l-inv}
\end{equation}
  where $(-\mathcal{L})^{-1}$ denotes the operator inverse of $-\mathcal{L}$,
 $[(-\mathcal{L})^{-1}]^{\mathrm{sym}}$ is the symmetric part of
  $(-\mathcal{L})^{-1}$, and we have used that $\psi$ is the solution to the
  Poisson equation~\eqref{eqn-poisson-equation} along with $\bar f=
  \E_{\cmu}[f] = 0$ since $f\in\mathcal H$. The invertibility of $\mathcal L$
  follows by standard arguments as in~\cite{DuncanLelievrePavliotis16}.
Similarly, for the asymptotic variance \eqref{eqn-chi-reversible}
  corresponding to $A=0$, we have   
\begin{equation}
  \chi_{f,0}^2 = 2\int_{\Sigma} (-\mathcal{L}_0\psi_0)\psi_0\,d\cmu = 
 2\int_{\Sigma} f\, (-\mathcal{L}_0)^{-1}\,f\,d\cmu\,,
  \label{eqn-chi-by-l-inv-0}
\end{equation}
where $\psi_0$ is the solution to the Poisson equation
  \eqref{eqn-poisson-equation-reversible}, $\mathcal{L}_0$ is defined in
  \eqref{generator-l-reversible}, 
  and $(-\mathcal{L}_0)^{-1}$ is self-adjoint. 
  Using the decomposition \eqref{generator-l-decomp}, and $-\mathcal L^*=-\mathcal S^*-\mathcal A^*=-\mathcal S+\mathcal A$ we find 
\begin{align}
  [(-\mathcal{L})^{-1}]^{\textrm{sym}} 
= \frac{1}{2}[(-\mathcal{L})^{-1} + (-\mathcal{L}^*)^{-1}]
=
  \frac{1}{2}\big[(-\mathcal{S}-\mathcal{A})^{-1} +
  (-\mathcal{S}+\mathcal{A})^{-1}\big]
  = \big[-\mathcal{S} + \mathcal{A}^*(-\mathcal{S})^{-1}\mathcal{A}\big]^{-1}\,,
\end{align}
where $\mathcal A^*$ is the adjoint operator in $L^2(\Sigma,\mu)$ (recall
  Proposition~\ref{prop-generator-l}). Here the final equality above can be computed using a similar calculation as \eqref{pi-sym-by-q1-q2}.
  Applying Lemma~\ref{lemma-compare-of-dirichlet-energy}, we obtain 
  \begin{equation*}
-\mathcal{S} + \mathcal{A}^*(-\mathcal{S})^{-1}\mathcal{A} \succeq -\mathcal{S} \succeq -\mathcal{L}_0\,,
  \end{equation*}
where $\succeq$ denotes the Loewner ordering between self-adjoint operators. Therefore we find 
  \begin{equation}
    [(-\mathcal{L})^{-1}]^{\textrm{sym}} \preceq (-\mathcal{L}_0)^{-1}\,.
    \label{eqn-l-inv-and-l0-inv}
  \end{equation}
  The conclusion is obtained after combining \eqref{eqn-l-inv-and-l0-inv} with
  \eqref{eqn-chi-by-l-inv}--\eqref{eqn-chi-by-l-inv-0}.
\end{proof}

\section{Numerical example}\label{sec-numerical}

\begin{figure}[b!]
\includegraphics[width=0.8\textwidth]{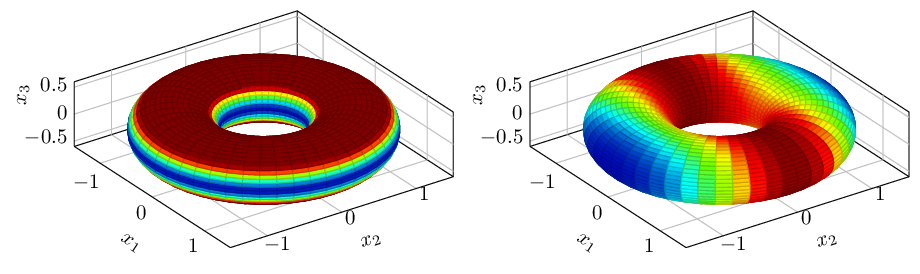}
\centering
  \caption{Left: potential profile $U_1$ in the first test. Right: potential
  profile $U_2$ in the second test. 
  There are two regions where the value of $U_2$ is small. In both plots, blue and red colors
  correspond to small and large values of the potentials respectively.
  \label{fig-ex1-torus-potential}}
\end{figure}
\begin{figure}[b!]
\centering
\begin{subfigure}{0.25\textwidth}
\includegraphics[width=1.0\textwidth]{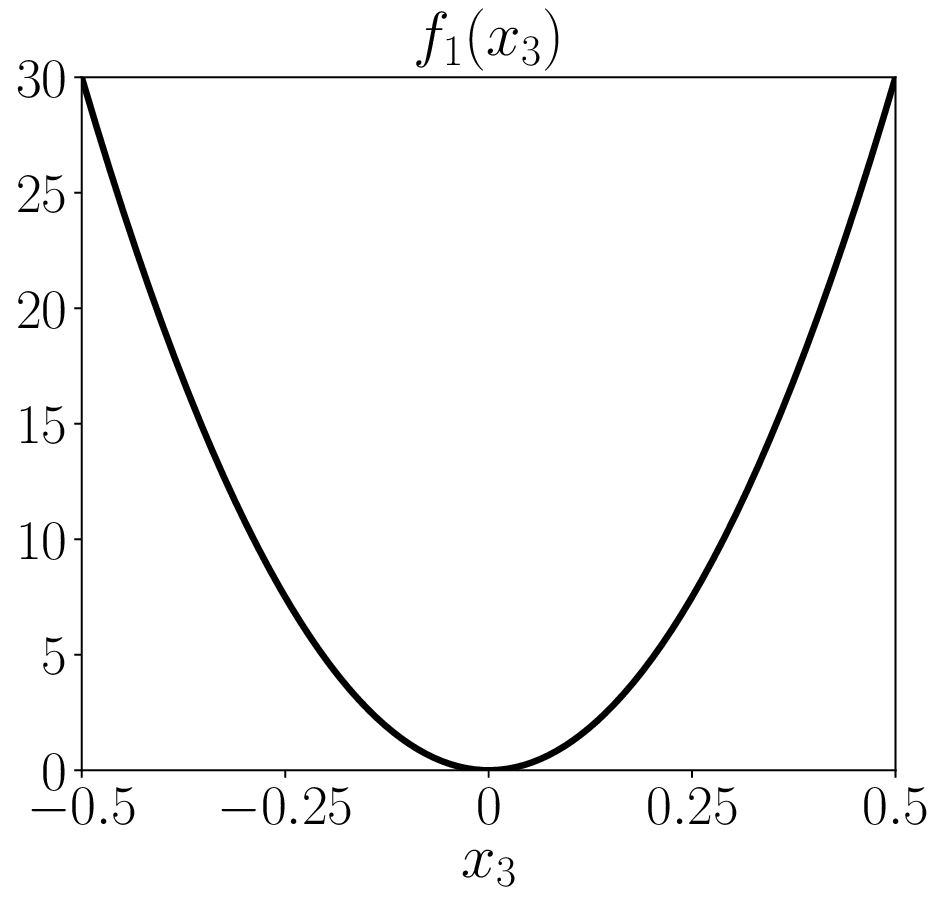}
\end{subfigure}
\begin{subfigure}{0.24\textwidth}
\includegraphics[width=1.0\textwidth]{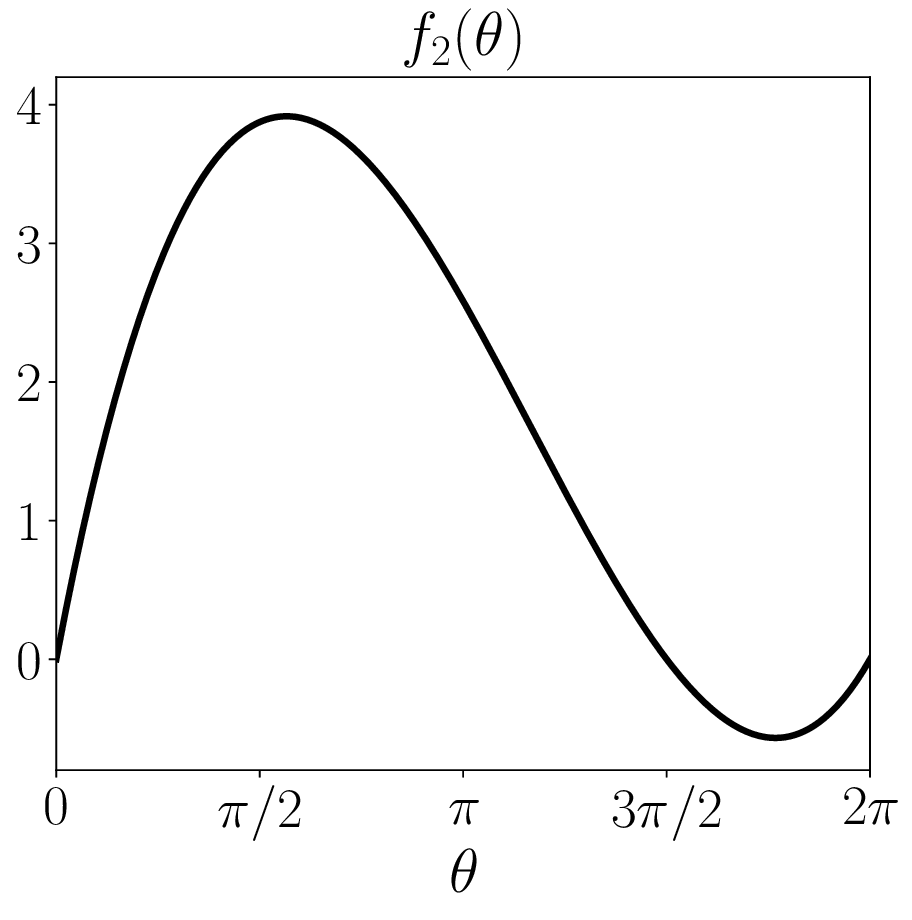}
\end{subfigure}
  \caption{Profiles of functions $f_1$ and $f_2$ used in the first and second tests respectively.
  \label{fun-f-plot}}
\end{figure}
As an illustrative example, we consider the sampling on a two-dimensional
torus $\Sigma$ as a submanifold of $\mathbb{R}^3$~\cite{hmc-submanifold-tony,multiple-projection-mcmc-submanifolds}.
Specifically, we define $\Sigma$ as the zero level set of the polynomial  
\begin{equation}
  \xi(x) = \left(R^2-r^2 + x_1^2+x_2^2+x_3^2\right)^2 - 4R^2\left(x_1^2 +
  x_2^2\right)\,,\quad x=(x_1,x_2,x_3)^T\in \mathbb{R}^3\,,
  \label{ex1-torus-xi}
\end{equation}
for some $0 < r < R$, i.e.\ $\Sigma=\big\{x\in \mathbb{R}^3\,|\,\xi(x) = 0\big\}$.
 Below we will use the following parametrisation of $\Sigma$,
\begin{align}
  x_1 = (R+r\cos \phi) \cos\theta, \quad
  x_2 = (R+r\cos \phi) \sin\theta, \quad
  x_3 = r \sin\phi,
  \label{ex1-polar}
\end{align}
where $(\phi, \theta) \in [0, 2\pi)^2$. In particular, it
can be verified that the normalised surface measure of $\Sigma$ and the norm
of gradient $|\nabla\xi|$ in variables $\theta,\phi$ are  given by
\begin{align}
  \nu_\Sigma(d\phi\,d\theta) = \frac{1}{(2\pi)^2} \left(1 + \frac{r}{R}\cos\phi\right)\,d\phi\,d\theta\,,
  \label{ex1-sigma}
\end{align}
and $|\nabla\xi| = 8R^2r(1+\frac{r}{R}\cos\phi)$.
As a result, the probability measure $\mu$~\eqref{mu1-intro} with potential $U$ is 
\begin{align}
  \label{ex1-mu}
  \cmu(d\phi\,d\theta) = \frac{1}{Z} \mathrm{e}^{-\beta U} d\phi\,d\theta\,,
\end{align}
where $Z$ is the normalisation constant.
In the numerical experiment below, we fix $R=1.0$, $r=0.5$ in \eqref{ex1-sigma}
and study the scheme~\eqref{scheme-non-reversible-numerical} (i.e.\ the numerical version of the scheme~\eqref{scheme-non-reversible}) on two different tests.
Also, as discussed in Remark~\ref{rmk-on-choice-of-noise}, we will ignore the
boundedness assumption on the random variables $\bm{\eta}^{(\ell)}$ in \eqref{conditions-on-eta} and will simply use independent and identically distributed standard Gaussian random variables.

\begin{figure}[t!]
\centering
\begin{subfigure}{0.40\textwidth}
\includegraphics[width=1.0\textwidth]{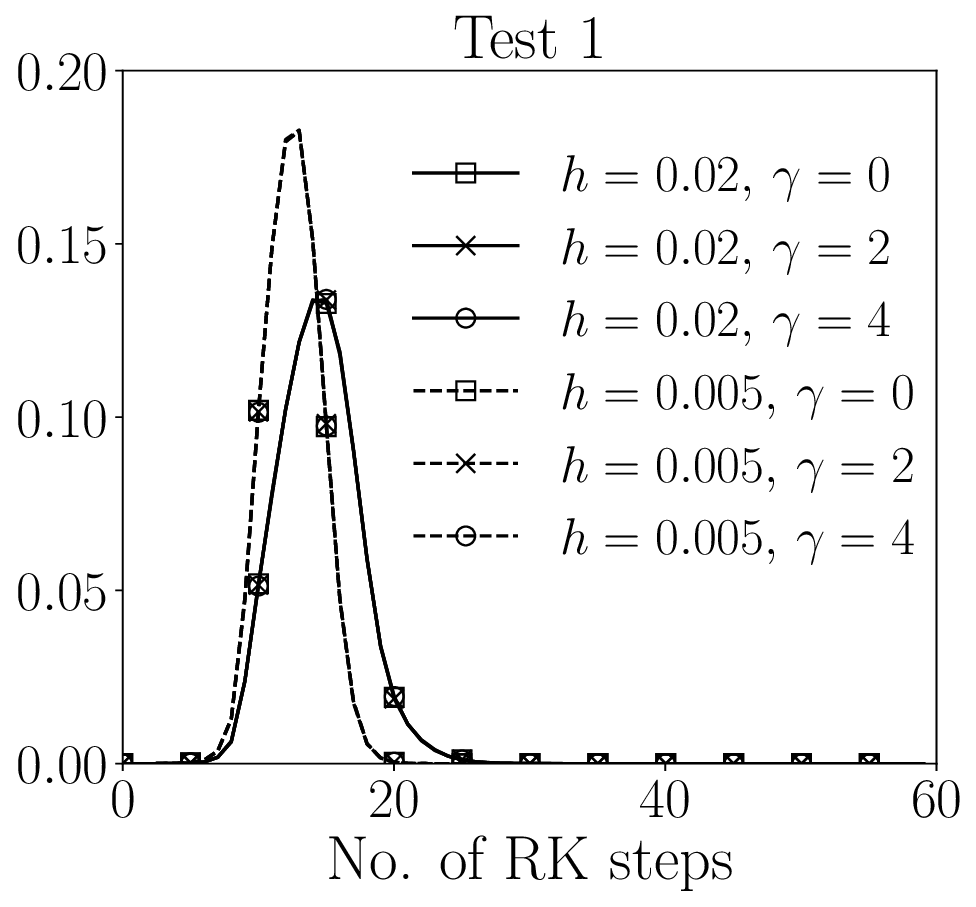}
\end{subfigure}
\begin{subfigure}{0.40\textwidth}
\includegraphics[width=1.0\textwidth]{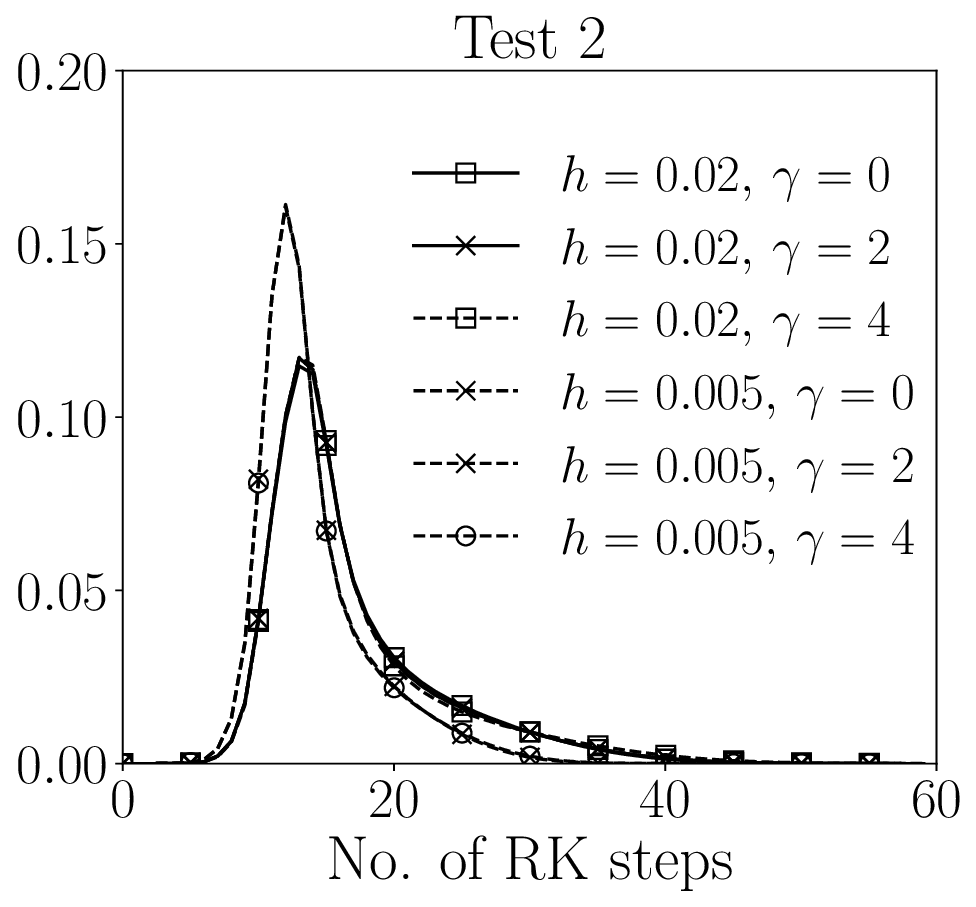}
\end{subfigure}
  \caption{Distributions of numbers of Runge-Kutta steps required in the
  first test~\eqref{ex:first-test} (left panel) and in the second test~\eqref{ex:sec-test} (right panel) for different
  $h$ and $\gamma$ (see \eqref{mat-a-bar}). The step-sizes in the Runge-Kutta
  method are $\Delta t=0.005, 0.01$ in the first and second test respectively. In each case, the distribution is calculated based on the sampling of states in one of $10$ Monte Carlo runs.  \label{dist-plot-of-RK-steps}}
\end{figure}

\begin{figure}[t!]
\centering
\begin{subfigure}{0.32\textwidth}
\includegraphics[width=1.0\textwidth]{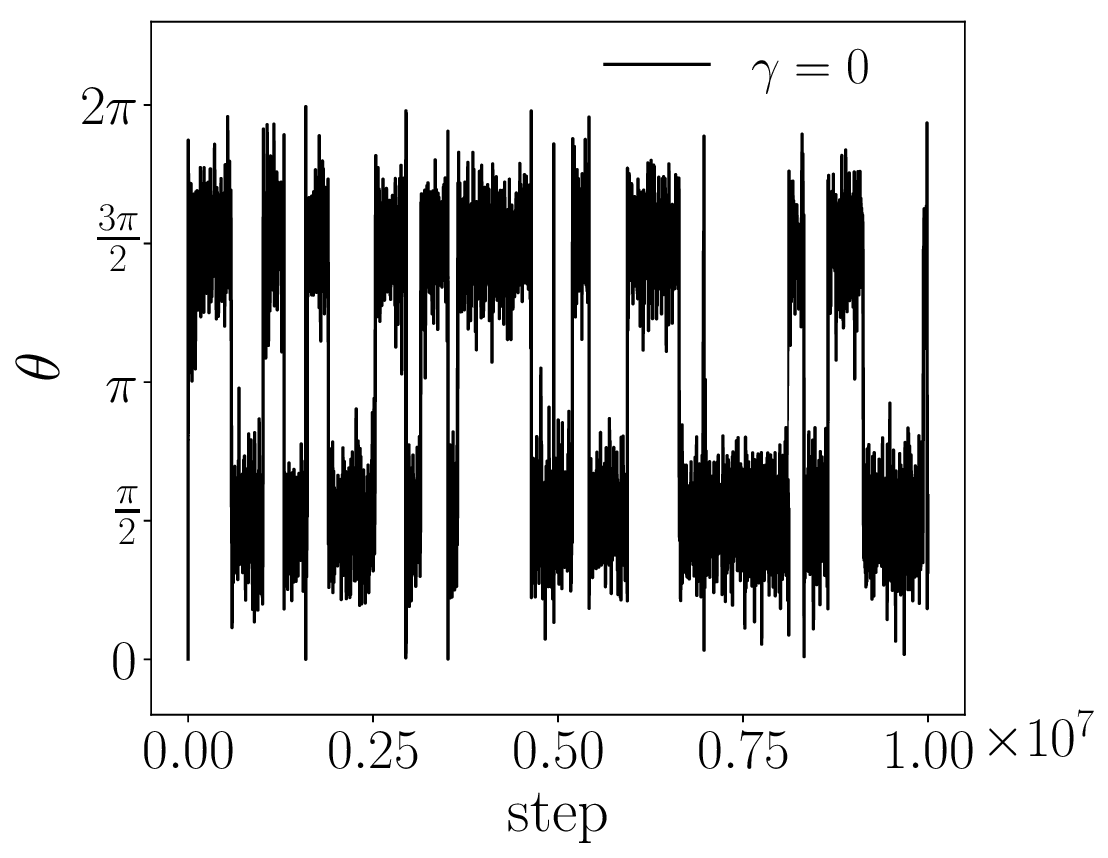}
\end{subfigure}
\begin{subfigure}{0.32\textwidth}
\includegraphics[width=1.0\textwidth]{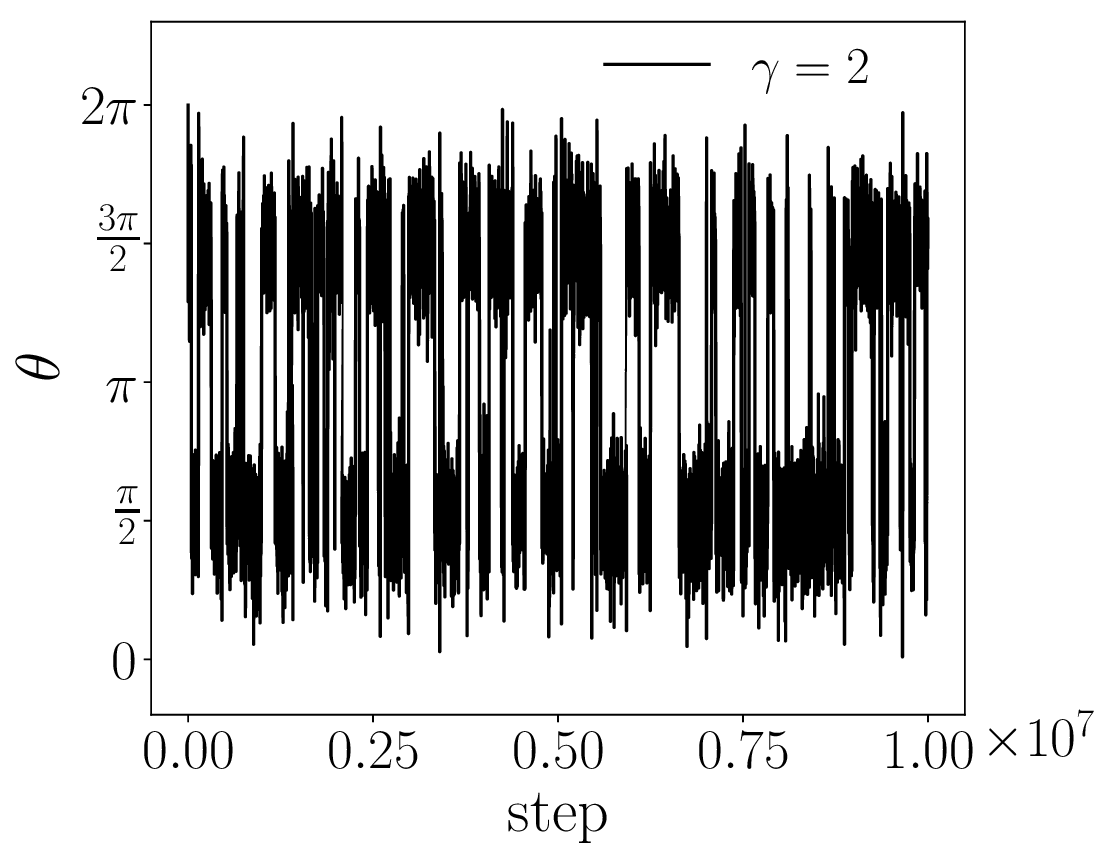}
\end{subfigure}
\begin{subfigure}{0.33\textwidth}
\includegraphics[width=1.0\textwidth]{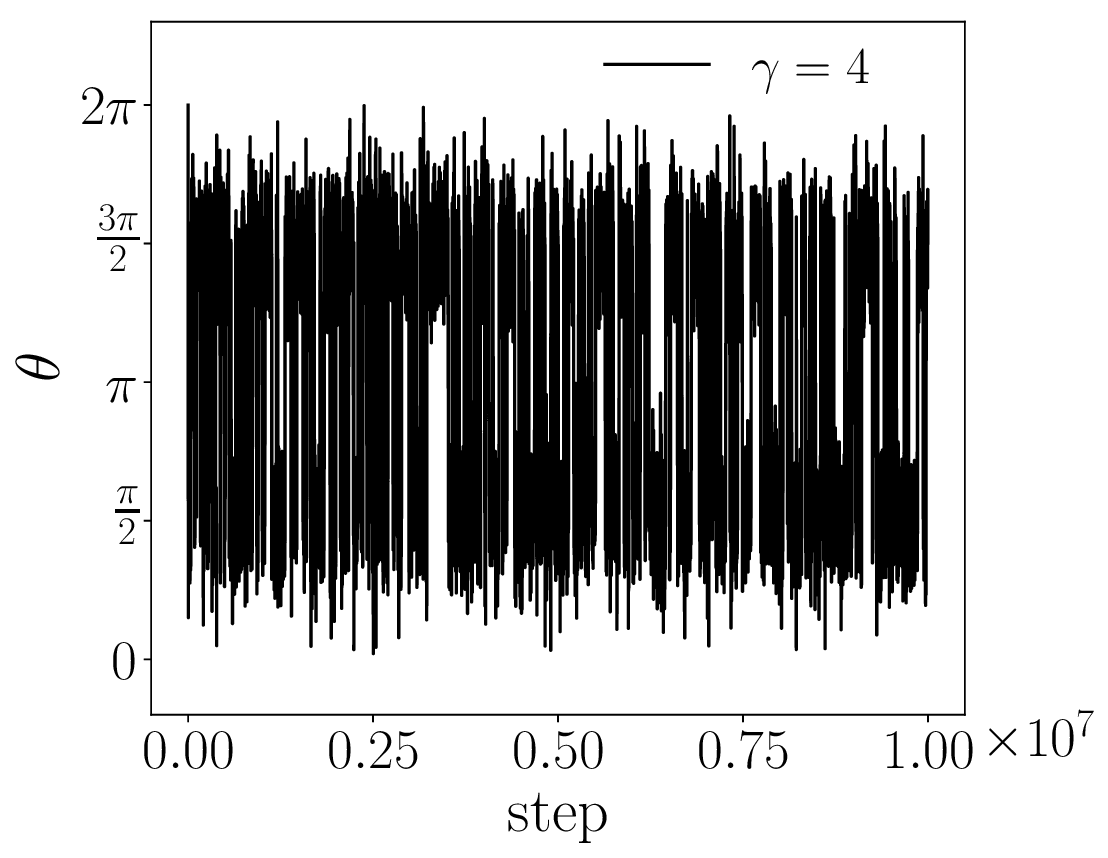}
\end{subfigure}
  \caption{Typical trajectories of angle $\theta$ in the second test~\eqref{ex:sec-test} are plotted for $\gamma=0,2,4$
  respectively. In comparison to $\gamma=0$ (left panel), the  
  transitions between the two low-potential regions of $U_2$ (see the right panel in Figure~\ref{fig-ex1-torus-potential}) occur more frequently when
  $\gamma=2, 4$ are used (middle, right panels).\label{traj-angle-plot}}
\end{figure}

In the first test, we choose 
\begin{equation}\label{ex:first-test}
\beta=20, \  U(x) = U_1(x_3) = 10 x_3^2, \  
f(x) = f_1(x_3) = 30 \Bigl(\frac{x_3}{r}\Bigr)^2,
\end{equation}
i.e.\ both $U$ and $f$ depend only on
$x_3$. See the left panels in Figure~\ref{fig-ex1-torus-potential} and Figure~\ref{fun-f-plot} for the
profiles of $U$ and $f$ respectively. In this test the asymptotic variance
$\chi_f^2$ \eqref{eqn-chi} is small thanks to the choice of the Gaussian
potential $U_1$ (the Poincar{\'e} constant of $\mu$ is large). This allows us
to focus on the estimation error in terms of step-size $h$ and to compare it with the error bound \eqref{mean-square-error-numerical-version}.

\begin{figure}[t!]
\centering
\begin{subfigure}{0.45\textwidth}
\includegraphics[width=1.0\textwidth]{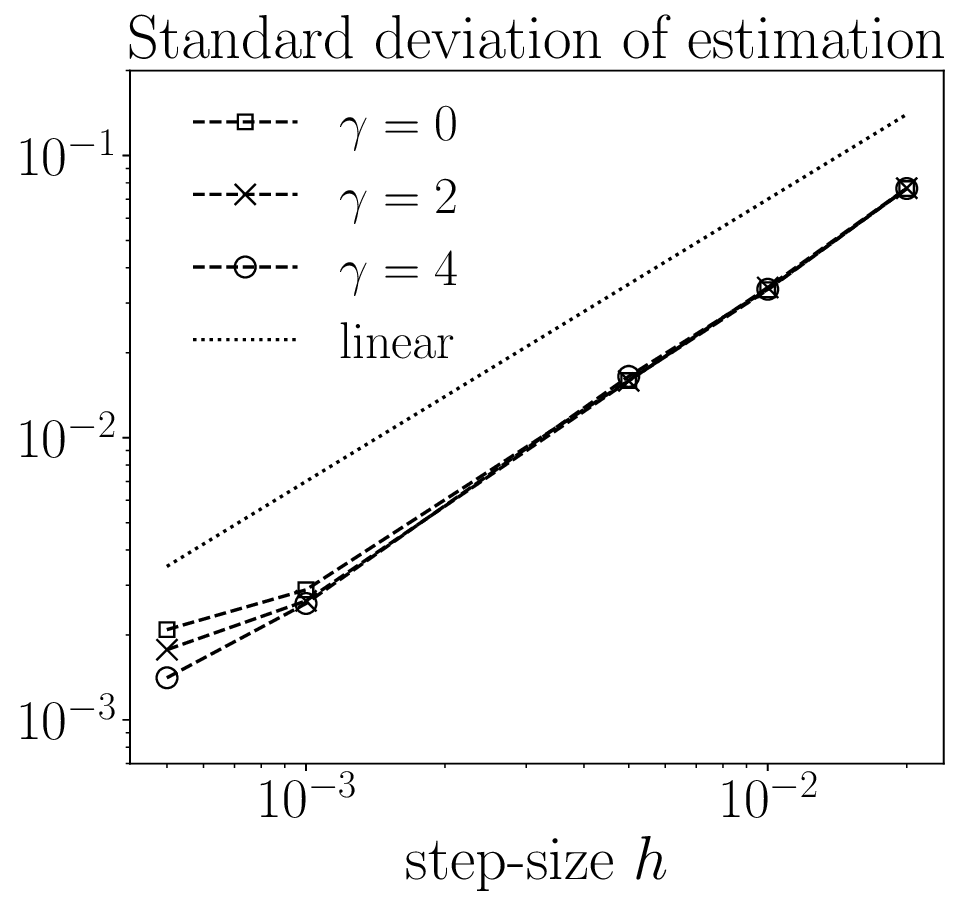}
\centering
\end{subfigure}
\begin{subfigure}{0.435\textwidth}
\includegraphics[width=1.0\textwidth]{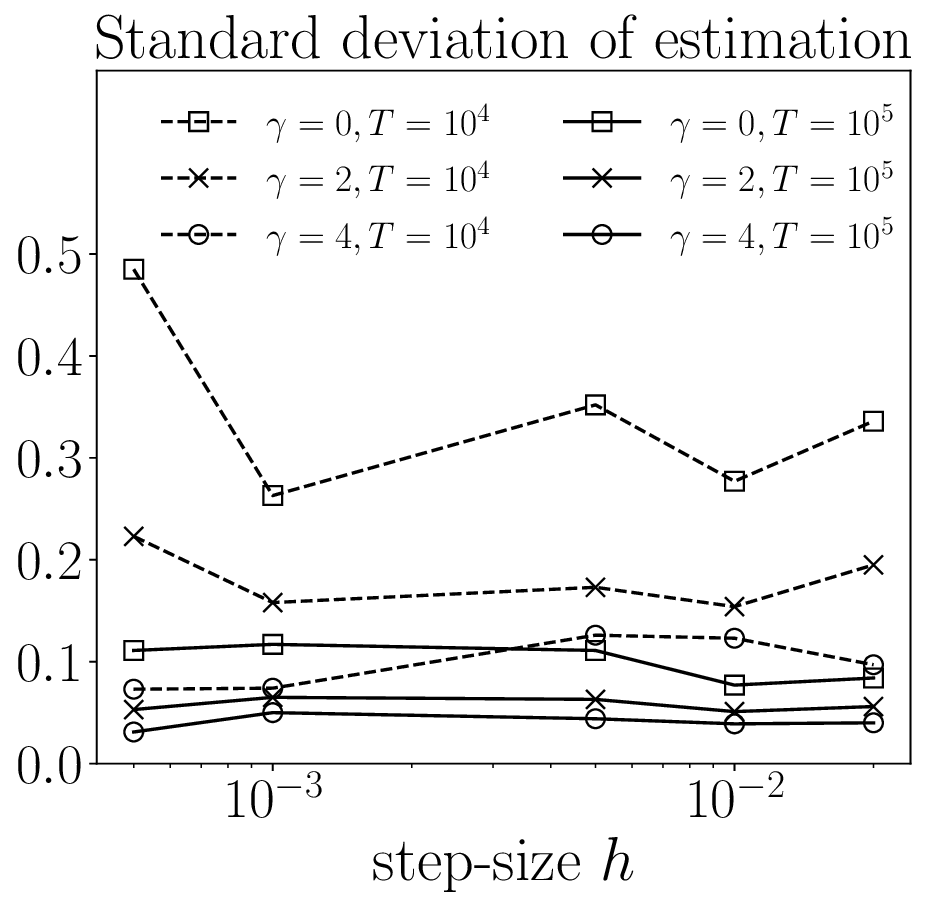}
\end{subfigure}
\vspace{-0.3cm}
  \caption{Left: Standard deviations of the mean
  $\E_\mu[f]$ based on $10$ runs of Monte Carlo estimations for the first test~\eqref{ex:first-test} 
   are plotted for $\gamma \in \{0, 2, 4\}$ with fixed total time $T=10^4$ in each run. The dotted line indicates linear scaling of standard
     deviation with respect to step-size $h$. Right: Standard
   deviations of the mean $\E_\mu[f]$ based on $10$ runs of Monte Carlo estimations for the second test~\eqref{ex:sec-test} are plotted for
   $\gamma\in \{0,2,4\}$ with two choices for total time $T=10^{4},10^{5}$. 
The standard deviation decreases when $T$ increases and also when $\gamma=2$ or $\gamma=4$ are used.
In both tests, for each step-size $h$ in \eqref{ex1-step-size}, each of the
  $10$ runs of the scheme~\eqref{scheme-non-reversible-numerical} gives a random estimations of $\E_\mu[f]$ and the standard deviation of these
  $10$ estimations with respect to the mean value calculated from direct calculation is plotted. \label{std-dev-plot}}
\end{figure}

The true value of $\E_\mu[f]$ is $0.303$ by direct numerical calculations. To test the estimation error in terms of step-size $h$ we use 
\begin{equation}  \label{ex1-step-size}
  h=\,2.0\times 10^{-2}, \quad 1.0 \times 10^{-2},\quad 5.0\times 10^{-3},\quad 1.0\times
  10^{-3},\quad 5.0\times 10^{-4}\,.
\end{equation}
We fix $a=I_3$ and choose the matrix $A$ to be 
\begin{equation}
  A =
\begin{pmatrix}
  0 & \gamma & 0\\
  -\gamma & 0 & 0\\
  0 & 0 & 0
\end{pmatrix}\,,
  \label{mat-a-bar}
\end{equation}
where $\gamma\in \{0, 2, 4\}$. When $\gamma\neq 0$, a rotational
effect is introduced in the plane spanned by $x_1$ and $x_2$. For each
step-size $h$ in \eqref{ex1-step-size} and each $\gamma \in \{0, 2, 4\}$, we estimate the mean value $\E_\mu[f]$ for $10$ runs using the scheme~\eqref{scheme-non-reversible-numerical}, where in each run $n$ states are sampled up to the fixed total time $T=nh=10^{4}$. In the projection step, the ODE \eqref{phi-map-a-kappa} 
 is solved with $\kappa = 0.5$ using the fourth-order Runge-Kutta method. 
 In order to focus on the effect of the step-size $h$, a step-size
 $\Delta t = 0.005$ is used initially and is halved each time $|\xi|$ increases during the ODE integration.
The convergence criterion is set to $|\xi| < \eps_{tol}=10^{-7}$.
Each run of the scheme~\eqref{scheme-non-reversible-numerical} gives an  estimation of $\E_\mu[f]$ and the standard deviation of $10$ runs 
(i.e.\ the standard deviation of $10$  estimations with respect to the
mean value $\E_\mu[f]=0.303$ from direct calculation) for each $h$ in
\eqref{ex1-step-size} is shown in the left panel in Figure~\ref{std-dev-plot}.
We can observe that the standard deviations decrease linearly as $h$ decreases
for each $\gamma \in \{0,2,4\}$ (the fluctuation when $h=5\times 10^{-4}$ is visible in the left
panel in Figure~\ref{std-dev-plot} due to the use of finite time $T$ and the logarithmic scale of the $y$-axis). Note that this is accordant with the error
estimate \eqref{mean-square-error-numerical-version} in Theorem~\ref{thm-numerical-theta} (see also \eqref{mean-square-error-thm2} and \eqref{mean-square-estimate-poincare}) . In fact, the term involving $\frac{\chi_f^2}{T}$ in \eqref{mean-square-error-numerical-version} is negligible for $T=10^{4}$, since 
in this test the asymptotic variance $\chi_f^2$ \eqref{eqn-chi} is small thanks to the choice of the Gaussian potential $U_1$ (the Poincar{\'e} constant of $\mu$ is large). The term involving $(\Delta t)^p$ in \eqref{mean-square-error-numerical-version} is also
small since we use the small step-size $\Delta t=0.005$ and we have $p=4$ for the fourth-order Runge-Kutta method. 
Therefore, after taking square root  one observes that the dominant term in
the error bound \eqref{mean-square-error-numerical-version} is linear in $h$. In each case, about $12$ Runge-Kutta steps are required on average in order to achieve the
convergence criterion when solving ODE~\eqref{phi-map-a-kappa}. As $h$ increases, the number of Runge-Kutta steps increases slightly, due to
the fact that the intermediate states $\widetilde{x}^{(\l+\frac{1}{2})}$ in
the scheme~\eqref{scheme-non-reversible-numerical} move further away from the torus (see left panel in Figure~\ref{dist-plot-of-RK-steps}). 
It is interesting to note that for fixed step-size $h$ in \eqref{ex1-step-size} the numbers of Runge-Kutta steps used to meet the convergence criterion are very
similar for different $\gamma\in \{0,2,4\}$ (i.e.\ different $A$),
and this is in fact consistent with Lemma~\ref{lemma-rescaled} (also see
\eqref{estimate-of-Lyapunov-f} in its proof), where the estimate of the time to reach the submanifold is independent of $A$.
In this test, because the matrix $A$ in \eqref{mat-a-bar} does not affect the system along the direction $x_3$ significantly, the sampling errors with $\gamma=2,4$ are
close to the results with $\gamma=0$ (left panel in Figure~\ref{std-dev-plot}).

\begin{table}[h!]
  \centering
  \begin{tabular}{c|cccccccc}
    \hline
    $\gamma$ & $h$ & $n$ &  mean $f$ & std.\ $f$ & $|\xi(\widetilde{x}^{(\ell+\frac{1}{2})})|$ &
    RK-S & RK-Err & frequency $\theta$-trans. \\
    \hline
    \multirow{5}{*}{$0$} & $2.0\times 10^{-2}$ & $5\times 10^{6}$ &  $1.89$ &
    $0.08$ & $2.1 \times 10^{-1}$ & $16.7$ & $2.5\times 10^{-7}$ & $3.0\times 10^{-3}$  \\
    & $1.0\times 10^{-2}$ & $1\times 10^{7}$ & $1.96$ & $0.08$ & $1.4 \times
    10^{-1}$ & $15.2$ & $1.9\times 10^{-7}$ & $2.9\times 10^{-3}$  \\
    & $5.0\times 10^{-3}$ & $2\times 10^{7}$ & $1.91$ & $0.11$ & $1.0 \times
    10^{-1}$ & $14.1$ &$1.3\times 10^{-7}$ & $3.0\times 10^{-3}$  \\
    & $1.0\times 10^{-3}$ & $1\times 10^{8}$ & $1.87$ & $0.12$ & $4.5 \times
    10^{-2}$ & $11.9$ &$1.5\times 10^{-7}$ & $3.0\times 10^{-3}$  \\
    & $5.0\times 10^{-4}$ & $1\times 10^{8}$ & $1.87$ & $0.11$ & $3.2 \times
    10^{-2}$ & $11.2$ &$1.2\times 10^{-7}$ & $2.9\times 10^{-3}$  \\
    \hline
    \hline
    \multirow{5}{*}{$2$} & $2.0\times 10^{-2}$ & $5\times 10^{6}$ & $1.90$ & $0.06$ & $2.1 \times
    10^{-1}$ & $16.8$ &$4.3\times 10^{-6}$ & $1.1\times 10^{-2}$  \\
    & $1.0\times 10^{-2}$ & $1\times 10^{7}$& $1.91$ & $0.05$ & $1.4 \times
    10^{-1}$ & $15.3$ &$2.2\times 10^{-6}$ & $1.1\times 10^{-2}$  \\
    & $5.0\times 10^{-3}$ & $2\times 10^{7}$& $1.88$ & $0.06$ & $1.0 \times
    10^{-1}$ & $14.1$ &$1.0\times 10^{-6}$ & $1.0\times 10^{-2}$  \\
    & $1.0\times 10^{-3}$ & $1\times 10^{8}$& $1.87$ & $0.06$ & $4.5 \times
    10^{-2}$ & $11.9$ &$1.1\times 10^{-6}$ & $1.1\times 10^{-2}$  \\
    & $5.0\times 10^{-4}$ & $1\times 10^{8}$& $1.89$ & $0.05$ & $3.2 \times
    10^{-2}$ & $11.3$ &$1.7\times 10^{-6}$ & $1.1\times 10^{-2}$  \\
    \hline
    \hline
    \multirow{5}{*}{$4$} & $2.0\times 10^{-2}$ & $5\times 10^{6}$ & $1.89$ & $0.04$ & $2.2 \times
    10^{-1}$ & $17.0$ &$1.5\times 10^{-5}$ & $7.2\times 10^{-2}$  \\
    & $1.0\times 10^{-2}$ & $1\times 10^{7}$& $1.90$ & $0.04$ & $1.5 \times
    10^{-1}$ & $15.4$ &$9.6\times 10^{-6}$ & $3.8\times 10^{-2}$  \\
    & $5.0\times 10^{-3}$ & $2\times 10^{7}$& $1.89$ & $0.04$ & $1.0 \times
    10^{-1}$ & $14.1$ &$1.3\times 10^{-5}$ & $3.3\times 10^{-2}$  \\
    & $1.0\times 10^{-3}$ & $1\times 10^{8}$& $1.90$ & $0.05$ & $4.5 \times
    10^{-2}$ & $11.9$ &$1.4\times 10^{-5}$& $3.3\times 10^{-2}$  \\
    & $5.0\times 10^{-4}$ & $1\times 10^{8}$& $1.92$ & $0.03$ & $3.2 \times
    10^{-2}$ & $11.2$ &$1.1\times 10^{-5}$ & $3.3\times 10^{-2}$  \\
    \hline
  \end{tabular}
  \caption{ Monte Carlo estimations of $\E_{\mu}[f]$ in the second test. 
  The true value of $\E_{\mu}[f]$ found by direct calculation is $1.923$. Different step-sizes $h$ are used in the
  scheme~\eqref{scheme-non-reversible-numerical} to estimate $\E_{\mu}[f]$,
  for different choices of $A$ in \eqref{mat-a-bar} with $\gamma\in \{0,2,4\}$.
  For each choice of $\gamma$ and $h$, $10$ runs of Monte Carlo estimations are performed, by
  sampling states using the scheme~\eqref{scheme-non-reversible-numerical} up to total time $T=nh=10^{5}$. 
 Column ``mean $f$'' displays the average of $10$ Monte Carlo runs and column
  ``std.\ $f$'' displays the standard deviations of the $10$ Monte Carlo runs with respect to the true mean $1.923$. Column
  ``$|\xi(\widetilde{x}^{(\ell+\frac{1}{2})})|$'' contains the average value of
  $|\xi(\widetilde{x}^{(\ell+\frac{1}{2})})|$, where $\widetilde{x}^{(\ell+\frac{1}{2})}$ are the
  intermediate states in the scheme~\eqref{scheme-non-reversible-numerical}
  before the projection step. Column ``RK-S'' displays the average number of
  Runge-Kutta steps required to reach the convergence criterion when solving the ODE~\eqref{phi-map-a-kappa} in the projection step.
  Column ``RK-Err'' displays the average numerical error in solving the
  ODE~\eqref{phi-map-a-kappa}, estimated by comparing the numerical solution
  of the ODE with $\Delta t=0.01$ to the solution with $\Delta t=5\times 10^{-5}$, starting from $5000$ different intermediate states $\widetilde{x}^{(\ell+\frac{1}{2})}$.
  Column ``frequency $\theta$-trans.'' displays the frequency of transitions (i.e.\ total number of transitions divided by total time $T$) for the sampled states  between the two low-potential regions
  $\{x\in \Sigma\,,\,|\theta - \frac{\pi}{2}|\le \frac{\pi}{4}\}$ and $\{x\in \Sigma\,,\,|\theta -
\frac{3\pi}{2}|\le \frac{\pi}{4}\}$. }
\label{tab-test2}
\end{table}
To demonstrate the gain by introducing non-reversibility in the
scheme~\eqref{scheme-non-reversible-numerical} and also to investigate the
errors due to the use of finite time $T$ and the numerical evaluation of the projection (see the error bound~\eqref{mean-square-error-numerical-version}), in the second test we choose a bimodal potential with 
\begin{equation}\label{ex:sec-test}
U(x) = U_2(x) =\cos^2\theta, \  f(x) = f_2(\theta)=\frac{1}{6} \theta \Bigl(\theta - \frac{3\pi}{2}\Bigr)(\theta - 2\pi), 
\end{equation}
where $\theta\in [0,2\pi)$ is the angle in \eqref{ex1-polar}. See the right panels in Figure~\ref{fig-ex1-torus-potential} and Figure~\ref{fun-f-plot} for the plots of $U$ and $f$ respectively.
 As the figure depicts, there are two distinct regions on the torus, corresponding to $\theta$ close
 to $\frac{\pi}{2}$ or $\frac{3\pi}{2}$, where the value of $U$ is small. We set $\beta = 10$, in which case 
$\mu$ \eqref{ex1-mu} satisfies Poincar{\'e} inequality with a much smaller
constant compared to the previous test due to the bimodality of $U$ and the
asymptotic variance $\chi_f^2$ \eqref{eqn-chi} is considerably larger. Consequently the sampling of $\mu$ is more difficult. 

The true value of $\E_\mu[f]$ is $1.923$, which is calculated by direct
integration. As in the first test, we test the scheme
\eqref{scheme-non-reversible-numerical} by estimating the mean $\E_{\mu}[f]$
using different step-sizes $h$ in \eqref{ex1-step-size}, where $a=I_3$ and $A$
is chosen in \eqref{mat-a-bar} with $\gamma\in\{0,2,4\}$. To investigate the effect of the finite sampling time $T$, we set the total simulation time to be either $T=10^{4}$ or $T=10^{5}$ (the sample size $n$ is determined by the choices of $h$ and $T$ since $T=nh$). For each choice of
$h$, $\gamma$ and $T$, we estimate the mean $\E_\mu[f]$ for $10$ runs using the
scheme~\eqref{scheme-non-reversible-numerical}. Barring a slightly larger step-size $\Delta t=0.01$ which is adopted in order to reduce the total runtime when $T=10^{5}$, we use the same parameters as in the first test. 
As one can see from the right panel in Figure~\ref{std-dev-plot}, for this bimodal example, 
the standard deviation of estimations is largely due to the finite sample time $T$ (or
equivalently the finite sample size $n$), while the dependence on the step-size $h$ is less apparent.
This is indeed expected from the error
bound~\eqref{mean-square-error-numerical-version} since in this case
$\chi_f^2$ is large and therefore the term in
\eqref{mean-square-error-numerical-version} involving $\frac{\chi_f^2}{T}$
becomes dominant. It is clearly observed that both the use of a larger
sampling time $T=10^{5}$ and the use of a non-zero matrix $A$ (i.e.\
$\gamma\neq 0$) in the scheme~\eqref{scheme-non-reversible-numerical} help
decrease the estimation error significantly.
As shown in Figure~\ref{traj-angle-plot}, in comparison to the trajectory
corresponding to $\gamma=0$ (left panel in Figure~\ref{traj-angle-plot}), the
switching of the sampled states between the two low-potential regions indeed becomes more frequent when $\gamma\neq 0$ (middle and right panels in Figure~\ref{traj-angle-plot}), due to the (non-reversible) rotational effect introduced by $\gamma\neq 0$. 
As $\gamma$ increases from $2$ to $4$, this rotational effect becomes stronger
(see the middle and right panels in Figure~\ref{traj-angle-plot}) and the estimation error decreases further (see right panel in Figure~\ref{std-dev-plot}).
We refer to the column ``frequency $\theta$-trans.'' of Table~\ref{tab-test2}, where the 
frequency of transitions (computed based on one of the $10$
runs with the choice $T=10^5$) between the two regions $\{x\in
\Sigma\,,\,|\theta - \frac{\pi}{2}|\le \frac{\pi}{4}\}$ and $\{x\in
\Sigma\,,\,|\theta - \frac{3\pi}{2}|\le \frac{\pi}{4}\}$  are recorded for different $\gamma\in\{0,2,4\}$.
Similar to first test, the right panel in
Figure~\ref{dist-plot-of-RK-steps} and the column ``RK-S'' of
Table~\ref{tab-test2} show that the number of Runge-Kutta steps required in order to achieve the convergence
criterion when solving the ODE~\eqref{phi-map-a-kappa} slightly increases as step-size $h$
increases, due to the increase of the distance from the intermediate states to
$\Sigma$ (see the column ``$|\xi(\widetilde{x}^{(\ell+\frac{1}{2})})|$'' in
Table~\ref{tab-test2}), and it does not change evidently for different $\gamma\in \{0,2,4\}$. 
However, the column ``RK-Err'' of Table~\ref{tab-test2} indicates that the error of the numerical solutions to the
ODE~\eqref{phi-map-a-kappa} computed by Runge-Kutta method with fixed
step-size $\Delta t=0.01$ (compared to the reference solution computed using the small step-size $\Delta t=5 \times 10^{-5}$) increases as $\gamma$ increases.
This is probably due to the fact that the constant $C$ in \eqref{assump-on-tnum} in Assumption~\ref{assump-3}
increases as $\gamma$ increases, due to the increasing magnitude of the vector field in the ODE~\eqref{phi-map-a-kappa}. This suggests that in practice one needs to tune the magnitude of $A$ in order to balance the computational gains and error due to the non-reversibility of the numerical scheme. Overall, as shown in the right panel in Figure~\ref{std-dev-plot} and in the column ``std.\ $f$'' of Table~\ref{tab-test2}, in this test the standard deviation of $10$ Monte Carlo runs is significantly reduced when $A\neq 0$, while at the same time the number of
Runge-Kutta steps is maintained within $11-17$ steps on average. These observations comply with the theoretical results in Section~\ref{subsec-summary-of-main-results} and clearly display the efficacy of the non-reversible scheme (with $A\neq0$) over its reversible counterpart (with $A=0$).

\section{Conclusion and discussion}\label{sec:discuss}
In this paper we have analysed a non-reversible projection-based numerical
scheme, which samples the conditional invariant measure on the level set of a
reaction coordinate function. We have presented quantitative error estimates which show
that the scheme is consistent, i.e.\ long-time averages converge to averages
with respect to the conditional invariant measure on the submanifold. Additionally we
have shown that this scheme analytically outperforms its reversible
counterpart~\cite{Zhang20}, in terms of smaller or equal asymptotic variance. Moreover, these features are supported by numerical examples. The proofs of the error estimates require a delicate
treatment of the ODE-based projection and an analysis of an appropriate SDE
with the correct generator, while the analysis of the asymptotic variance
extends the corresponding analysis in Euclidean spaces~\cite{DuncanLelievrePavliotis16} to submanifolds.

We now comment on some related issues and open questions. 

\emph{Assumptions on reaction coordinate and noise.} 
While in practice Gaussian random variables are preferred, in our analysis we
use bounded random variables as noise (recall \eqref{conditions-on-eta} and
Remark~\ref{rmk-on-choice-of-noise}), which ensures that the states stay
within $\Sigma^{(\delta)}$ (recall Remark~\ref{rmk-well-posedness-of-limit-theta}).  If the gradient of the
reaction coordinate $\xi$ satisfies that $\nabla\xi^T\nabla\xi \succeq c_1 I_k$ for some $c_1>0$ on entire
  $\mathbb{R}^d$ (see Remark~\ref{rmk-on-set-sigma-delta} and the proof of
  Proposition~\ref{prop:varphi}), the analysis of this paper can be extended to the case of Gaussian random variables as well. Although we circumvent this global assumption, it is typically employed in the coarse-graining literature~\cite{legoll2010effective,DLPSS18}. 

\emph{Unbounded submanifolds.} Although in this paper we have focused on reaction coordinate with compact level set $\Sigma$, 
in applications one can encounter sampling problems on unbounded submanifolds. We expect that our results will apply to this setting, but care needs to be taken when handling the corresponding Poisson equations on unbounded domains (see~\cite{pardoux2003poisson}). 
To the best of our knowledge, the analysis of sampling schemes for nonlinear
reaction coordinates with unbounded level sets is open. 

\emph{Connections between the numerical scheme and SDE.} The
SDE~\eqref{sde-on-sigma} plays a crucial role in the proof of
Theorem~\ref{thm:Err}, wherein the zero-order terms of the Taylor expansion
are identified with the generator of~\eqref{sde-on-sigma}. This suggests that
there is a strong connection between the numerical scheme~\eqref{scheme-non-reversible} and the SDE~\eqref{sde-on-sigma}. While
we do not pursue this line of enquiry, following the results in~\cite[Section
6.1]{MattinglyStuartTretyakov10} it is possible to quantitatively estimate the
distance between the invariant measure of the numerical scheme~\eqref{scheme-non-reversible} and the invariant measure of the SDE~\eqref{sde-on-sigma}. 
Moreover we expect that \eqref{scheme-non-reversible} is a consistent numerical discretisation of the SDE~\eqref{sde-on-sigma} on finite time horizons.

\emph{Sampling schemes using Langevin dynamics. } The numerical scheme~\eqref{scheme-non-reversible} is inspired by the overdamped Langevin dynamics, and a natural extension would be to use the (underdamped) Langevin dynamics. In this case, one approach would be to propagate the position and momentum via  standard numerical schemes for the Langevin dynamics, and then to project the position onto the manifold $\Sigma$ using the approach proposed in this paper. However the right choice for the projected momentum such that the numerical scheme samples the correct target measure requires further investigation.  

\emph{Metropolisation.} Constrained numerical schemes on submanifolds using
Lagrange multipliers can be metropolised by adding a Metropolis-Hasting
acceptance-rejection step~\cite{hmc-submanifold-tony,goodman-submanifold,multiple-projection-mcmc-submanifolds}.  An interesting open problem is to explore the metropolisation of the reversible version of our numerical scheme (i.e.\ with $A=0$). This will be addressed in future work.

\vspace{1em}
\noindent\textbf{Acknowledgments.} The authors would like to thank the anonymous referees for valuable suggestions and comments. US thanks P\'eter Koltai for
stimulating discussion on the ODE flow.  The work of US and WZ is supported by the DFG under Germany's Excellence Strategy-MATH+: The Berlin Mathematics Research Centre (EXC-2046/1)-project ID:390685689 (subproject EF4-4). US also acknowledges support from the Alexander von Humboldt foundation.

\begin{appendices}
\section{Auxiliary identity}\label{app:lemma-identity-proof}
We prove the following lemma which has been used in the proof of
  Proposition~\ref{prop2} (see Section~\ref{sec-map-theta}).
 \begin{lemma}
  For $x\in \Sigma$ we have
  \begin{equation*}  \label{lemma-identity}
\nabla\xi^T \lim_{t\rightarrow +\infty}  \int_0^t
    \Big[\mathrm{e}^{-s\Gamma} a\,(\mathrm{e}^{-s\Gamma})^T\Big] \,ds = 
    \frac{1}{2}\Phi^{-1} \nabla\xi^T(a-A)   \,.
  \end{equation*}
\end{lemma}
\begin{proof}
Using~\eqref{eqn-1st-varphi} we find
\begin{align}\label{eqn-lemma-identity-1}
       &\nabla\xi^T\int_0^t \Big[\mathrm{e}^{-s\Gamma} a (\mathrm{e}^{-s\Gamma})^T\Big] \,ds \nonumber\\
 =&     \int_0^t\,\nabla\xi^T\Big(P + (a-A)\nabla\xi \mathrm{e}^{-s\Phi} \Phi^{-1}\nabla\xi^T\Big) 
     a\Big(I_d + (a-A)\nabla\xi (\mathrm{e}^{-s\Phi}-I_k) \Phi^{-1}\nabla\xi^T\Big)^T \, ds \nonumber\\
    =&  \int_0^t \,\bigg[\nabla\xi^T(a-A)\nabla\xi \mathrm{e}^{-s\Phi}
    \Phi^{-1}\nabla\xi^T a \Big(I_d + (a-A)\nabla\xi (\mathrm{e}^{-s\Phi} - I_k)
    \Phi^{-1}\nabla\xi^T\Big)^T \,\bigg] ds \nonumber\\
    =&  \int_0^t \,\bigg[\mathrm{e}^{-s\Phi} \nabla\xi^T a \Big(I_d +
    (a-A)\nabla\xi (\mathrm{e}^{-s\Phi} - I_{k})
    \Phi^{-1}\nabla\xi^T\Big)^T \,\bigg] ds \nonumber\\
    =&  \left[\int_0^t \mathrm{e}^{-s\Phi}\,ds\right]\nabla\xi^Ta + \int_0^t
    \left[\mathrm{e}^{-s\Phi} \nabla\xi^T\,a \,\nabla\xi \Phi^{-T}
    (\mathrm{e}^{-s\Phi^T} - I_k) \nabla\xi^T (a+A)\right] ds \,,
\end{align}
where  
  the second equality follows from $\nabla\xi^T P=0$ (see Lemma~\ref{lemma-p-b}), the third equality 
  follows from the definition of $\Phi$ in~\eqref{phi-gamma} and $\Phi \mathrm{e}^{-s\Phi}\Phi^{-1} =
  \mathrm{e}^{-s\Phi}$. 

Using $2a = (a-A) + (a+A)$, for the final integral term in the right hand side of~\eqref{eqn-lemma-identity-1} we have 
  \begin{align*}
      &\quad \int_0^t \left[\mathrm{e}^{-s\Phi} \nabla\xi^T\,a \,\nabla\xi \Phi^{-T}
    (\mathrm{e}^{-s\Phi^T} - I_k) \nabla\xi^T (a+A)\right] ds \\
      &=   \frac{1}{2} \int_0^t
      \left[\mathrm{e}^{-s\Phi} \nabla\xi^T\,(a-A) \,\nabla\xi \Phi^{-T}
      (\mathrm{e}^{-s\Phi^T} - I_k) \nabla\xi^T (a+A)\,\right] ds \\
      & \ \ \ + \frac{1}{2}\int_0^t \left[\mathrm{e}^{-s\Phi} \nabla\xi^T\,(a+A)
      \,\nabla\xi \Phi^{-T} (\mathrm{e}^{-s\Phi^T} - I_k) \nabla\xi^T (a+A)\right] ds \\
      &=  \frac{1}{2}\int_0^t \left[\mathrm{e}^{-s\Phi} \Phi \Phi^{-T}
      (\mathrm{e}^{-s\Phi^T} - I_k) \nabla\xi^T (a+A)\,\right] ds 
     +\frac{1}{2} \int_0^t \left[\mathrm{e}^{-s\Phi} (\mathrm{e}^{-s\Phi^T} - I_k)
      \nabla\xi^T (a+A)\right]\, ds \\
      &=  -\frac{1}{2}\left[\int_0^t \mathrm{e}^{-s\Phi} \,ds\right] (\Phi\Phi^{-T}+I_k)\nabla\xi^T (a+A) 
      + \frac{1}{2} \left[\int_0^t \mathrm{e}^{-s\Phi} (\Phi\Phi^{-T}+I_k)
      \mathrm{e}^{-s\Phi^T}\, ds \right] \nabla\xi^T (a+A)\,, 
  \end{align*}
where the final equality follows by rearranging terms. Substituting this relation into \eqref{eqn-lemma-identity-1} we arrive at
    \begin{equation*}
      \begin{aligned}
	&\nabla\xi^T \int_0^t \Big[\mathrm{e}^{-s\Gamma} a (\mathrm{e}^{-s\Gamma})^T\Big] \,ds\\
      &=  \frac{1}{2}\left[\int_0^t \mathrm{e}^{-s\Phi} \,ds\right] 
	\left[\nabla\xi^T (a-A)- \Phi\Phi^{-T}\nabla\xi^T (a+A) \right]
      + \frac{1}{2} \left[\int_0^t \mathrm{e}^{-s\Phi} (\Phi\Phi^{-T}+I_k)
      \mathrm{e}^{-s\Phi^T}\, ds \right] \nabla\xi^T (a+A) \\
      &=  \frac{1}{2} 
(I_k-\mathrm{e}^{-t\Phi})
	\left[\Phi^{-1}\nabla\xi^T (a-A)- \Phi^{-T}\nabla\xi^T (a+A) \right]
      + \frac{1}{2} (I_k-\mathrm{e}^{-t\Phi}\mathrm{e}^{-t\Phi^{T}})\Phi^{-T}
      \nabla\xi^T (a+A)\,,
  \end{aligned}
\end{equation*}
where the final equality follows from the integration identities (which can be
verified by differentiating both sides)
\begin{equation}\label{eq-exp-iden}
  \begin{aligned}
    & \int_0^t \mathrm{e}^{-s\Phi} ds =  (I_k-\mathrm{e}^{-t\Phi})\Phi^{-1}\,,\\
    & \int_0^t \mathrm{e}^{-s\Phi} (\Phi\Phi^{-T} + I_k) \mathrm{e}^{-s\Phi^T} ds =
    (I_k-\mathrm{e}^{-t\Phi}\mathrm{e}^{-t\Phi^{T}})\Phi^{-T}\,.
  \end{aligned}
\end{equation}
Since all eigenvalues of $\Phi$ have positive real parts (recall
Lemma~\ref{lemma-on-matrix-pi-phi}), passing $t\rightarrow +\infty$ we find 
  \begin{equation*}
    \begin{aligned}
      &\quad \nabla\xi^T \lim_{t\rightarrow +\infty}\int_0^t \Big[\mathrm{e}^{-s\Gamma} 
    a (\mathrm{e}^{-s\Gamma})^T\Big] \,ds\\
    &=  \lim_{t\rightarrow +\infty}
      \bigg[\frac{1}{2} 
(I_k-\mathrm{e}^{-t\Phi})
	\left[\Phi^{-1}\nabla\xi^T (a-A)- \Phi^{-T}\nabla\xi^T (a+A) \right]
      + \frac{1}{2} (I_k-\mathrm{e}^{-t\Phi}\mathrm{e}^{-t\Phi^{T}})\Phi^{-T}
      \nabla\xi^T (a+A)\bigg]\\
    &=   \frac{1}{2}\Phi^{-1} \nabla\xi^T(a-A)\,. 
    \end{aligned}
  \end{equation*}
\end{proof}

\section{Proof of Theorem~\ref{thm:Err}}\label{app:thm-err}
In this section, we prove Theorem~\ref{thm:Err}, Corollary~\ref{corollary-on-spectral-gap} and Theorem~\ref{thm-numerical-theta}.

Since we will work with higher-order derivatives, let us first introduce some additional notations that will be used below. For any function $g:
\Sigma^{(\delta)}\rightarrow \mathbb{R}^m$, 
where $m\ge 1$, we denote by
\begin{equation*}
    D^j g_i [\bm u_1,\dots,\bm u_j]:= \sum_{r_1=1}^d\sum_{r_2=1}^d\cdots\sum_{r_j=1}^d
    \frac{\partial^j g_i}{\partial
  x_{r_1}\dots \partial x_{r_j}} u_{1 r_1}\,u_{2 r_2} \dots u_{j r_j}\,, \quad 1 \le i
    \le m\,,
\end{equation*}
the directional derivatives along vectors $\bm u_1,\dots,\bm u_j\in \R^d$, where $\bm u_r = (u_{r1},
  u_{r2}, \dots, u_{rd})^T$ for $1 \le r \le j$. When $j=1$, we also use the notation $Dg_i$ for $D^1g_i$.
  We denote by $D^j g[\bm u_1,\dots,\bm u_j]$ the vector $$(D^j g_1[\bm
  u_1,\dots,\bm u_j], \dots, D^j g_m[\bm u_1,\dots,\bm u_j])^T \in \mathbb{R}^m.$$   
  Note that $D^jg$ defines a multilinear operator acting on vectors. 
  We denote by $\|D^jg\|_\infty$ the supremum norm of $D^jg$ on $\Sigma^{(\delta)}$, i.e.\ at each point on $\Sigma^{(\delta)}$, we have
  \begin{equation}
|D^j g [\bm u_1,\dots,\bm u_j]| \le \|D^jg\|_\infty \,|\bm u_1|\cdot \dots
    \cdot|\bm u_j|\,, \quad \forall~\bm u_1,\ldots, \bm u_j \in \mathbb{R}^d\,.
  \end{equation}

  To keep the notations both unified and consistent with Section~\ref{secsub-notations}, we 
  define $D g := \nabla g \in \mathbb{R}^{d\times m}$, with 
  $(D g)_{ji} = \frac{\partial g_i}{\partial x_j}$ for $1 \le j\le d$ and $1 \le i \le m$,
and we denote by $a:D^2 g$ the vector in $\mathbb{R}^m$ with
  components $\big(a : D^2 g\big)_i=\sum_{j,r=1}^d\frac{\partial^2 g_i}{\partial x_j\partial x_r} a_{jr}$ for  $1 \le i \le m$.

  In the proof below, we will consider values of a function $g$ at different states $x^\lb$ (for $\ell \ge 0$) generated by the numerical scheme~\eqref{scheme-non-reversible}. To simplify the presentation, we write $g^\lb:=g(x^\lb)$. 

Next, we recall the generator $\mathcal{L}$ of SDE~\eqref{sde-on-sigma} (defined in \eqref{generator-l})
\begin{equation}\label{def:L-app}
  \mathcal{L}
    = \frac{\mathrm{e}^{\beta U}}{\beta} \sum_{i,j=1}^d\frac{\partial}{\partial
    x_j}\left(B_{ij} \mathrm{e}^{-\beta U} \frac{\partial }{\partial x_i} \right) 
    = -\sum_{i,j=1}^d B_{ij}\frac{\partial U}{\partial
    x_j}\frac{\partial}{\partial x_i} + \frac{1}{\beta}
    \sum_{i,j=1}^d\frac{\partial B_{ij}}{\partial x_j}\frac{\partial}{\partial
    x_i} +\frac{1}{\beta}\sum_{i,j=1}^dB_{ij}\frac{\partial^2}{\partial
    x_i\partial x_j}\,.
\end{equation}
The following result discusses the well-posedness of the Poisson problem~\eqref{eqn-poisson-equation}, which will be used in proof of Theorem~\ref{thm:Err}.
\begin{prop}\label{prop-Poisson}
For any $f\in C^2(\Sigma)$, the Poisson problem   
\begin{equation}\label{eq:Poisson-App}
\mathcal L \psi = f-\bar f\,,\quad \mbox{on}~ \Sigma\quad \textrm{with}~ \E_{\cmu}[\psi] = 0\,,
\end{equation}
  with $\bar f = \E_{\cmu}[f]$, has a unique solution $\psi\in C^4(\Sigma)$. Furthermore,
  there exists an extension of $\psi$ to $C^4(\Sigma^{(\delta)})$, and a constant $C>0$ independent of $f$, such that  
\begin{equation*}
\|\psi\|_{\infty},\, \|D\psi\|_{\infty},\,
  \|D^2\psi\|_{\infty},\,\|D^3\psi\|_{\infty},\,
  \|D^4\psi\|_{\infty} \leq C(\|f\|_{\infty, \Sigma}+\|Df\| _{\infty, \Sigma}+\|D^2f\| _{\infty, \Sigma})\,,
\end{equation*}
where $\|\cdot\|_{\infty, \Sigma}$ denotes the supremum norm on $\Sigma$.
\end{prop}
The proof of Proposition~\ref{prop-Poisson} follows from classical elliptic
theory (see~\cite[Section 4.1]{MattinglyStuartTretyakov10} for instance) and
standard extension results. 
Applying Proposition~\ref{prop-1st-varphi} and Proposition~\ref{prop2}, 
we can characterise the generator $\mathcal L$ in terms of derivatives of $\Theta^A$.
\begin{lemma}\label{app:lem-gen-mod}
For any $g\in C^2(\Sigma)$, by abuse of notation $g$ also denotes a smooth extension to $C^2(\Sigma^{(\delta)})$.
  We have
\begin{align}\label{eq:gen-Theta}
  \begin{split}
  \mathcal L g =& \sum_{i=1}^d\frac{\partial g}{\partial
    x_i}\sum_{r,j=1}^d  \bigg[ -(a-A)_{rj}\frac{\partial U}{\partial x_j}\frac{\partial \Theta^A_i}{\partial x_r} 
+\frac{1}{\beta} \frac{\partial \Theta_i^A}{\partial x_r}\frac{\partial a_{rj}}{\partial x_j}
+\frac{1}{\beta} a_{rj}\frac{\partial^2 \Theta_i^A}{\partial x_r\partial x_j} 
\bigg]\\
  &
+\frac{1}{\beta} 
  \sum_{i,j=1}^d
    \frac{\partial^2 g}{\partial x_i\partial x_j} \Big(\sum_{\l, r=1}^d
    \frac{\partial \Theta_i^A}{\partial x_r} a_{r\l} \frac{\partial
    \Theta_j^A}{\partial x_\l}\Bigr), \qquad \mbox{on}~\Sigma\,,
  \end{split}
\end{align}
where $\mathcal L$ is in~\eqref{def:L-app}. The right hand side of \eqref{eq:gen-Theta} does not depend on the choice of extensions used for $g$.
\end{lemma}
\begin{proof}
For states on $\Sigma$, using~\eqref{projection-p}, \eqref{eqn-of-theta-map} and~\eqref{eqn-2nd-derivative-theta} with $1 \le i \le d$ we find
\begin{align*}
\sum_{r,j=1}^d \bigg[
  -(a-A)_{rj}\frac{\partial U}{\partial x_j}\frac{\partial \Theta^A_i}{\partial x_r} 
  + \frac{1}{\beta}\frac{\partial \Theta_i^A}{\partial x_r}\frac{\partial a_{rj}}{\partial x_j}
+\frac{1}{\beta}a_{rj}\frac{\partial^2 \Theta_i^A}{\partial x_r\partial x_j} \bigg]
= 
  \sum_{j=1}^d \bigg(-B_{ij} \frac{\partial U}{\partial x_j} +
  \frac{1}{\beta}\frac{\partial B_{ij}}{\partial x_j}\bigg).
\end{align*}
Substituting this relation along with 
\begin{align*}
  \sum_{i,j=1}^d \frac{\partial^2 g}{\partial x_i\partial
  x_j}\Big(\sum_{\l,r=1}^d \frac{\partial \Theta_i^A}{\partial x_r} a_{r\l}
  \frac{\partial \Theta_j^A}{\partial x_\l}\Big) = \sum_{i,j=1}^d
  \frac{\partial^2 g}{\partial x_i\partial x_j} (PaP^T)_{ij} = \sum_{i,j=1}^d
  \frac{\partial^2 g}{\partial x_i\partial x_j} B_{ij} 
\end{align*}
into the right hand side of~\eqref{eq:gen-Theta} we conclude that \eqref{eq:gen-Theta} is the same as \eqref{def:L-app}.
  Since $\mathcal{L}g$ is independent of extensions of $g$ (see the
  first item of Remark~\ref{rmk-on-sde-and-l}), 
  we conclude that the right hand side of \eqref{eq:gen-Theta} is independent of extensions of $g$ as well.
\end{proof}

We now present the proof of Theorem~\ref{thm:Err}.
\begin{proof}[Proof of Theorem~\ref{thm:Err}]
Since the proof is similar to that of~\cite[Theorem~3.5]{Zhang20} and follows on the
  lines of~\cite[Section 5]{MattinglyStuartTretyakov10}, we only outline the
  main steps here. 

Define the vector $\bm{b}^\lb=(b^\lb_1, b^\lb_2, \dots, b^\lb_d)^T$ by 
  \begin{equation}  \label{b-vector}
    b_i^\lb = 
    \sum_{j=1}^d \bigg[-\big(a_{ij}(x^\lb)-A_{ij}\big)\frac{\partial U}{\partial
    x_j}(x^\lb)+\frac1\beta \frac{\partial a_{ij}}{\partial x_j}(x^\lb)\bigg]\,, \quad 1
  \le i \le d\,,
\end{equation} 
  for $\l = 0,1,\dots$, and set 
  \begin{equation}    \label{delta-vector-by-b}
    \bm{\delta}^\lb = \bm{b}^\lb h + \sqrt{2\beta^{-1} h}\,\sigma^\lb\bm{\eta}^\lb\,.
\end{equation}
The numerical scheme~\eqref{scheme-non-reversible} can be written as 
\begin{equation}\label{delta-and-x}
  x^{(\ell+\frac{1}{2})}= x^\lb + \bm{\delta}^\lb \quad\mbox{and}\quad
  x^{(\ell+1)} = \Theta^A\big(x^{(\ell+\frac{1}{2})}\big) = \Theta^A\big(x^\lb+\bm{\delta}^\lb\big)\,. 
\end{equation}
  Since $\Sigma$ is compact (Assumption~\ref{assump-2}) and $\bm{\eta}^\lb$ is a bounded random variable (see \eqref{conditions-on-eta}),
  from \eqref{delta-vector-by-b} it is easy to see that there exists $h_0>0$, such that for $h< h_0$ we have $x^{(\ell+\frac{1}{2})}\in \Sigma^{(\delta)}$ for any $\ell \ge 0$.

In what follows we will make use of the Poisson equation~\eqref{eq:Poisson-App} on $\Sigma$,  
  which has a unique solution $\psi$ such that $\psi\in C^4(\Sigma)$ by Proposition~\ref{prop-Poisson}. 
For simplicity we use the same notation to denote a smooth
  extension of $\psi$ to $C^4(\Sigma^{(\delta)})$. The calculation below is
  independent of the choice of extensions (see Lemma~\ref{app:lem-gen-mod}).
  
  Applying Taylor's theorem in $\Sigma^{(\delta)}\subseteq \mathbb{R}^d$ and using 
  $\Theta^A(x^\lb)=x^\lb$ since $x^\lb\in\Sigma$ (Proposition~\ref{prop:varphi}) we find
\begin{align}
  \psi^{(\ell+1)} &= \psi(x^{(\ell+1)})\notag\\
  &= (\psi\circ\Theta^A)\big(x^\lb + \bm{\delta}^\lb\big)\notag\\
    &= \psi^\lb + D(\psi\circ\Theta^A)^\lb[\bm{\delta}^\lb]
    + \frac{1}{2}D^2(\psi\circ\Theta^A)^\lb[\bm{\delta}^\lb, \bm{\delta}^\lb] +
    \frac{1}{6}D^3(\psi\circ\Theta^A)^\lb[\bm{\delta}^\lb,\bm{\delta}^\lb,\bm{\delta}^\lb] + R^\lb \notag\\
  &= \psi^\lb + D\psi^\lb\Big[(D\Theta^A)^\lb[\bm{\delta}^\lb]+\frac{1}{2} (D^2\Theta^A)^\lb[\bm{\delta}^\lb,\bm{\delta}^\lb]\Big] \label{eq:aux1}\\
  &\quad+ \frac{1}{2}D^2\psi^\lb\Big[(D\Theta^A)^\lb[\bm{\delta}^\lb], (D\Theta^A)^\lb[\bm{\delta}^\lb]\Big]  
  +
  \frac{1}{6}D^3(\psi\circ\Theta^A)^\lb[\bm{\delta}^\lb,\bm{\delta}^\lb,\bm{\delta}^\lb]
  + R^\lb\,\notag\,,
\end{align}
where we have used chain rule to compute the first and second derivatives of 
  $\psi\circ\Theta^A$, and the reminder term is explicitly defined as 
\begin{equation*}
R^\lb:=\frac16\Bigl(\int_0^1 (1-s)^3D^4(\psi\circ\Theta^A)(x^\lb+s
  \bm{\delta}^\lb)ds\Bigr)
  \big[\bm{\delta}^\lb,\bm{\delta}^\lb,\bm{\delta}^\lb,\bm{\delta}^\lb\big].
\end{equation*}
Using~\eqref{delta-vector-by-b} for the second term in the right hand side of~\eqref{eq:aux1} we compute
\begin{align*}
&D\psi^\lb\Big[(D\Theta^A)^\lb[\bm{\delta}^\lb]+\frac{1}{2} (D^2\Theta^A)^\lb[\bm{\delta}^\lb,\bm{\delta}^\lb]\Big] \\
&=\, h D\psi^\lb \Big[ (D\Theta^A)^\lb [\bm b^\lb] + \beta^{-1}a^\lb :
  (D^2\Theta^A)^\lb  \Big]\\
  & + \beta^{-1} h  D\psi^\lb \Big[ (D^2\Theta^A)^\lb [\sigma^\lb\bm\eta^\lb,\sigma^\lb\bm\eta^\lb] 
   - a^\lb : (D^2\Theta^A)^\lb  \Big] \\
&+ \sqrt{2\beta^{-1} h} D\psi^\lb \Big[(D\Theta^A)^\lb [\sigma^\lb \bm\eta^\lb]\Big]
  + (2\beta)^{-\frac12} h^{\frac32} D\psi^\lb \Big[(D^2\Theta^A)^\lb [\bm
  b^\lb,\sigma^\lb\bm\eta^\lb]\Big] \\
&+ \frac{h^2}{2}D\psi^\lb \Big[(D^2\Theta^A)^\lb [\bm b^\lb,\bm b^\lb]\Big],
\end{align*}
where we have added and subtracted the term $a^\lb : (D^2 \Theta^A)^\lb$.

Similarly, the third term in the right hand side of~\eqref{eq:aux1} gives
\begin{equation*}
\begin{aligned}
&\frac{1}{2}D^2\psi^\lb\Big[(D\Theta^A)^\lb[\bm{\delta}^\lb], (D\Theta^A)^\lb[\bm{\delta}^\lb]\Big] \\
  &=\, \beta^{-1}h D^2\psi^\lb : \big((D\Theta^A)^T a\,D\Theta^A\big)^\lb \\
  &\,+ \frac{h^2}{2}D^2\psi^\lb \Big[(D\Theta^A)^\lb \bm [\bm b^\lb],
  (D\Theta^A)^\lb [\bm b^\lb]\Big] 
  + (2\beta)^{-\frac12} h^{\frac32}D^2\psi^\lb\Big[(D\Theta^A)^\lb[\bm b^\lb],(D\Theta^A)^\lb[\sigma^\lb\bm\eta^\lb]\Big] \\ 
  &\,+ \beta^{-1}h\Big(D^2\psi^\lb
  \Big[(D\Theta^A)^\lb[\sigma^\lb\bm\eta^\lb],(D\Theta^A)^\lb[\sigma^\lb\bm\eta^\lb]\Big]
  -D^2\psi^\lb : \big((D\Theta^A)^T a\, D\Theta^A\big)^\lb \Big)\,.
\end{aligned}
\end{equation*}

Substituting these expressions back into~\eqref{eq:aux1}, using Lemma~\ref{app:lem-gen-mod} which states that 
\begin{align*}
(\mathcal L\psi)^\lb = D\psi^\lb\Big[(D\Theta^A)^\lb[\bm b^\lb] +\beta^{-1}
  a^\lb:(D^2 \Theta^A)^\lb\Big] + \beta^{-1} D^2\psi^\lb
: \big((D\Theta^A)^T a\, D\Theta^A\big)^\lb , 
\end{align*}
summing over $\ell = 0,1,\ldots,n-1$, and dividing by $T$ we find
\begin{align}\label{eq:avg-diff}
\hat f_n -\bar f = \frac 1n \sum\limits_{\ell = 0}^{n-1} f(x^\lb) -\bar f = \frac hT\sum\limits_{\ell = 0}^{n-1}(\mathcal L\psi)^\lb = \frac1T (\psi^{(n)}-\psi^{(0)}) +\frac1T \sum\limits_{i=0}^5 M_{i,n} + \frac1T \sum\limits_{i=0}^4 S_{i,n}.
\end{align}
Here we have used the Poisson equation~\eqref{eq:Poisson-App} to arrive at the second equality and 
\begin{align*}
M_{1,n}&:=- \sqrt{2\beta^{-1} h} \sum\limits_{\ell=0}^{n-1}
  D\psi^\lb\Big[(D\Theta^A)^\lb[\sigma^\lb\bm\eta^\lb]\Big],\\
  M_{2,n}&:=-(2\beta)^{-\frac12} h^{\frac32} \sum\limits_{\ell=0}^{n-1}D\psi^\lb
  \Big[(D^2\Theta^A)^\lb[\bm b^\lb,\sigma^\lb\bm\eta^\lb]\Big] ,\\
  M_{3,n}&:=-\beta^{-1} h \sum\limits_{\ell=0}^{n-1} \Big(D\psi^\lb
  \Big[(D^2\Theta^A)^\lb \big[\sigma^\lb\bm\eta^\lb,\sigma^\lb\bm\eta^\lb\big] -
  a^\lb : (D^2\Theta^A)^\lb\Big] \Big),\\
  M_{4,n}&:=-(2\beta)^{-\frac12} h^{\frac32} \sum\limits_{\ell=0}^{n-1}
  D^2\psi^\lb\Big[(D\Theta^A)^\lb[\bm b^\lb],(D\Theta^A)^\lb[\sigma^\lb\bm\eta^\lb]\Big],\\
  M_{5,n}&:=-\beta^{-1}h\sum\limits_{\ell=0}^{n-1}\Big(D^2\psi^\lb \Big[
  (D\Theta^A)^\lb
  [\sigma^\lb\bm\eta^\lb],(D\Theta^A)^\lb[\sigma^\lb\bm\eta^\lb]\Big]-
D^2\psi^\lb : \big((D\Theta^A)^T a\, D\Theta^A\big)^\lb \Big)\,,
\end{align*}
and 
\begin{align*}
S_{1,n}&:=-\frac{1}{2}h^2\sum\limits_{\ell=0}^{n-1} D\psi^\lb
  \Big[(D^2\Theta^A)^\lb [\bm b^\lb,\bm b^\lb]\Big],\\
S_{2,n}&:=-\frac12 h^2\sum\limits_{\ell=0}^{n-1}D^2\psi^\lb
  \Big[(D\Theta^A)^\lb [\bm b^\lb],(D\Theta^A)^\lb [\bm b^\lb]\Big],\\
S_{3,n}&:=-\sum\limits_{\ell=0}^{n-1}R^\lb.
\end{align*}
Furthermore using~\eqref{delta-vector-by-b} we split the fourth term in~\eqref{eq:aux1} to arrive at
\begin{align*}
M_{0,n} :=& -\frac16\sum\limits_{\ell=0}^{n-1}
  \Big((2\beta^{-1} h)^{\frac32}\,
  D^3(\psi\circ\Theta^A)^\lb \big[\sigma^\lb\bm\eta^\lb,\sigma^\lb\bm\eta^\lb,\sigma^\lb\bm\eta^\lb\big]
  \\
  & \hspace{1.7cm} + 3h^2\sqrt{2\beta^{-1} h} \,D^3(\psi\circ\Theta^A)^\lb \big[\bm b^\lb,\bm
  b^\lb,\sigma^\lb\bm\eta^\lb\big]\Big),\\
  S_{0,n} :=&-\frac16\sum\limits_{\ell=0}^{n-1}\Big(
  h^3\,D^3(\psi\circ\Theta^A)^\lb \big[\bm b^\lb,\bm b^\lb,\bm b^\lb\big] + 6\beta^{-1} h^2 
  D^3(\psi\circ\Theta^A)^\lb
  \big[\bm b^\lb,\sigma^\lb\bm\eta^\lb,\sigma^\lb\bm\eta^\lb\big] \Big).
\end{align*}

In the following, we denote by $C>0$ a generic constant which is independent of $h,n$.
Using \eqref{conditions-on-eta}, a straightforward calculation shows that $\E
M_{i,n}=0$ for $i=0,\ldots,5$. Using compactness of $\Sigma$ and the
boundedness of $\bm\eta^\lb$ we have the estimates
\begin{equation}\label{eq:E-est}
|S_{1,n}| \leq C hT \|D \psi\|_{\infty},~ \|S_{2,n}| \leq C hT\|D^2
  \psi\|_{\infty},~ |S_{3,n}|\leq ChT\sum_{j=1}^4\|D^j\psi\|_{\infty},~
  |S_{0,n}|\leq C hT\sum_{j=1}^3 \|D^j\psi\|_{\infty}\,,     
\end{equation}
where the last two bounds hold almost surely. Combining these estimates along with
$|\psi^{(n)}-\psi^{(0)}|\leq 2\|\psi\|_\infty$ and applying Proposition~\ref{prop-Poisson}, we arrive at the first result. 

Concerning the estimate on mean square error, using the fact that $M_{i,n}$ are martingales, we obtain the estimates
\begin{align}
  \begin{split}
    &\E|M_{0,n}|^2\leq C h^2T \sum_{j=1}^3\|D^j\psi\|^2_\infty\,, \quad \E|M_{2,n}|^2 \leq Ch^2T
  \|D\psi\|^2_\infty\,, \quad \E|M_{3,n}|^2 \leq ChT \|D\psi\|^2_\infty\,,\\ 
    & \E|M_{4,n}|^2 \leq Ch^2T \|D^2\psi\|^2_\infty\,, \quad 
  \E|M_{5,n}|^2 \leq ChT \|D^2\psi\|^2_\infty\,. 
  \end{split}
  \label{bound-of-mi-without-m1}
\end{align}
For the term $M_{1,n}$, since $\bm\eta^\lb$ for different $\l$ are independent and (see \eqref{eqn-of-theta-map} in Proposition~\ref{prop-1st-varphi})
$$
D\psi^\lb\big[(D\Theta^A)^\lb[\sigma^\lb\bm\eta^\lb]\big] =\bm\eta^\lb \cdot
\big((P\sigma)^T D\psi\big)^\lb,$$ 
we have
  \begin{align}
    \frac{1}{T^2}\E |M_{1,n}|^2 & = \frac{2\beta^{-1}}{T}
    \frac{1}{n} \sum_{\l=0}^{n-1}
    \E \Big[\big((P\sigma)^T D\psi\big)^{(\ell)}
    \cdot \big((P\sigma)^T D\psi\big)^{(\ell)}\Big] \notag \\
    & = \frac{2\beta^{-1}}{T}
    \frac{1}{n} \sum_{\l=0}^{n-1}
    \E \Big[\big((B^{\mathrm{sym}} D\psi\big)^{(\ell)} \cdot D\psi^{(\ell)}\Big]\,, 
    \label{eqn-of-m1}
  \end{align}
  where we have used $P\sigma(P\sigma)^T=PaP^T=B^{\mathrm{sym}}$
  (Proposition~\ref{lemma-p-b}).
   Applying the estimate~\eqref{eq:MainThm-1} to the running average in \eqref{eqn-of-m1}, we obtain
  \begin{align}
    \frac{1}{T^2}\E |M_{1,n}|^2 \le&
    \frac{2\beta^{-1}}{T}\int_{\Sigma} (B^{\mathrm{sym}}D\psi)\cdot
    D\psi\, d\cmu + C\Big(\frac{h}{T}+\frac{1}{T^2}\Big) 
     = \frac{\chi_f^2}{T} + C\Big(\frac{h}{T}+\frac{1}{T^2}\Big) \,,
    \label{bound-of-square-m1}
  \end{align}
where $\chi_f^2$ is the asymptotic variance~\eqref{eqn-chi}. 
Taking square on both sides of \eqref{eq:avg-diff},
using Young's inequality, the estimates \eqref{eq:E-est}, \eqref{bound-of-mi-without-m1}, \eqref{bound-of-square-m1}, and applying Proposition~\ref{prop-Poisson}, we arrive at the second result.  

Now we prove the final pathwise result. Substituting the bounds in~\eqref{eq:E-est} into~\eqref{eq:avg-diff} we find
\begin{equation}\label{eq:diff-path}
|\hat f_n -\bar f| \leq \frac1T |\psi^{(n)}-\psi^{(0)}| +\frac1T \sum\limits_{i=0}^5 |M_{i,n}| + \frac1T \sum\limits_{i=0}^4 |S_{i,n}| \leq C\Bigl(h+\frac1T\Bigr) + \frac1T \sum\limits_{i=0}^5 |M_{i,n}|.
\end{equation}
Concerning the last term above, for any $r\geq 1$ we can derive the bounds (we omit the details and refer to the argument in~\cite[Theorem 5.3]{MattinglyStuartTretyakov10}),
\begin{align*}
&\frac{1}{T^{2r}}\E|M_{1,n}|^{2r}\leq \frac{C}{T^r}, \quad
  \frac{1}{T^{2r}}\E|M_{2,n}|^{2r}\leq \frac{Ch^{2r}}{T^r}, \quad
  \frac{1}{T^{2r}}\E|M_{3,n}|^{2r}\leq \frac{Ch^r}{T^r}\,, \\
&\frac{1}{T^{2r}}\E|M_{4,n}|^{2r}\leq \frac{Ch^{2r}}{T^r}, \quad
  \frac{1}{T^{2r}}\E|M_{5,n}|^{2r}\leq \frac{Ch^{r}}{T^r}, \quad \frac{1}{T^{2r}}\E|M_{0,n}|^{2r}\leq \frac{Ch^{2r}}{T^r},    
\end{align*}
which implies that
\begin{equation*}
\E\Bigl( \frac1T\sum_{i=0}^5|M_{i,n}| \Bigr)^{2r} \leq
  \frac{C}{T^{2r}}\sum_{i=0}^5\E|M_{i,n}|^{2r}\leq \frac{C}{T^r}\,.
\end{equation*}
This allows us (using Markov inequality and Borel-Cantelli lemma,
see~\cite[Theorem 5.3 and Section 4.2]{MattinglyStuartTretyakov10} for details) to conclude that for any $\eps\in(0,\frac12)$ there exists an almost surely bounded
random variable $\zeta=\zeta(\omega)$ such that 
\begin{equation*}
\frac1T\sum_{i=0}^5|M_{i,n}|\leq \frac{\zeta}{T^{\frac12-\eps}}\,.
\end{equation*}
The final result follows by substituting this result into~\eqref{eq:diff-path}.  
\end{proof}

Next, we prove Corollary~\ref{corollary-on-spectral-gap}.
\begin{proof}[Proof of Corollary~\ref{corollary-on-spectral-gap}]
 Recall that $\psi$ is the solution to the Poisson equation~\eqref{eq:Poisson-App}
  with $\E_{\cmu}[\psi] = 0$. Applying the Poincar{\'e} inequality~\eqref{poincare-inequality-mu1} followed by Cauchy-Schwarz inequality we have the standard estimates (see proofs of \cite[Corollary $2$]{Zhang20} and \cite[Lemma $9$]{legoll2017pathwise})
    \begin{equation*}
    \int_{\Sigma} \psi^2 d\cmu \le -\frac{1}{K} \int_{\Sigma}
    (\mathcal{L}\psi) \psi\,d\cmu 
    \le \frac{1}{K} \Big(\int_{\Sigma} (\mathcal{L}\psi)^2
    d\cmu\Big)^{\frac{1}{2}} \Big( \int_{\Sigma}
    \psi^2\,d\cmu\Big)^{\frac{1}{2}} 
    = \frac{1}{K} \Big(\int_{\Sigma} (f-\overline{f})^2 d\cmu\Big)^{\frac{1}{2}} \Big( \int_{\Sigma}
    \psi^2\,d\cmu\Big)^{\frac{1}{2}}\,, 
  \end{equation*}
  which implies
  \begin{align}    \label{estimate-followed-from-poincare}
    \begin{split}
        \Big(\int_{\Sigma} \psi^2\,d\cmu\Big)^{\frac{1}{2}} \le \frac{1}{K} 
\Big(\int_{\Sigma} (f-\overline{f}\,)^2\,d\cmu\Big)^{\frac{1}{2}} 
      \quad \mbox{and}\quad
 \chi_f^2 = -2\int_{\Sigma} (\mathcal{L}\psi) \psi\,d\cmu 
\le \frac{2}{K} \int_{\Sigma} (f-\overline{f}\,)^2\,d\cmu\,.
    \end{split}
  \end{align}
  The conclusion follows after we combine
  \eqref{estimate-followed-from-poincare} with the mean square error estimate
  of Theorem~\ref{thm:Err}.
\end{proof}

Finally, we prove Theorem~\ref{thm-numerical-theta}.
\begin{proof}[Proof of Theorem~\ref{thm-numerical-theta}]
We denote by $C>0$ a generic constant independent of $n,h$ and $\Delta t$.
  Recall that $\widetilde{x}^\lb$ and $\widetilde{x}^{(\l+\frac{1}{2})}$, $\l=0,1,\dots$, are the states given by the
  scheme~\eqref{scheme-non-reversible-numerical} with $\widetilde{x}^{(0)}\in \Sigma$.
  Define  
  \begin{equation}
    x^{(0)}=\widetilde{x}^{(0)} \ \ \text{ and } \ \ x^{(\l+1)}= \Theta^{A}(\widetilde{x}^{(\l+\frac{1}{2})}) \,, 
     \quad \l = 0,1,\dots\,.
     \label{yl}
  \end{equation}
  Since $\widetilde{x}^\lb \in \Sigma^{(\eps_{\mathrm{tol}})}$, there exists $h_0>0$, such that 
  $\widetilde{x}^{(\l+\frac{1}{2})} \in \Sigma^{(\delta)}$, for all $\l\ge 0$. Therefore using  Assumption~\ref{assump-3} we find  
  \begin{equation}
    |\widetilde{x}^{(\l+1)} - x^{(\l+1)}|=
     \left|\Tnum(\widetilde x^{(\l+\frac{1}{2})}) -
    \Theta^A(\widetilde x^{(\l+\frac{1}{2})})\right| \le C(\Delta t)^p \,, \quad \forall ~ \l \ge 0\,.
    \label{bound-of-el}
  \end{equation}
  Using \eqref{scheme-non-reversible-numerical} and \eqref{yl}, it is straightforward to verify that 
\begin{equation}\label{delta-and-x-numerical}
  \widetilde{x}^{(\ell+\frac{1}{2})}= x^\lb + \bm{\delta}^\lb  \ \ 
  \text{ and } \ \  x^{(\l+1)}  = \Theta^A(x^{(\l)} + \bm{\delta}^{\lb})\,,\quad \forall ~ \l=0,1,\dots, 
\end{equation}
with (compare with \eqref{delta-vector-by-b})
  \begin{equation}      \label{delta-vector-b-numerical}
    \bm{\delta}^\lb = \bm{b}^\lb h + \sqrt{2\beta^{-1} h}\,\sigma^\lb\bm{\eta}^\lb + \mathbf{r}^\lb\,,
\end{equation}
where  $\bm{b}^\lb=(b^\lb_1, b^\lb_2, \dots, b^\lb_d)^T$ is the vector with components given in \eqref{b-vector} (evaluted at $x^\lb$),
  $\sigma^{\lb}=\sigma(x^\lb)$, $\bm{\eta}^\lb$ is the random variable in \eqref{conditions-on-eta},
  and $\mathbf{r}^\lb=(r^\lb_1, r^\lb_2, \dots, r^\lb_d)^T$ with
  \begin{equation}  
    \begin{aligned}
      r_i^\lb &= 
      \widetilde{x}^\lb_i - x^\lb_i +  \bigg\{\sum_{j=1}^d \bigg[-\big(a_{ij}(\widetilde{x}^\lb)-A_{ij}\big)\frac{\partial U}{\partial
    x_j}(\widetilde{x}^\lb)+\frac1\beta \frac{\partial a_{ij}}{\partial
    x_j}(\widetilde{x}^\lb)\bigg] - b^\lb_i\bigg\}h \\
      & + \sqrt{2\beta^{-1} h}\sum_{j=1}^{d_1} \big(\sigma_{ij} (\widetilde{x}^\lb) -
      \sigma_{ij}(x^\lb)\big) \eta^{(l)}_j\,, 
    \end{aligned}
      \label{compoent-ril}
\end{equation} 
  for $1 \le i \le d$.

  Since the functions in \eqref{compoent-ril} are sufficiently regular (see
  Assumption~\ref{assump-1}), $\bm{b}^\lb$ is given by \eqref{b-vector} and 
$\bm{\eta}^\lb$ is almost surely bounded, the uniform bound
  \eqref{bound-of-el} implies that $|\mathbf{r}^\lb| \le C (\Delta t)^p$
  almost surely (whenever $h\le 1$). Since $x^\lb\in \Sigma$ for
  $\l=0,1,\dots$, the proof of Theorem~\ref{thm:Err} (using the relations \eqref{delta-and-x-numerical}--\eqref{delta-vector-b-numerical}) 
  carries over for $\frac1n
  \sum\limits_{\l=0}^{n-1}f(x^\lb)$, where the remainder terms arising due to $\mathbf{r}^\lb$ stay uniformly  bounded by $C(\Delta t)^p$.
  Concerning the running average $\widetilde{f}_n$ \eqref{running-average-approx}, using \eqref{bound-of-el} and the fact that the extension of $f$ 
  to $\Sigma^{(\eps_{\mathrm{tol}})}$ is $C^2$-smooth, we find
  \begin{equation}    \label{diff-running-averages-xy}
    \biggl|\widetilde{f}_n - \frac1n \sum\limits_{\l=0}^{n-1}f(x^\lb)\biggr| 
    = 
    \biggl|\frac1n \sum\limits_{\l=0}^{n-1} \big[f(\widetilde{x}^\lb) - f(x^{\lb})\big]\biggr|
    \le C(\Delta t)^p\,.
  \end{equation}
 The estimates in Theorem~\ref{thm-numerical-theta} then follow by using \eqref{diff-running-averages-xy}.
\end{proof}
\end{appendices}

{\small
\bibliographystyle{alphainitials}
\bibliography{NonRevSamp}

\newcommand{\etalchar}[1]{$^{#1}$}
\begin{thebibliography}{HHMS93}

\bibitem[BSU12]{pmlr-v22-brubaker12}
M.~Brubaker, M.~Salzmann, and R.~Urtasun.
\newblock A family of {MCMC} methods on implicitly defined manifolds.
\newblock In N.~D. Lawrence and M.~Girolami, editors, {\em Proceedings of the
  Fifteenth International Conference on Artificial Intelligence and
  Statistics}, volume~22 of {\em Proceedings of Machine Learning Research},
  pages 161--172. PMLR, 2012.

\bibitem[CKVE05]{blue-moon}
G.~Ciccotti, R.~Kapral, and E.~Vanden-Eijnden.
\newblock Blue moon sampling, vectorial reaction coordinates, and unbiased
  constrained dynamics.
\newblock {\em ChemPhysChem}, 6(9):1809--1814, 2005.

\bibitem[CLVE08]{projection_diffusion}
G.~Ciccotti, T.~Leli{\`e}vre, and E.~Vanden-Eijnden.
\newblock Projection of diffusions on submanifolds: Application to mean force
  computation.
\newblock {\em Commun. Pure Appl. Math.}, 61(3):371--408, 2008.

\bibitem[DLP16]{DuncanLelievrePavliotis16}
A.~B. Duncan, T.~Leli{\`e}vre, and G.~Pavliotis.
\newblock Variance reduction using nonreversible {L}angevin samplers.
\newblock {\em J. Stat. Phys.}, 163(3):457--491, 2016.

\bibitem[DLP{\etalchar{+}}18]{DLPSS18}
M.~H. Duong, A.~Lamacz, M.~A. Peletier, A.~Schlichting, and U.~Sharma.
\newblock {Quantification of coarse-graining error in Langevin and overdamped
  Langevin dynamics}.
\newblock {\em Nonlinearity}, 31(10):4517, 2018.

\bibitem[DPZ17]{DuncanPavliotisZygalakis17}
A.~B. Duncan, G.~A. Pavliotis, and K.~Zygalakis.
\newblock Nonreversible langevin samplers: Splitting schemes, analysis and
  implementation.
\newblock {\em arXiv preprint arXiv:1701.04247}, 2017.

\bibitem[FKVE10]{fatkullin2010}
I.~Fatkullin, G.~Kovacic, and E.~Vanden-Eijnden.
\newblock Reduced dynamics of stochastically perturbed gradient flows.
\newblock {\em Commun. Math. Sci.}, 8(2):439--461, 2010.

\bibitem[FL09]{FaouL09}
E.~Faou and T.~Leli{\`{e}}vre.
\newblock Conservative stochastic differential equations: Mathematical and
  numerical analysis.
\newblock {\em Math. Comput.}, 78(268):2047--2074, 2009.

\bibitem[GM96]{glynn1996}
P.~W. Glynn and S.~P. Meyn.
\newblock {A Liapounov bound for solutions of the Poisson equation}.
\newblock {\em Ann. Probab.}, 24(2):916--931, 1996.

\bibitem[HHMS93]{HwangHwangSheu93}
C.-R. Hwang, S.-Y. Hwang-Ma, and S.-J. Sheu.
\newblock Accelerating {G}aussian diffusions.
\newblock {\em Ann. Appl. Probab.}, 3(3):897--913, 1993.

\bibitem[HNS20]{HartmannNeureitherSharma20}
C.~Hartmann, L.~Neureither, and U.~Sharma.
\newblock Coarse graining of nonreversible stochastic differential equations:
  Quantitative results and connections to averaging.
\newblock {\em SIAM J. Math. Anal.}, 52(3):2689--2733, 2020.

\bibitem[HSZ19]{non-equilibrium-2018}
C.~Hartmann, C.~Sch{\"u}tte, and W.~Zhang.
\newblock Jarzynski equality, fluctuation theorems, and variance reduction:
  Mathematical analysis and numerical algorithms.
\newblock {\em J. Stat. Phys.}, 175(6):1214--1261, 2019.

\bibitem[LL10]{legoll2010effective}
F.~Legoll and T.~Leli{\`e}vre.
\newblock Effective dynamics using conditional expectations.
\newblock {\em Nonlinearity}, 23(9):2131, 2010.

\bibitem[LLO17]{legoll2017pathwise}
F.~Legoll, T.~Leli{\`e}vre, and S.~Olla.
\newblock Pathwise estimates for an effective dynamics.
\newblock {\em Stoch. Proc. Appl.}, 127(9):2841--2863, 2017.

\bibitem[LLS19]{LegollLelievreSharma18}
F.~Legoll, T.~Leli{\`e}vre, and U.~Sharma.
\newblock Effective dynamics for non-reversible stochastic differential
  equations: a quantitative study.
\newblock {\em Nonlinearity}, 32(12):4779, 2019.

\bibitem[LM15]{Leimkuhler-Matthews-MD-book}
B.~Leimkuhler and C.~Matthews.
\newblock {\em Molecular Dynamics: With Deterministic and Stochastic Numerical
  Methods}.
\newblock Interdisciplinary Applied Mathematics. Springer, 2015.

\bibitem[LNP13]{lelievre2013optimal}
T.~Leli{\`e}vre, F.~Nier, and G.~A. Pavliotis.
\newblock Optimal non-reversible linear drift for the convergence to
  equilibrium of a diffusion.
\newblock {\em J. Stat. Phys.}, 152(2):237--274, 2013.

\bibitem[LPVS20]{leimkuhler-constraint-regularization-nn}
B.~Leimkuhler, T.~Pouchon, T.~Vlaar, and A.~Storkey.
\newblock Constraint-based regularization of neural networks.
\newblock {\em arXiv e-prints}, arXiv:2006.10114, 2020.

\bibitem[LRS10]{LelievreRoussetStoltz10}
T.~Leli\`{e}vre, M.~Rousset, and G.~Stoltz.
\newblock {\em {Free Energy Computations: A Mathematical Perspective}}.
\newblock Imperial College Press, 2010.

\bibitem[LRS12]{Tony-constrained-langevin2012}
T.~Leli\`{e}vre, M.~Rousset, and G.~Stoltz.
\newblock Langevin dynamics with constraints and computation of free eneregy
  differences.
\newblock {\em Math Comput.}, 81(280):2071 -- 2125, 2012.

\bibitem[LRS19]{hmc-submanifold-tony}
T.~Leli{\`e}vre, M.~Rousset, and G.~Stoltz.
\newblock {Hybrid Monte Carlo methods for sampling probability measures on
  submanifolds}.
\newblock {\em Numer. Math.}, 143(2):379--421, 2019.

\bibitem[LS18]{LuSpiliopoulos18}
J.~Lu and K.~Spiliopoulos.
\newblock Analysis of multiscale integrators for multiple attractors and
  irreversible langevin samplers.
\newblock {\em Multiscale Modeling \& Simulation}, 16(4):1859--1883, 2018.

\bibitem[LSZ20]{multiple-projection-mcmc-submanifolds}
T.~Leli{\`e}vre, G.~Stoltz, and W.~Zhang.
\newblock Multiple projection {MCMC} algorithms on submanifolds.
\newblock {\em arXiv e-prints}, arXiv:2003.09402, 2020.

\bibitem[LV20]{rk-schemes-submanifolds}
A.~Laurent and G.~Vilmart.
\newblock Order conditions for sampling the invariant measure of ergodic
  stochastic differential equations on manifolds.
\newblock {\em arXiv e-prints}, arXiv:2006.09743, 2020.

\bibitem[LZ19]{LelievreZhang18}
T.~Leli{\`e}vre and W.~Zhang.
\newblock Pathwise estimates for effective dynamics: the case of nonlinear
  vectorial reaction coordinates.
\newblock {\em SIAM Multiscale Model. Sim.}, 17(3):1019--1051, 2019.

\bibitem[MMG19]{manifold-mcmc-Bayesian}
A.~B. Matthew M.~Graham, Alexandre H.~Thiery.
\newblock {Manifold Markov chain Monte Carlo methods for Bayesian inference in
  a wide class of diffusion models}.
\newblock {\em arXiv e-prints}, arXiv:1912.02982, 2019.

\bibitem[MST10]{MattinglyStuartTretyakov10}
J.~C. Mattingly, A.~M. Stuart, and M.~V. Tretyakov.
\newblock {Convergence of numerical time-averaging and stationary measures via
  Poisson equations}.
\newblock {\em SIAM J. Numer. Anal.}, 48(2):552--577, 2010.

\bibitem[PV03]{pardoux2003poisson}
E.~Pardoux and A.~Y. Veretennikov.
\newblock {On Poisson equation and diffusion approximation 2}.
\newblock {\em Ann. Probab.}, 31(3):1166--1192, 2003.

\bibitem[RBS16]{improve-reversible-sampler}
L.~Rey-Bellet and K.~Spiliopoulos.
\newblock Improving the convergence of reversible samplers.
\newblock {\em J. Stat. Phys.}, 164:472--494, 2016.

\bibitem[Zha20]{Zhang20}
W.~Zhang.
\newblock Ergodic {SDEs} on submanifolds and related numerical sampling
  schemes.
\newblock {\em ESAIM: Mathematical Modelling and Numerical Analysis},
  54(2):391--430, 2020.

\bibitem[ZHCG18]{goodman-submanifold}
E.~Zappa, M.~Holmes-Cerfon, and J.~Goodman.
\newblock {Monte Carlo on Manifolds: Sampling densities and integrating
  functions}.
\newblock {\em Commun. Pure Appl. Math.}, 71(12):2609--2647, 2018.

\bibitem[ZHS16]{ZhangHartmannSchutte16}
W.~Zhang, C.~Hartmann, and C.~Sch{\"u}tte.
\newblock Effective dynamics along given reaction coordinates, and reaction
  rate theory.
\newblock {\em Faraday Discussions}, 195:365--394, 2016.

\end{thebibliography}
}
\end{document}